%% file: ribbon.tex
\documentclass[noaddress, tikz, eulercal, defaultenums,
               sepfigandtablenums, smalltitle]{nmd/nmd-article}

\ifnmd
\fi

\DeclareFixedFont{\titlefont}{T1}{fvs}{b}{n}{18}

\title{Ribbon concordances and slice obstructions: experiments and examples}

\author{Nathan M. Dunfield}
\givenname{Nathan}
\surname{Dunfield}
\address{ Dept.~of Math., MC-382 \\
          University of Illinois \\
          1409 W. Green St. \\
          Urbana, IL 61801 \\ 
          USA
}
\email{nathan@dunfield.info}
\urladdr{http://dunfield.info}

\author{Sherry Gong}
\givenname{Sherry}
\surname{Gong}
\address{Texas A\&M University,  USA
}
\email{}
\urladdr{}

\arxivreference{}
\arxivpassword{}

\usepackage{placeins}
\usepackage[shortlabels]{enumitem}
\usepackage[group-separator={,}]{siunitx}
\usepackage{bm}
\newcommand{\lk}{\mathit{lk}}
\newcommand{\injects}{\hookrightarrow}
\newcommand{\surjects}{\twoheadrightarrow}
\renewcommand{\Hbar}{\kernoverline{5}{H}{0}\vphantom{H}}
\newcommand{\PSM}{\mathcal{PS}_{19}}
\newcommand{\PS}{$\PSM$}
\newcommand{\ZF}{$\mathcal{ZF}$}
\newcommand{\HKLC}{\mathcal{HKL}^c}
\newcommand{\Kh}{\mathit{Kh}}
\newcommand{\odd}{\mathit{odd}}

\newcommand{\svec}{\bm{s}}
\newcommand{\svectil}{\tilde{\bm{s}}\vphantom{\bm{s}}}
\newcommand{\Sq}{\mathrm{Sq}}
\newcommand{\med}{\mathrm{med}}
\newcommand{\HFK}{\widehat{\mathit{HFK}}}
\renewcommand{\Re}{\mathrm{Re}}
\renewcommand{\Im}{\mathrm{Im}}
\newcommand{\ConcordGraph}{$\Gamma_{RC}$}
\newcommand{\PartialConcordGraph}{$\Gamma_{\PSM}$}

\newcommand{\approxribboncount}{$\approx 1.6$ million} 
\newcommand{\approxnonslicecount}{$\approx 350.5$ million}

\newcommand{\NumNineteenKnots}{\num{352152252}}
\newcommand{\NumPlausiblySlice}{\num{3869541}}
\newcommand{\NumRibbonKnots}{\num{1633786}}                 
\newcommand{\approxremainingknots}{$\approx \num{11400}$}   

\newcommand{\NumTopSliceKnots}{\num{1656198}}               
\newcommand{\NumNotTopSlice}{\num{2211914}}                 
\newcommand{\NumTopSliceUnknowns}{\num{1429}}               
\newcommand{\approxremainingknotstop}{$\approx \num{1400}$} 

\newcommand{\NumObsByAnyHKLTest}{\num{2211761}}             
\newcommand{\NumObsByBasicHKLTest}{\num{2135384}}           
\newcommand{\NumObsByFancyHKLTest}{\num{72544}}             
\newcommand{\NumObsByDirectHKLTest}{\num{3833}}             
\newcommand{\NumObsByHLKWhereQIsTwo}{\num{67025}}           

\newcommand{\NumSmoothSliceUnknowns}{\num{11383}}           
\newcommand{\NumNotSmoothSlice}{\num{2224372}}              

\newcommand{\NumTopSliceButNotSmoothly}{\num{12162}}        
\newcommand{\NumTopSliceNotOrNotKnownSmoothly}{\num{22412}} 
\newcommand{\NumNotHomotopyRibbon}{24}                      

\newcommand{\NumObsBySmoothInvariants}{\num{333912}}        
\newcommand{\NumObsBySmoothInvButNotHKL}{\num{10631}}       
\newcommand{\NumSLThreeObsComputed}{\num{15355}}            

\newcommand{\NumObsBySmoothInvsBeyondFullTopObs}{\num{10618}}  
\newcommand{\NumNotTopSliceOutOfFullNineteen}{\num{350494625}} 

\newcommand{\NumVertsInConcordGraph}{\num{1677696}}
\newcommand{\NumEdgesInConcordGraph}{\num{1835130}}
\newcommand{\NumComponentsConcordGraph}{\num{524}}
\newcommand{\NumRibbonConcordToUnknot}{\num{1420528}}

\newcommand{\NumConcordNonAltToAlt}{\num{1484676}}

\newcommand{\NumKnotsNeedingOneBand}{\num{1249604}}       
\newcommand{\NumKnotsNeedingTwoBands}{\num{381869}}       
\newcommand{\NumKnotsNeedingThreeBands}{\num{2238}}       
\newcommand{\NumKnotsNeedingFourBands}{\num{75}}          
\newcommand{\RibbonKnotsWithXTorsionOrderTwo}{\num{582}}  

\definecolor{sherrycomment}{rgb}{0.067, 0.412, 0.067}

\graphicspath{{figures/}{plots/images/}}

\begin{document}

\begin{dedication}
  Dedicated to Cameron McA.~Gordon for his impact on topology and topologists.
\end{dedication}

\begin{abstract}
  There are 352.2 million prime knots in the 3-sphere with at most 19
  crossings.  We study which of these knots are slice, in both the
  smooth and topological categories.  While no algorithm is known for
  deciding whether a given knot is slice in either setting, we are
  able to determine it smoothly for all but about 11,400 knots
  (0.003\% or 1 in 30,000) and topologically for all but about 1,400
  knots (0.0004\% or about 1 in 250,000).  In particular, we show that
  some 1.6~million of these knots (0.46\%) are smoothly slice (in fact
  ribbon) and that 350.5~million are not even topologically slice
  (99.54\%).  We use a wide range of tools and techniques, and
  introduce several new or refined methods for probing these
  properties.  Along the way, we produce 500,000 pairs of 0-friends,
  that is, pairs of distinct knots with the same 0-surgery.  We
  discuss how our data is consistent with several important
  conjectures and suggests new ones, and highlight the simplest knots
  where sliceness remains unknown.
\end{abstract}
\maketitle

{\small
\tableofcontents
}

\input{intro}

\input{band_search}

\input{HKL}
\input{smooth}
\input{friends}

\input{nonslice_scraps}
\input{concordances}

{\RaggedRight
  \small
\bibliographystyle{nmd/math} 
\bibliography{\jobname}
}
\end{document}

%% file: intro.tex
\section{Introduction}

\enlargethispage{0.5cm}

Recall a \emph{knot} is the image $K$ of a smooth embedding
$S^1 \hookrightarrow S^3$.  A knot $K$ is \emph{smoothly slice} when
there is a smoothly embedded disk $D^2$ in the 4-ball $D^4$ whose
boundary is $K$.  Similarly, a knot $K$ is \emph{topologically
  slice} when it is the boundary of a locally-flat embedded disk in
$D^4$. Any smoothly slice knot is topologically slice, but the
converse need not be the case.  There are 352.2~million prime knots in
the 3-sphere with at most 19 crossings \cite{Burton2020}.
We study which of these knots are slice, in both the smooth and
topological categories.  While no algorithm is known
for deciding whether a given knot is slice in either setting, we are
able to determine it smoothly for all but \approxremainingknots\ knots
(0.003\% or about 1 in 30,000) and topologically for all but
\approxremainingknotstop\ knots (0.0004\% or about 1 in 250,000), using a
wide variety of tools and techniques.  In particular, we show that
some \approxribboncount\ of these knots (0.46\%) are smoothly slice
and that \approxnonslicecount\ are not even topologically slice
(99.54\%). More precisely:

\begin{theorem}
  \label{thm: main}
  Let $\cK_{19}$ denote the nontrivial prime knots in the 3-sphere
  with at most 19~crossings, up to homeomorphism of $S^3$, so that
  $\# \cK_{19} = \NumNineteenKnots$.  Then:
  \begin{enumerate}
  \item For $S = \NumRibbonKnots \ (0.46\%)$, the number of smoothly
    slice knots in $\cK_{19}$ is in $[S, S + \NumSmoothSliceUnknowns]$.
  \item For
    $T = S + \NumTopSliceNotOrNotKnownSmoothly = \NumTopSliceKnots \ (0.47\%)$, the
    number of topologically slice knots in $\cK_{19}$ is in
    $[T, T + \NumTopSliceUnknowns]$.  There are at least
    \NumTopSliceButNotSmoothly\ knots in
    $\cK_{19}$ which are topologically slice but not smoothly slice.
  \end{enumerate}
  Consequently, for
  $R = \NumNotTopSliceOutOfFullNineteen \ (99.53\%)$, the number of
  knots in $\cK_{19}$ that are not topologically slice is in
  $[R, R + \NumTopSliceUnknowns]$.
\end{theorem}

\subsection{Ribbon knots}

A \emph{ribbon disk} for a knot $K$ is a smooth slice disk $D$ with
image in $D^4 \setminus \{0\} \cong S^3 \times (0, 1]$ so that the
$(0, 1]$-coordinate is a Morse function on $D$ with no local maxes in
the interior of $D$.  A knot is \emph{ribbon} when it has a ribbon disk.
(Equivalently, a knot is ribbon when it is the boundary of a
restricted kind of immersed disk $D \looparrowright S^3$ as in
e.g.~\cite{Fox1961} or \cite[Definition 13.1.9]{Kawauchi1996}.) As every
ribbon knot is smoothly slice, one is led to ask:
\begin{question}[\cite{Fox1961}]
  Are all smoothly slice knots ribbon?
\end{question}
The slice-ribbon conjecture is that the answer to this question is
yes. Here, all the smoothly slice knots identified in
Theorem~\ref{thm: main} are indeed ribbon.  

\subsection{Prior work}
\label{sec: prior}
For knots with at most 12 crossings, the complete
picture was finished only in 2020:

\begin{theorem}[\cite{LewarkMcCoy2019, Piccirillo2020}]
  \label{thm: 12 known}
  Of the \num{2977} nontrivial prime knots with at most 12 crossings,
  exactly 158 are smoothly slice and 159 are topologically slice.  All
  of the smoothly slice knots are moreover ribbon.
\end{theorem}
Several groups have worked to identify ribbon and/or slice knots with
more crossings, including \cite[Appendix~A]{Stoimenow2002} and
\cite{GukovEtAl2023}. The largest general list we are aware of is
\cite{GukovEtAl2023}, which identifies \num{1705} ribbon knots from
among the \num{59937} nontrivial prime knots with at most
14 crossings.  We independently found the same 1705 knots as part of
Theorem~\ref{thm: main}; see Section~\ref{sec: ribbon comp} for more.

For alternating knots, Owens and Swenton studied these questions at a
vastly larger scale than Theorem~\ref{thm: 12 known} or
\cite{GukovEtAl2023} using techniques specific to the alternating
case:

\begin{theorem}[\cite{OwensSwenton2023}]
  \label{thm: OS alt}
  Of the 1.2 billion nontrivial prime alternating knots with at most 21
  crossings, at least $R = \num{664633}$ (0.05\%) are ribbon 
  and at most $R + \num{3276}$ are smoothly slice.
\end{theorem}
For alternating knots with at most 19 crossings, Owens and Swenton
had only 86 knots whose smooth slice status was unknown.  We resolved
all of these in the course of proving Theorem~\ref{thm: main} and
moreover settled the question in the topological category:

\begin{theorem}
  \label{thm: alt} 
  Of the 51.3 million nontrivial prime alternating knots with at
  most 19 crossings, exactly \num{82043} (0.16\%) are ribbon and the
  rest are not topologically slice.
\end{theorem}

\begin{figure}
  \centering
  \begin{tikzpicture}[font=\scriptsize]
    \tikzset{axis label/.style={font=\footnotesize}}
    \newcommand{\nmdsubfigurewidth}{6.2cm}
    \begin{scope}
      \input plots/knot_count.tex
    \end{scope}
    \begin{scope}[shift={(7.2, 0)}]
      \input plots/knot_proportion.tex
    \end{scope}
  \end{tikzpicture}

  \caption{ The number of prime plausibly slice and prime ribbon knots, as
    compared to all prime knots. For 20 crossings, we expect about 15
    million prime plausibly slice knots and 5.5 million prime ribbon knots.
    (There are 1.8 billion prime knots with 20 crossings
    \cite{Morwen23} and the plot at right suggests 0.8\% are plausibly
    slice and 0.3\% are ribbon.)}
  \label{fig: num PS}
\end{figure}
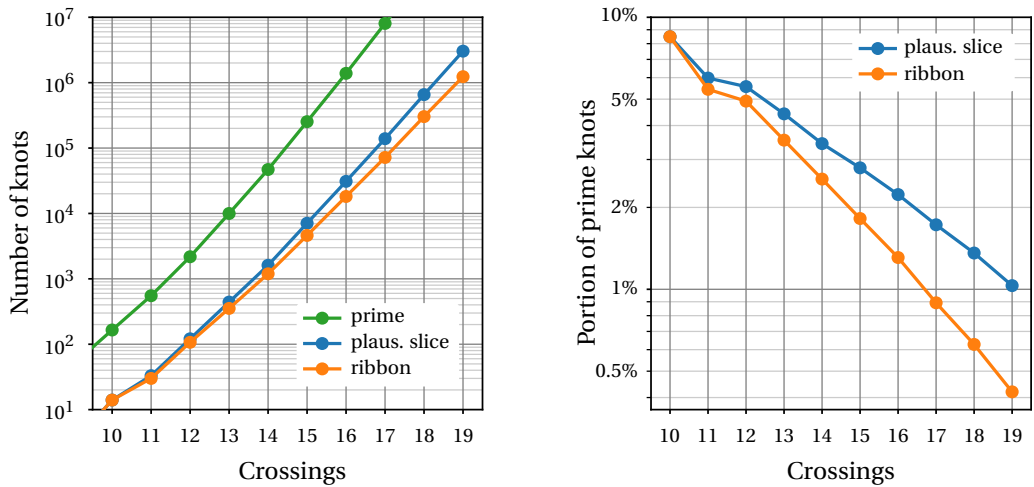

\begin{figure}
  \centering
  \begin{tikzpicture}[font=\scriptsize]
  \tikzset{axis label/.style={font=\footnotesize}}
  \tikzset{annotate plot/.style={font=\scriptsize}}
  \newcommand{\nmdsubfigurewidth}{6.3cm}
  \begin{scope}
    \input plots/volume_hist.tex
  \end{scope}
  \begin{scope}[shift={(7.0, 0)}] 
    \input plots/sqrt_det_hist.tex
  \end{scope}
  \begin{scope}[shift={(0, -5.0)}] 
    \begin{scope}
      \input plots/seifert_genus_hist.tex
    \end{scope}
    \begin{scope}[shift={(7.0, 0)}]
      \input plots/log_10_rk_HFK_hist.tex
    \end{scope}
  \end{scope}
\end{tikzpicture}

\caption{Some basic statistics about the 3.87 million plausibly slice
  knots.  Here $\mu$ is the mean, $\med$ is the median, and $\sigma$
  is the standard deviation. The ``spikes'' in the plot of
  $\sqrt{\det(K)}$ are when that number is divisible by $3$.  The plot
  of the knot Floer homology $\HFK$ is indistinguishable from the one
  of reduced Khovanov homology over $\F_2$ or $\Q$.}
\label{fig: PS basic}
\end{figure}
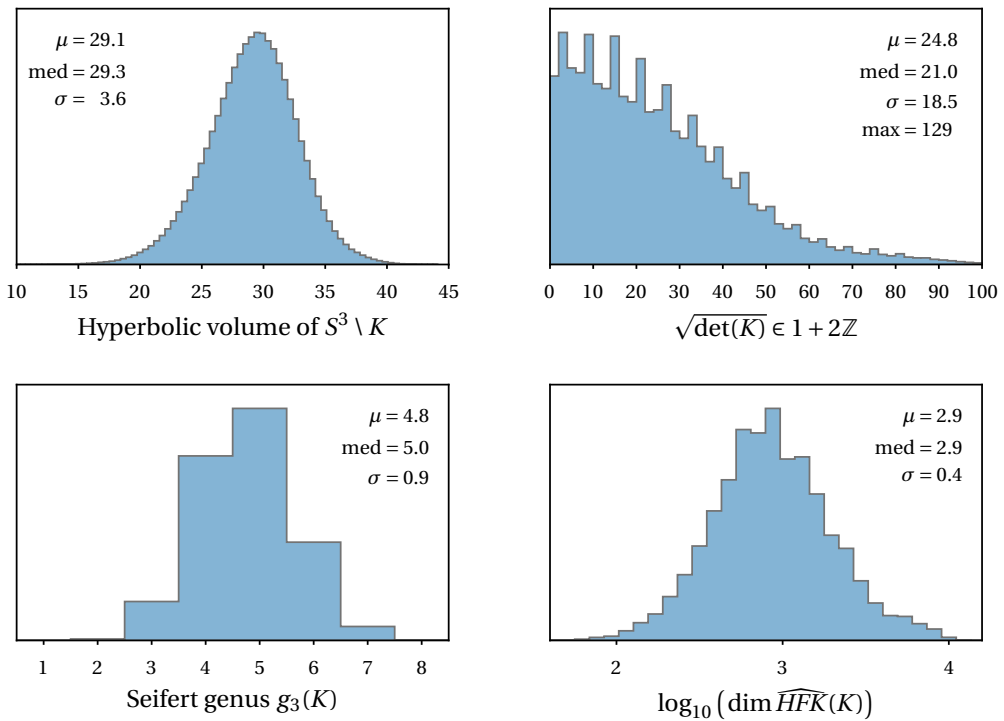

\begin{figure}
  \centering
  \begin{tikzpicture}[font=\footnotesize]
  \tikzset{axis label/.style={font=\small}}
  \tikzset{annotate plot/.style={font=\scriptsize}}
  \begin{scope}
    \input plots/vol_vs_log_HFK.tex
  \end{scope}
\end{tikzpicture}
\caption{Comparing hyperbolic volume and the knot Floer homology
  $\HFK$ for knots in \PS.  Here, the alternating knots are the higher
  blob (peach-purple) and the nonalternating ones are the lower blob
  (green-blue); only 5.3\% of the knots in \PS\ are alternating, so
  this plot over-emphasizes them to show the relative behavior.  This
  plot of $\HFK$ is indistinguishable from the one for reduced
  Khovanov homology over $\F_2$ or $\Q$, so this matches the
  observation from \cite[\S 8]{Khovanov2003}.}
\label{fig: vol vs HFK}
\end{figure}

\subsection{Plausibly slice knots}

We now outline the methods behind Theorem~\ref{thm: main}, starting
with the basic topological obstructions from the 1960s. We call a knot
\emph{plausibly slice} when its Murasugi signature $\sigma(K) = 0$ and
its Alexander polynomial $\Delta_K(t)$ satisfies the Fox-Milnor
criteria, that is, there is an $f(t) \in \Z[t^{\pm 1}]$ with
$\Delta_K(t) = f(t) f(t^{-1})$ up to a unit in $\Z[t^{\pm 1}]$.  Any
topologically slice knot is plausibly slice, see e.g.~Theorems 8.18
and 8.19 of \cite{Lickorish1997} or Corollary 12.3.13 of
\cite{Kawauchi1996}.  There are \NumPlausiblySlice\ nontrivial
plausibly slice prime knots with at most 19 crossings, which is 1.10\%
of the prime knots in that range; we use \PS\ to denote these knots,
and these are what we will study for the rest of this paper.
(Algebraically slice knots are plausibly slice, and we suspect that
nearly all knots in \PS\ are actually algebraically slice.  However,
checking this is quite involved \cite{Livingston2010}, and we are
unaware of any software that handles the general case.  Ad hoc methods
applied to a random sample of 10,000 knots from \PS\ not known to be
topologically slice found that at least 99.9\% are algebraically
slice.)

Statistics of some basic properties of the knots in \PS\ are shown in
Figures~\ref{fig: num PS}, \ref{fig: PS basic}, and \ref{fig: vol
  vs HFK}.  Theorem~\ref{thm: main} follows immediately from:

\begin{theorem}
  \label{thm: real main}
  Of the knots in \PS, at least \NumRibbonKnots\ are ribbon and at
  least \NumNotSmoothSlice\ are not smoothly slice.  Also, at least
  \NumTopSliceKnots\ are topologically slice and \NumNotTopSlice\
   are not topologically slice.  This leaves
  \NumSmoothSliceUnknowns\ and \NumTopSliceUnknowns\ unknowns in the
  smooth and topological settings respectively. There are at least
  \NumTopSliceButNotSmoothly\ knots in \PS\ that are topologically but not
  smoothly slice.
\end{theorem}
Two methods handle all but 0.6\% of the knots in
Theorem~\ref{thm: real main}:
\begin{enumerate}
\item
  \label{item: ribbon}
  All \NumRibbonKnots\ ribbon knots were identified using the process
  detailed in Section~\ref{sec: bands}.  Given a knot $K$, we did a
  combinatorial search for a sequence of band moves that turns $K$
  into the unlink and so specifies a ribbon disk. As
  needed, we looked for such moves in many different diagrams for $K$;
  these diagrams were generated by a more global method than just doing
  random Reidemeister moves.  We also created a table of
  $\approx \num{12100}$ two-component ribbon links which sometimes
  allowed us to certify a knot is ribbon after doing a single band
  move.

\item
  \label{item: HKL}
  The main slicing obstruction used was in the topological
  category, namely the technique of Herald, Kirk, and Livingston
  \cite{HeraldKirkLivingston2010}.  This method uses twisted Alexander
  polynomials to provide a computationally-efficient form of the
  Casson-Gordon invariants \cite{CassonGordonOrsay}. We introduce
  several refinements of this method in Sections~\ref{sec: HKL
    refinement}, \ref{sec: poly is a norm}, \ref{sec: HKL with q=2},
  and \ref{sec: q^e}.  
  Collectively, such 
  \emph{HKL tests} obstruct \NumObsByAnyHKLTest\ knots in \PS\ from
  being topologically slice (Theorem~\ref{thm: HKL obs summary}).
\end{enumerate}
To handle the remaining \num{23994} knots, we used the following methods:
\begin{enumerate}[resume]
\item
  \label{item: smooth}
  We computed many different smooth slice obstructions from Khovanov
  homology, knot Floer homology, and gauge theory; see
  Section~\ref{sec: smooth} for details, but this includes the
  $s$-invariant in Khovanov homology and several of its
  generalizations.  Collectively, these show that
  \NumObsBySmoothInvariants\ knots (8.6\% of \PS ) are not smoothly
  slice (Theorem~\ref{thm: smooth summary}).  However, only
  \NumObsBySmoothInvButNotHKL\ of these knots (0.3\% of \PS) are not
  already covered by the HKL tests in (\ref{item: HKL}).

\item
  \label{item: scraps}
  Some 36 knots need special attention including the Conway knot and
  the 2-1 cable on the figure 8 knot, which are not smoothly slice by
  \cite{Piccirillo2020} and \cite{DaiEtAl2024} respectively.
  Theorem~\ref{thm: slice via friends} shows that 25 knots are not
  smoothly slice using the 0-friends method discussed below in
  Section~\ref{sec: zero intro}. Theorems~\ref{thm: cable 8}
  and~\ref{thm: one knot} each show a single knot is not smoothly
  slice. Finally, Theorem~\ref{thm: alt scraps} shows nine alternating
  knots are not topologically slice.

\item \label{item: top slice only} The additional
  \NumTopSliceNotOrNotKnownSmoothly\ topologically slice knots were
  primarily identified by having $\Delta_K = 1$, see \cite[\S
  11.7]{FreedmanQuinn1990}, or by the method of Section~\ref{sec:
    bands} producing a ribbon concordance to a knot with
  $\Delta_{K} = 1$.  One knot, $18nh_{00000601}$, was shown to be
  topologically slice by the 0-friend method as part of
  Theorem~\ref{thm: mystery 2}.

\item The final $\approx \num{2000}$ knots we dealt with in
  Theorem~\ref{thm: real main} were handled in Section~\ref{sec: ribbon
    graph} by using the ribbon concordances between knots in \PS\
  produced in Section~\ref{sec: bands}.  For example, there are
  \num{1672} other knots in \PS\ that are ribbon concordant to the Conway
  knot; since the Conway knot is not smoothly slice, the same is true
  for those \num{1672} knots.
\end{enumerate}

\begin{remark}
  It is worth noting that the topological HKL tests of (\ref{item:
    HKL}) were about 6.6~times as effective as all the smooth
  obstructions in (\ref{item: smooth}).  Thus, in practice it seems
  that one should start with HKL tests rather the smooth
  invariants even if one only cares about slicing in the smooth
  category.  

  Also, the 10 knots in Theorems~\ref{thm: cable 8} and \ref{thm: alt
    scraps} as part of (\ref{item: scraps}) are the only place where
  we rely directly on the vast literature of slice properties of
  particular knots. In all other cases, our computations are at least
  nominally independent of any prior work.
\end{remark}

\begin{figure}
  \centering
  \begin{tikzpicture}[line cap=round, line join=round,
                      line width=1.01, font=\small]
    \begin{scope}[shift={(0, 2.5)}, rotate=-45, scale=0.44]
      \input plots/18nh_00000601_smooth_gaps_20_0.3_width_800.tikz
    \end{scope}
  
    \begin{scope}[ shift={(7, 0.2)}, rotate=0, scale=0.6]
      \input plots/0-friend_smooth_gaps_20_0.3_width_800.tikz
    \end{scope}
    
    \node[below] at (3, 0) {$K = 18nh_{00000601}$ in \PS};
    \node[below] at (10.5, 0) {A 0-friend $K'$ of $K$}; 
  \end{tikzpicture}
  \caption{A pair of 0-friends. The knot $K$ has 18
    crossings and the simplest diagram for $K'$ known has 31 crossings.
    This pair will appear again in Section~\ref{sec: mystery}.}
  \label{fig: besties}
\end{figure}

\subsection{0-friends}
\label{sec: zero intro}

The 0-surgery on a knot $K$ in $S^3$ is the unique Dehn surgery where
$b_1 > 0$.  A pair of knots in $S^3$ are \emph{0-friends} when their
0-surgeries are homeomorphic.  A secondary focus of this paper,
contained in Section~\ref{sec: friends}, is to study 0-friends of 
knots in \PS. The motivation for this is twofold.
First, and most tantalisingly, if $K$ and $K'$ are 0-friends where $K$
is smoothly slice and $K'$ is not, then there is an exotic smooth structure on
$S^4$, disproving the smooth 4-dimensional Poincar\'e Conjecture; see
\cite[\S 1]{ManolescuPiccirillo2023}.  Second, even with that
conjecture unresolved, in favorable circumstances one can show that
certain 0-friends must have the same smooth slice status; for example, this
was used to show that the Conway knot is not smoothly slice
\cite{Piccirillo2020}.

There are only 101 pairs of knots in \PS\ that are 0-friends (see
Section~\ref{sec: friend finder}), so we used the method of
\cite[\S9.3]{DunfieldObeidinRudd2024} to generate more than
\num{500000} 0-friends of the knots in \PS.  Very roughly, starting
with a knot $K$ whose 0-surgery $Z_K = S^3_0(K)$ is hyperbolic, one
drills out a closed geodesic in $Z_K$ and tests whether the resulting
manifold is the exterior of a knot in $S^3$.  When
it is, one can recover a diagram for the 0-friend using
\cite{DunfieldObeidinRudd2024}.  We outline in Section~\ref{sec:
  friends} why hyperbolic geometry suggests this method produces most
small-to-medium 0-friends of knots in \PS; a sample pair of 0-friends found
is shown in Figure~\ref{fig: besties}.

Using this plethora of 0-friends and the methods of
\cite{Piccirillo2020} and \cite{ManolescuPiccirillo2023}, we show 25
knots in \PS\ are not smoothly slice in Theorem~\ref{thm: slice via
  friends}.  Following \cite{ManolescuPiccirillo2023}, this involves
producing an RBG link (see Definition~\ref{def: RBG link} below) for
each pair of 0-friends; we introduce a new computational method to
search for such links in Section~\ref{sec: drill RBG}, again using
\cite{DunfieldObeidinRudd2024}.

We also found four pairs where we could determine the smooth sliceness
of the 0-friend but not of the original knot in \PS, leading to:
\begin{theorem}
  \label{thm: mystery}
  If $18nh_{00000601}$ is not smoothly slice, or if any of
  $\{16n68278$, $17nh_{0010647}$, $18nh_{00098198}\}$ are smoothly
  slice, then there is an exotic smooth 4-sphere.
\end{theorem}
All four knots in Theorem~\ref{thm: mystery} are topologically slice,
and $18nh_{00000601}$ is even smoothly slice in a homotopy 4-ball; see
Theorem~\ref{thm: mystery 2}.

\begin{figure}
  \centering
  \begin{tikzpicture}[font=\footnotesize]
    \input plots/concord_intro.tex
  \end{tikzpicture}

  \caption{A visualization of ribbon concordances between 168 smoothly
    slice knots.
    Each knot is indicated by a dot, where alternating knots are red
    and other knots are blue.  Each edge represents a ribbon
    concordance $K_1 \geq K_0$ where $K_1$ always corresponds to the
    dot farther to the left; that is, the hyperbolic volume of
    $S^3 \setminus K$ respects the ribbon partial order.  The color of
    each edge is the average of its vertices. Note there are no
    $K_1 \geq K_0$ where $K_1$ is alternating and $K_0$ is not. See
    Section~\ref{sec: ribbon graph} and Figure~\ref{fig: piece of U}
    for more details.}
  \label{fig: piece of U intro}
\end{figure}
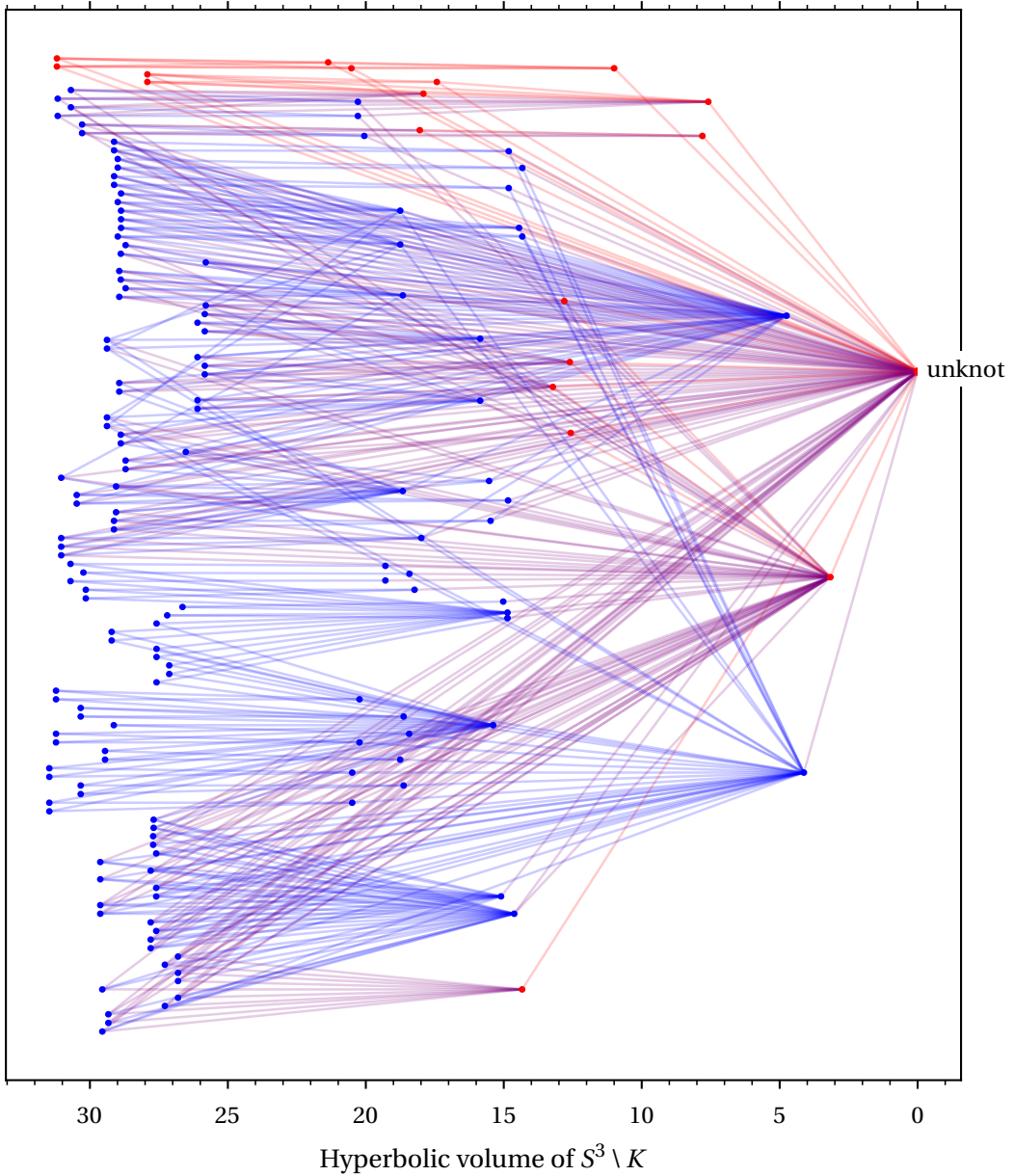

\subsection{Ribbon concordances}
\label{sec: concord intro}

Suppose $K_0$ and $K_1$ are knots in
$S^3$.  Recall from \cite{Gordon1981} that a \emph{ribbon concordance
  from $K_1$ to $K_0$} is a smooth annulus $A$ embedded in
$S^3 \times I$ with boundary components $K_0 \times \{0\}$ and
$K_1 \times \{1\}$ so that the $I$-coordinate is a Morse function on
$A$ with no local maxima.  In this case, we write $K_1 \geq K_0$.  In
particular, $K$ is ribbon if and only if $K \geq \mathrm{unknot}$.
Agol recently proved that $\geq$ gives a partial order on isotopy
classes of knots \cite{Agol2022}. Our search for ribbon disks produced
\NumEdgesInConcordGraph\ ribbon concordances $K_1 \geq K_0$; in
\NumRibbonConcordToUnknot\ cases, $K_0$ was the unknot.  In
Section~\ref{sec: ribbon graph}, we describe how these give evidence
for conjectures that some invariants, such as the hyperbolic volume,
are decreasing with respect to this partial order.  A sampling of the
ribbon concordances we found is shown in Figure~\ref{fig: piece of U
  intro}.

\subsection{Category alignment}

We now turn to two situations where there seems to be no difference
between topological and smooth sliceness.  First, recall Lisca showed
that a 2-bridge knot is ribbon if and only if it is smoothly slice
\cite{Lisca2007}.  Moreover, it is posited that a 2-bridge knot $K$ is
smoothly slice if and only if it is topologically slice
\cite[Conjecture~1]{FellerMcCoy2016}.  While there are 2-bridge knots
whose topological and smooth slice genera differ
\cite{FellerMcCoy2016}, our Theorem~\ref{thm: alt} above
still compels us to propose:

\begin{conjecture}
  \label{conj: alt}
  A prime alternating knot is smoothly slice if and only if it is
  topologically slice.
\end{conjecture}

Second, recall that a knot $K$ is \emph{fibered} when the exterior
$E_K = S^3 \setminus \nu(K)$ fibers over the circle; here, the fiber
surface is necessarily the unique minimum genus Seifert surface.
Using knot Floer homology computed via \cite{HFKCalculator}, we found
that 29.9\% of the knots in \PS\ are fibered.  Per Table~\ref{tab:
  fibered unknown} below, we resolved both topological and smooth
sliceness for all but 25 of the fibered knots in \PS.  Thus, we were
vastly more successful with fibered knots than nonfibered ones
(factors of $\approx 200$ and $\approx 25$ in the smooth and
topological settings respectively). This is despite not employing any
tests specific to the fibered case, though some exist at least in
theory \cite{CassonGordon1983, CassonLong1985}. Regardless, in all but
the 25 unknown cases, the topological and smooth sliceness matched for
fibered knots, suggesting:
\begin{conjecture}
  A fibered hyperbolic knot is smoothly slice if and only if it is
  topologically slice.
\end{conjecture}
Here, we restrict to hyperbolic knots because there is only a single
non-hyperbolic fibered knot in \PS, namely $17ns_{29}$. Note that a
fibered knot that is topologically but not smoothly slice would give
rise to a counterexample either to the smooth 4D Poincar\'e conjecture
or to the topological slice-ribbon conjecture discussed in the next
subsection \cite{Ruberman2015}.

\subsection{The slice-ribbon conjecture} To search for counterexamples to
the slice-ribbon conjecture, one needs methods for showing that a knot
$K$ is smoothly slice without exhibiting a ribbon disk.  Recall from
e.g.~\cite[Lemma 2.5]{Teichner2011} that a knot $K$ is smoothly slice
if and only if there is a ribbon knot $J$ so that $K \# J$ is
ribbon. As recounted in Section~\ref{sec: slice only}, we tried this
strategy for several small $J$.  Taking a hint from how
\cite{HeraldKirkLivingston2010} found a slice disk for $K12a990$, we
can sometimes choose $J$ so that we know in advance that $K \# J$ is
ribbon concordant to a knot that is simpler than $K$.  Using this
method, we discovered more than 500 smoothly slice knots where
repeated passes of our ribbon disk search had come up empty.  However,
we eventually found ribbon disks for all of them; see
Section~\ref{sec: slice only} for details.  Thus we have identified no
smoothly slice knots that are not known to be ribbon.

However, the knot $K = 18nh_{00000601}$ from Theorem~\ref{thm:
  mystery} is intriguing in the context of the slice-ribbon
conjecture.  Unless it and its 0-friend give a counterexample to the
smooth 4D Poincar\'e conjecture, the knot $K$ must be smoothly slice.
However, we tried very hard to find a ribbon disk for $K$ to no
avail.  This makes $K$ a plausible candidate for a knot that is
smoothly slice but not ribbon.

A version of the slice-ribbon conjecture in the topological category
is as follows.  Recall that a knot $K$ is \emph{topologically homotopy
  ribbon} when it bounds a locally flat disk $F \subset D^4$ so that
the inclusion $S^3 \setminus K \hookrightarrow D^4 \setminus F$ is
surjective on fundamental groups.  (Any ribbon knot is topologically
homotopy ribbon \cite[Lemma 3.1]{Gordon1981}.)  The topological
slice-ribbon conjecture is that every topologically slice knot is
topologically homotopy ribbon.  Of the topologically slice knots in
Theorem~\ref{thm: real main}, all but possibly $18nh_{00000601}$ are
in fact topologically homotopy ribbon.  This is because knots with
$\Delta_K = 1$ are topologically homotopy ribbon by \cite[Theorem
11.7B]{FreedmanQuinn1990} and this is the basis for our primary method
(\ref{item: top slice only}) above for showing that a knot is
topologically slice.

\subsection{Inscrutable knots}

We next highlight some of the simpler knots not covered by
Theorem~\ref{thm: main} so that they can serve as challenges and
benchmarks for future work in this area.  Table \ref{tab: smooth
  unknown} lists the knots of small crossing number whose smooth
sliceness remains unknown; Table~\ref{tab: top unknown} is the
topological analogue.  Table~\ref{tab: low vol unknown} focuses on
unknown knots with small hyperbolic volume and Table~\ref{tab: non hyp
  unknown} lists the non-hyperbolic knots; together, these two tables
includes all unknown knots in \PS\ with Seifert genus~1.  Finally,
Table~\ref{tab: fibered unknown} lists all unknown knots in \PS\ that
are fibered.

\begin{table}[b]
  \centering
  {\small
    \begin{tabular}{lllll}
      \toprule
      $K13n65$ & $K14n10011$ & $K15n17501$ & $K15n110439$ & $K15n132965$ \\
      $K13n3871$ & $K14n11063$ & $K15n21905$ & $K15n117232$ & $K15n137921^*$ \\
      $K13n3872$ & $K14n18909$ & $K15n27582$ & $K15n120055$ & $K15n138033$ \\
      $K13n3897$ & $K14n18911$ & $K15n33471^*$ & $K15n121598$ & $K15n138051$ \\
      $K13n3936$ & $K14n21673$ & $K15n40402$ & $K15n121834$ & $K15n140327$ \\
      $K13n4582$ & $K15n2810$ & $K15n49735$ & $K15n121916$ & $K15n140449$ \\
      $K14n3713^*$ & $K15n4646$ & $K15n53931^*$ & $K15n122603$ & $K15n144034$ \\
      $K14n4425$ & $K15n11287$ & $K15n94339^*$ & $K15n123414$ & $K15n146982$ \\
      $K14n4621^*$ & $K15n11568$ & $K15n95989$ & $K15n124496$ & $K15n154389$ \\
      $K14n5486$ & $K15n11570$ & $K15n95995$ & $K15n124640$ & $K15n155464$ \\
      $K14n9023$ & $K15n16056$ & $K15n103703$ & $K15n130504$ & $K15n157903^*$ \\
      \bottomrule
    \end{tabular}
  }
  
  \caption{Above are the 55 knots with at most 15 crossings where we
    could not determine whether they are smoothly slice. With the
    exception of the seven marked with a ``$*$'', these knots have
    $\Delta_K = 1$ and hence are topologically slice \cite[\S
    11.7]{FreedmanQuinn1990}.  We do not know whether any of the ones
    marked with ``$*$'' are topologically slice; they all have
    $\Delta_K = (t^2 - t + 1)^2$.}
  \label{tab: smooth unknown}
\end{table}

\begin{table}[b]
  \centering
  {\small
    \begin{tabular}{lllll}
      \toprule
      $K14n3713$ & $K15n115646^*$ & $16n254312$ & $16n658612^*$ & $16n889116$ \\
      $K14n4621$ & $K15n137921$ & $16n288832$ & $16n679074$ & $16n898244$ \\
      $K15n33471$ & $K15n157903$ & $16n312764^*$ & $16n716954$ & $16n911851$ \\
      $K15n53931$ & $16n86850$ & $16n357101$ & $16n740455$ & $16n914151$ \\
      $K15n73236^*$ & $16n129260$ & $16n374204$ & $16n800378^*$ & $16n977108$ \\
      $K15n77799^*$ & $16n178564$ & $16n429842$ & $16n829217$ & $16n982575$ \\
      $K15n94339$ & $16n180683$ & $16n478839^*$ & $16n843878$ &  \\
      $K15n99019^*$ & $16n218932^*$ & $16n481747$ & $16n889114$ &  \\
      \bottomrule
    \end{tabular}
  }
  \caption{Above are the 38 knots with at most 16 crossings where we could
    not determine whether they are topologically slice. The 9 known not to
    be smoothly slice are marked with a ``$*$''.  Some 30 of these
    knots have $\Delta_K = (t^2 - t + 1)^2$.}
  \label{tab: top unknown}
\end{table}

\begin{table}
  \centering
  {\small
    \begin{tabular}{lrc@{\hskip 0.4em}c@{\hskip 0.4em}clrc@{\hskip 0.4em}c@{\hskip 0.4em}c}
      \toprule
      knot & vol & $g_3(K)$ & sm & top &  knot & vol & $g_3(K)$ & sm & top  \\
      \midrule
      $19nh_{000000055}$ & 8.66 & 2 &  &  & $19nh_{000003713}$ & 14.04 & 6 &  & y \\
      $19nh_{000000156}$ & 9.96 & 1 &  &  & $19nh_{000003849}$ & 14.09 & 1 &  & y \\
      $17nh_{0000497}$ & 11.85 & 2 &  &  & $19nh_{000004248}$ & 14.22 & 2 &  &  \\
      $18nh_{00000597}$ & 11.92 & 3 &  & y & $19nh_{000004269}$ & 14.23 & 2 &  & y \\
      $18nh_{00000601}$ & 11.93 & 5 &  & n & $18nh_{00003640}$ & 14.36 & 3 &  & y \\
      $18nh_{00000707}$ & 12.20 & 1 & n &  & $16n429842$ & 14.45 & 2 &  &  \\
      $18nh_{00000752}$ & 12.30 & 2 &  &  & $18nh_{00004014}$ & 14.49 & 2 &  & y \\
      $18nh_{00000957}$ & 12.63 & 2 &  & & $18nh_{00004122}$ & 14.53 & 2 & n &  \\
      $K15n77799$ & 13.15 & 2 & n &  & $19nh_{000005728}$ & 14.60 & 2 & n &  \\
      $18nh_{00001632}$ & 13.31 & 1 & n &  & $19nh_{000006042}$ & 14.67 & 2 &  &  \\
      $18nh_{00001818}$ & 13.44 & 2 &  &  & $17nh_{0004389}$ & 14.73 & 3 &  & y \\
      $16n481747$ & 13.59 & 2 & &  & $19nh_{000007003}$ & 14.85 & 3 &  & y \\
      $18nh_{00002591}$ & 13.90 & 2 &  & y & $18nh_{00005355}$ & 14.87 & 3 &  & y \\
      $18nh_{00002656}$ & 13.93 & 2 &  & y & $17nh_{0005136}$ & 14.94 & 3 &  & y \\
      \bottomrule
  \end{tabular}
  }
  \caption{The hyperbolic knots in \PS\ of smallest volume whose slice
    status is not completely known.  Here, an ``n'' in the ``sm''
    column means the knot is not smoothly slice, and a ``y'' or ``n''
    in the ``top'' column indicates a known topologically slice status.  All
    the unknown knots with $g_3(K) = 1$ appear either in this table or
    Table~\ref{tab: non hyp unknown}.}
  \label{tab: low vol unknown}
\end{table}

\begin{table}
  \centering
  {\small
    \begin{tabular}{llccc}
      \toprule
      knot & structure & $g_3(K)$ & smooth & top \\
      \midrule
      $K15n115646$ & Trefoil[3/2] & 1 & n &  \\
      $16n800378$ & Trefoil[$-1/4$] & 1 & n &  \\
      $17ns_{29}$ & Fig8[1]        & 2 & n &  \\
      $18ns_{51}$ & Trefoil[A(1/2, 4/3)] & 2 & n & \\
      $18ns_{52}$ & Trefoil[A(1/2, 3/5)] & 2 & n &  \\
      $18ns_{86}$ & Fig8[$-1/2$] & 1 & & y \\
      $19ns_{006}$ & Trefoil[11/6] & 1 & n &  \\
      $19ns_{018}$ & Trefoil[15/4] & 1 & n &  \\
      $19ns_{101}$ & Trefoil[A($-1/2, 5/7$)] & 2 & n & \\
      $19ns_{157}$ & Trefoil[C($-1/2, 1/3, -1/2$)] & 2 & n &  \\
      $19ns_{193}$ & Trefoil[D($2, -3, -1/2$)] & 2 & n &  \\
      $19ns_{244}$ & Fig8[$-1/3$] & 2 & &  \\
      \bottomrule
    \end{tabular}
  }
  \caption{The 12 non-hyperbolic knots in \PS\ whose slice
    status is not completely known.  These are all satellite knots;
    see \cite[\S 5.9]{Burton2020} for how to read the structure
    column. Here an ``n'' in the ``smooth'' column means the knot is
    not smoothly slice, and a ``y'' in the ``top'' column means the
    knot is topologically slice.  There are only 21 non-hyperbolic
    knots in \PS, so we had a much lower success rate in this
    setting.  All the unknown knots with $g_3(K) = 1$ appear either in
    this table or Table~\ref{tab: low vol unknown}.}
  \label{tab: non hyp unknown}
     
\end{table}

\FloatBarrier

\begin{table}[h]
  \centering
  {\small
    \begin{tabular}{l@{\hskip 0.7em}ccc@{\hskip 2.0em}l@{\hskip 0.7em}ccc}
      \toprule
      knot & $g_3(K)$ &  smooth & top & knot & $g_3(K)$ &  smooth & top \\
      \midrule
      $17nh_{4549325}$ & 6 &  &  & $19nh_{000458099}$ & 7 &  &  \\
      $17ns_{29}$ & 2 & n &  & $19nh_{000616024}$  & 4 &  &  \\
      $18nh_{00000601}$ & 5 &  & n & $19nh_{000889975}$ & 4 & &  \\
      $18nh_{00068958}$ & 4 &  &  & $19nh_{000941086}$ & 4 &  &  \\
      $18nh_{00251375}$ & 4 &  &  & $19nh_{003772023}$ & 4 &  &  \\
      $18nh_{00446655}$ & 4 &  &  & $19nh_{009668884}$ & 4 &  &  \\
      $18nh_{00584034}$ & 4 &  &  & $19nh_{010449461}$ & 4 &  &  \\
      $18nh_{01814106}$ & 4 &  &  & $19nh_{019331276}$ & 4 &  &  \\
      $18nh_{01872227}$ & 4 &  &  & $19nh_{019761670}$ & 4 &  &  \\
      $18nh_{04892103}$ & 4 &  &  & $19nh_{029859281}$ & 5 &  &  \\
      $18nh_{10190577}$ & 4 &  &  & $19nh_{032455745}$ & 4 &  & \\
      $19nh_{000054314}$ & 4 &  &  & $19nh_{080376121}$ & 4 & &  \\
      $19nh_{000107587}$ & 4 &  &  &  \\
     \bottomrule
    \end{tabular}
  }
  \caption{The 25 \emph{fibered} knots in \PS\ whose slice status is
    not completely known.  Here, the knot
    $17ns_{29}$ is not smoothly slice by Theorem~\ref{thm: cable 8},
    and $18nh_{00000601}$ is topologically slice by Theorem~\ref{thm:
      mystery 2}.  The genus of the fibered Seifert surface is listed
    for each knot.}
  \label{tab: fibered unknown}
\end{table}

\subsection{Code and data}

Of necessity, Theorem~\ref{thm: real main} is the result of a
large-scale computer computation.  The main software tools we used are
SnapPy \cite{SnapPy}, KnotJob \cite{KnotJob}, the HFK Calculator
\cite{HFKCalculator}, and SageMath \cite{SageMath}. All data and
code needed to check our results are permanently archived at
\cite{CodeAndData}, and key parts have been incorporated into the
development version of SnapPy.  The latter will be part of SnapPy 3.3
but can be used now following the instructions in \cite{CodeAndData}.
We include in \cite{CodeAndData} a quick-access list of PD codes for
all 286 knots in \PS\ mentioned in this paper by name.

Producing Theorem~\ref{thm: real main} consumed more than 100
CPU-years of computational time over the course of five
calendar-years.  However, it would take a tiny fraction of that to
confirm our results as, for example, we saved a description of each
ribbon disk found (see Section~\ref{sec: cert}), and these can be
quickly checked as valid. Similarly, the data includes the $(m, p)$
parameters of each successful HKL test.

One method we used to try to ensure the correctness of the code is
worth mentioning; for concreteness, we describe it in the context of
slice obstructions.  We initially ran each new slice obstruction on
the whole of \PS, even those knots which had already been identified
as ribbon.  Recall \PS\ is more than 42\% ribbon knots, and we
identified most of those early in our computations.  If the code had a
bug causing it to randomly report it had found this slice obstruction
for 1 in \num{100000} knots, we would have seen about 16 ribbon knots
being flagged as ``not smoothly slice''.  Even near the end, when we
were attacking the last \num{50000} knots with very time-consuming
computations, each run included a substantial portion of knots known
not to have the property being checked for as a trap for 
bugs.


\subsection{Acknowledgements} We thank Robert Lipshitz, Ciprian
Manolescu, Maggie Miller, Lisa Piccirillo, Jake Rasmussen, Danny
Ruberman, Dirk Sch\"utz, Zolt\'an Szab\'o, and Alex Zupan for
helpful conversations and correspondence.  This work was partially
conducted at the Analytic and Geometric Aspects of Gauge Theory program
held at SLMath (formerly MSRI).  Dunfield was partially supported by US
NSF grants DMS-1811156 and DMS-2303572 and by a fellowship from the
Simons Foundation (673067, Dunfield).  Gong was partially supported by
NSF grants DMS-2055736 and DMS-2340465.

%% file: plots/knot_count.tex
\begin{tikzoverlay*}[height=\nmdsubfigurewidth]{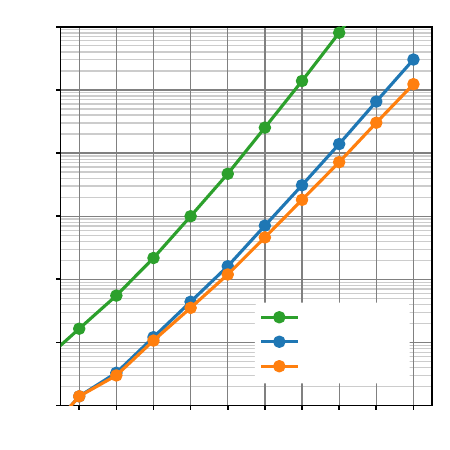}
  \draw (17.602674, 8) node[below] {$10$};
  \draw (25.855024, 8) node[below] {$11$};
  \draw (34.107374, 8) node[below] {$12$};
  \draw (42.359724, 8) node[below] {$13$};
  \draw (50.612074, 8) node[below] {$14$};
  \draw (58.864425, 8) node[below] {$15$};
  \draw (67.116775, 8) node[below] {$16$};
  \draw (75.369125, 8) node[below] {$17$};
  \draw (83.621475, 8) node[below] {$18$};
  \draw (91.873825, 8) node[below] {$19$};
  \draw (12, 9.879630) node[left] {${10^{1}}$};
  \draw (12, 23.905247) node[left] {${10^{2}}$};
  \draw (12, 37.930864) node[left] {${10^{3}}$};
  \draw (12, 51.956481) node[left] {${10^{4}}$};
  \draw (12, 65.982099) node[left] {${10^{5}}$};
  \draw (12, 80.007716) node[left] {${10^{6}}$};
  \draw (12, 94.033333) node[left] {${10^{7}}$};
  \draw (66, 29.7) node[right] {prime};
  \draw (66, 24.3) node[right] {plaus. slice};
  \draw (66, 19.2) node[right] {ribbon};
  \draw (55.749576, 1) node[below, axis label] {Crossings};
  \draw (2, 50) node[rotate=90, anchor=south, axis label] {Number of knots}; 
\end{tikzoverlay*}

%% file: plots/knot_proportion.tex
\begin{tikzoverlay*}[width=\nmdsubfigurewidth]%
  {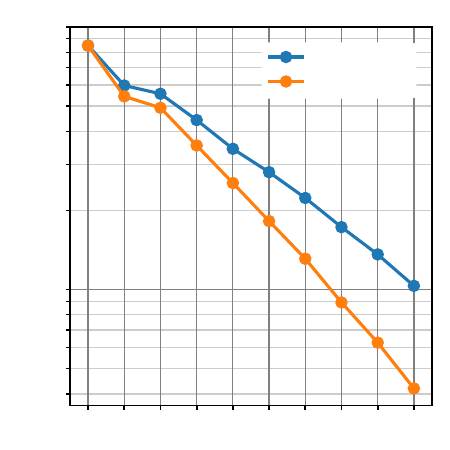}
  \draw (19.567449, 8) node[below] {$10$};
  \draw (27.612980, 8) node[below] {$11$};
  \draw (35.658512, 8) node[below] {$12$};
  \draw (43.704044, 8) node[below] {$13$};
  \draw (51.749576, 8) node[below] {$14$};
  \draw (59.795107, 8) node[below] {$15$};
  \draw (67.840639, 8) node[below] {$16$};
  \draw (75.886171, 8) node[below] {$17$};
  \draw (83.931702, 8) node[below] {$18$};
  \draw (91.977234, 8) node[below] {$19$};
  \draw (15, 18.387154) node[left] {$0.5\%$};
  \draw (15, 36.187154) node[left] {$1\%$};
  \draw (15, 53.5) node[left] {$2\%$};
  \draw (15, 76.5) node[left] {$5\%$};
  \draw (15, 94.533333) node[left] {$10\%$};
  \draw (67, 87.4) node[right] {plaus. slice};
  \draw (67, 82.1) node[right] {ribbon};
  \draw (55.749576, 1) node[below, axis label] {Crossings};
  \draw (7, 50) node[rotate=90, anchor=south, axis label] {Portion of prime knots}; 
\end{tikzoverlay*}

%% file: plots/volume_hist.tex
\begin{tikzoverlay*}[width=\nmdsubfigurewidth]{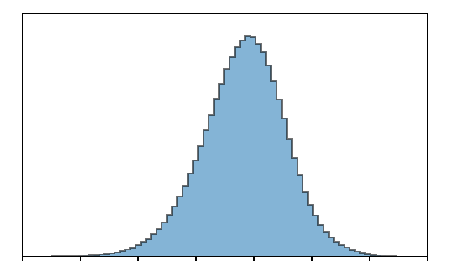}
  \begin{scope}[annotate plot]
  \draw (30, 50) node[left] {$\mu=29.1$}
               ++(0, -6) node[left] {$\mathrm{med}=29.3$}
               ++(0, -6) node[left] {$\sigma=\hphantom{0}3.6$};
  \end{scope}
  \draw (5.000000, 0.787037) node[below] {$10$};
  \draw (17.857143, 0.787037) node[below] {$15$};
  \draw (30.714286, 0.787037) node[below] {$20$};
  \draw (43.571429, 0.787037) node[below] {$25$};
  \draw (56.428571, 0.787037) node[below] {$30$};
  \draw (69.285714, 0.787037) node[below] {$35$};
  \draw (82.142857, 0.787037) node[below] {$40$};
  \draw (95.000000, 0.787037) node[below] {$45$};
  \draw (50, -5) node[below, axis label]
        {Hyperbolic volume of $S^3 \setminus K$};
  \begin{scope}[shift={(-20.71428571, 3.33333333)},
                xscale=2.57142857, yscale=467.41610338]
  \end{scope}
\end{tikzoverlay*}

%% file: plots/sqrt_det_hist.tex
\begin{tikzoverlay*}[width=\nmdsubfigurewidth]{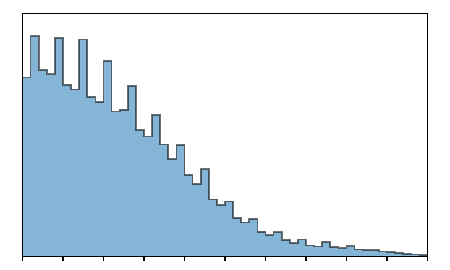}
    \begin{scope}[annotate plot]
    \draw (92, 50) node[left] {$\mu=24.8$}
       ++(0, -6)  node[left] {$\mathrm{med}=21.0$}
       ++(0, -6.5)  node[left] {$\sigma=18.5$}
       ++(0, -6)  node[left] {$\max = 129\hphantom{.}$};
  \end{scope}
  \draw (5.000000, 0.787037) node[below] {$0$};
  \draw (14.000000, 0.787037) node[below] {$10$};
  \draw (23.000000, 0.787037) node[below] {$20$};
  \draw (32.000000, 0.787037) node[below] {$30$};
  \draw (41.000000, 0.787037) node[below] {$40$};
  \draw (50.000000, 0.787037) node[below] {$50$};
  \draw (59.000000, 0.787037) node[below] {$60$};
  \draw (68.000000, 0.787037) node[below] {$70$};
  \draw (77.000000, 0.787037) node[below] {$80$};
  \draw (86.000000, 0.787037) node[below] {$90$};
  \draw (95.000000, 0.787037) node[below] {$100$};
  \draw (50, -5) node[below, axis label]
        {$\sqrt{\det(K)} \in 1 + 2 \Z$};
  \begin{scope}[shift={(5.00000000, 3.33333333)},
                xscale=0.90000000, yscale=1954.84103505]
  \end{scope}
\end{tikzoverlay*}

%% file: plots/seifert_genus_hist.tex
\begin{tikzoverlay*}[width=\nmdsubfigurewidth]{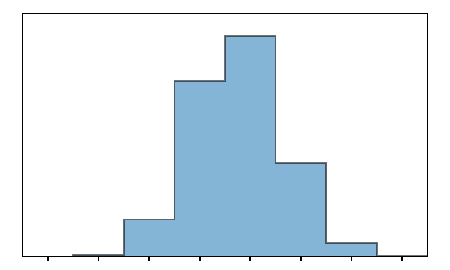}
  \begin{scope}[annotate plot]
    \draw (93, 50) node[left] {$\mu=4.8$}
           ++(0, -6) node[left] {$\mathrm{med}=5.0$}
           ++(0, -6.5) node[left] {$\sigma=0.9$};
  \end{scope}  
  \draw (10.625000, 0.787037) node[below] {$1$};
  \draw (21.875000, 0.787037) node[below] {$2$};
  \draw (33.125000, 0.787037) node[below] {$3$};
  \draw (44.375000, 0.787037) node[below] {$4$};
  \draw (55.625000, 0.787037) node[below] {$5$};
  \draw (66.875000, 0.787037) node[below] {$6$};
  \draw (78.125000, 0.787037) node[below] {$7$};
  \draw (89.375000, 0.787037) node[below] {$8$};
  \draw (50, -5) node[below, axis label]
        {Seifert genus $g_3(K)$};
  \begin{scope}[shift={(-0.62500000, 3.33333333)},
                xscale=11.25000000, yscale=133.60728839]
  \end{scope}
\end{tikzoverlay*}

%% file: plots/log_10_rk_HFK_hist.tex
\begin{tikzoverlay*}[width=\nmdsubfigurewidth]{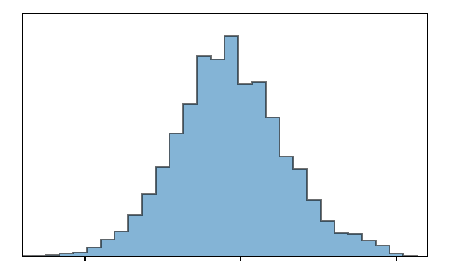}
  \begin{scope}[annotate plot]
    \draw (93, 50) node[left] {$\mu=2.9$}
          ++(0, -6)   node[left] {$\med=2.9$}
          ++(0, -6)  node[left] {$\sigma=0.4$};
  \end{scope}
  \draw (18.846154, 0.787037) node[below] {$2$};
  \draw (53.461538, 0.787037) node[below] {$3$};
  \draw (88.076923, 0.787037) node[below] {$4$};
  \draw (50, -5) node[below, axis label]
       {$\log_{10}\big(\dim \HFK (K)\big)$};
  \begin{scope}[shift={(-50.38461538, 3.33333333)},
                xscale=34.61538462, yscale=43.37504444]
  \end{scope}
\end{tikzoverlay*}

%% file: plots/vol_vs_log_HFK.tex
\begin{tikzoverlay*}[width=9.3cm]%
  {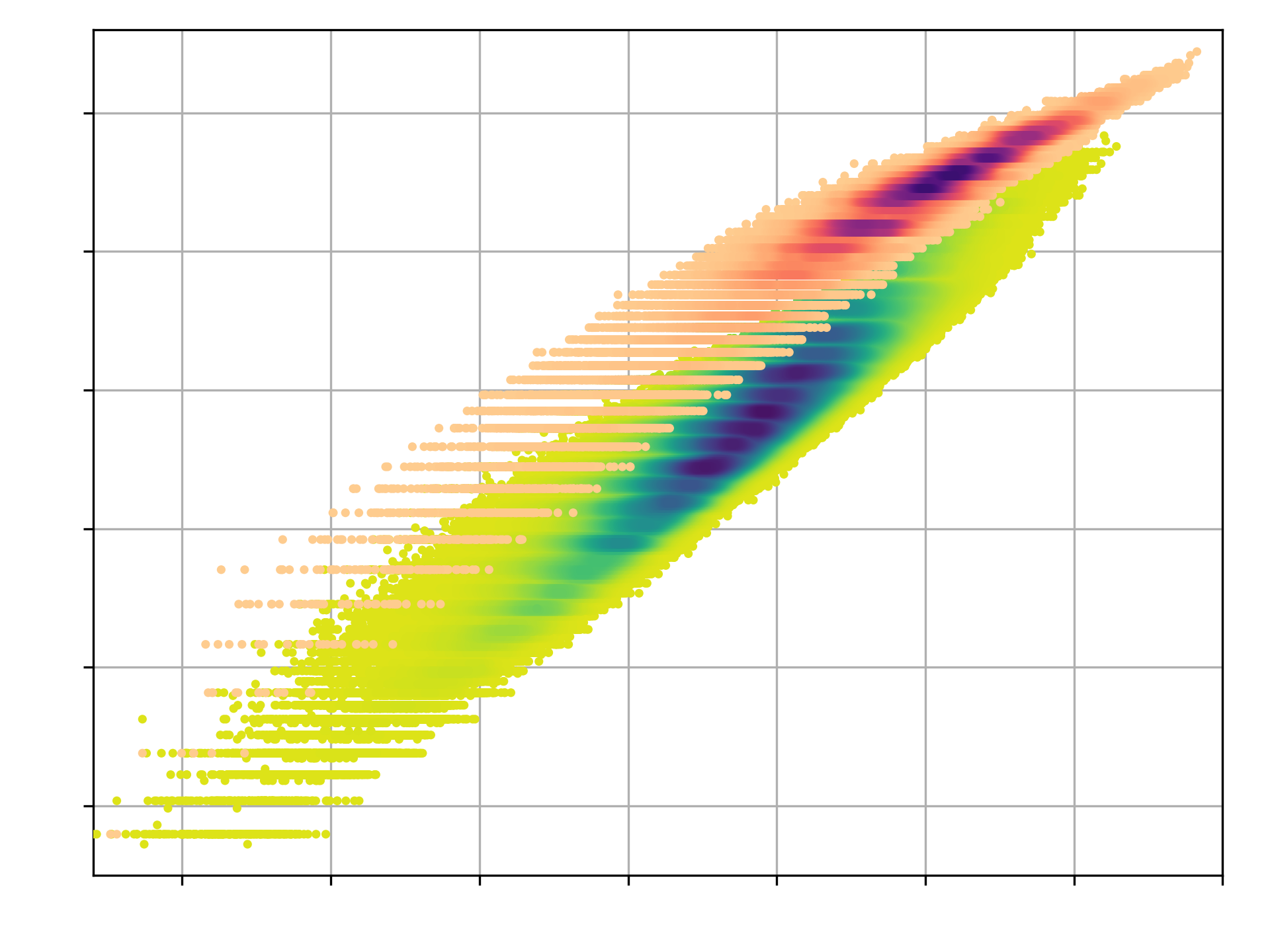}
  \draw (14.342219, 4.531250) node[below] {$10$};
  \draw (26.046121, 4.531250) node[below] {$15$};
  \draw (37.750023, 4.531250) node[below] {$20$};
  \draw (49.453925, 4.531250) node[below] {$25$};
  \draw (61.157826, 4.531250) node[below] {$30$};
  \draw (72.861728, 4.531250) node[below] {$35$};
  \draw (84.565630, 4.531250) node[below] {$40$};
  \draw (96.269531, 4.531250) node[below] {$45$};
  \draw (5.800781, 11.509847) node[left] {$1.5$};
  \draw (5.800781, 22.428848) node[left] {$2.0$};
  \draw (5.800781, 33.347848) node[left] {$2.5$};
  \draw (5.800781, 44.266849) node[left] {$3.0$};
  \draw (5.800781, 55.185849) node[left] {$3.5$};
  \draw (5.800781, 66.104850) node[left] {$4.0$};
    \draw (50, -2) node[below, axis label]
     {Hyperbolic volume of $S^3 \setminus K$};
  \draw (-2, 38) node[rotate=90, anchor=south, axis label]
     {$\log_{10}\big(\dim \HFK (K)\big)$};
  \begin{scope}[shift={(-9.06558388, -21.24715392)},
                xscale=2.34078034, yscale=21.83800091]
  \end{scope}
\end{tikzoverlay*}

%% file: plots/18nh_00000601_smooth_gaps_20_0.3_width_800.tikz
  \begin{scope}
    \draw (1.78, 7.26) .. controls (2.08, 8.12) and (2.76, 8.84) .. 
          (3.65, 8.86) .. controls (4.30, 8.88) and (4.97, 8.64) .. (4.97, 8.07);
    \draw (4.97, 7.55) .. controls (4.97, 6.69) and (4.97, 5.83) .. (4.97, 4.97);
    \draw (4.97, 4.97) .. controls (4.97, 4.74) and (4.97, 4.50) .. (4.97, 4.26);
    \draw (4.97, 3.73) .. controls (4.97, 3.49) and (4.97, 3.25) .. (4.97, 3.02);
    \draw (4.97, 3.02) .. controls (4.97, 1.98) and (4.26, 1.06) .. (3.28, 1.06);
    \draw (2.75, 1.06) .. controls (2.19, 1.06) and (1.62, 1.06) .. (1.06, 1.06);
    \draw (1.06, 1.06) .. controls (0.08, 1.06) and (0.08, 2.60) .. 
          (0.08, 3.90) .. controls (0.08, 5.28) and (0.28, 6.80) .. (1.52, 7.18);
    \draw (2.03, 7.33) .. controls (2.98, 7.63) and (3.97, 7.79) .. (4.97, 7.81);
    \draw (4.97, 7.81) .. controls (5.62, 7.82) and (6.28, 7.84) .. (6.93, 7.85);
    \draw (6.93, 7.85) .. controls (7.49, 7.86) and (8.06, 7.87) .. (8.62, 7.88);
    \draw (9.15, 7.90) .. controls (9.51, 7.90) and (9.87, 7.74) .. 
          (9.87, 7.42) .. controls (9.87, 7.02) and (9.36, 6.93) .. (8.89, 6.93);
    \draw (8.89, 6.93) .. controls (8.39, 6.93) and (7.91, 6.67) .. (7.91, 6.22);
    \draw (7.91, 5.69) .. controls (7.91, 5.54) and (7.91, 5.39) .. (7.91, 5.24);
    \draw (7.91, 4.71) .. controls (7.91, 4.35) and (7.74, 3.99) .. 
          (7.42, 3.99) .. controls (7.02, 3.99) and (6.93, 4.50) .. (6.93, 4.97);
    \draw (6.93, 4.97) .. controls (6.93, 5.30) and (6.93, 5.63) .. (6.93, 5.95);
    \draw (6.93, 5.95) .. controls (6.93, 6.50) and (6.93, 7.04) .. (6.93, 7.59);
    \draw (6.93, 8.11) .. controls (6.93, 8.58) and (7.40, 8.89) .. 
          (7.91, 8.89) .. controls (8.45, 8.89) and (8.89, 8.44) .. (8.89, 7.89);
    \draw (8.89, 7.89) .. controls (8.89, 7.66) and (8.89, 7.43) .. (8.89, 7.20);
    \draw (8.89, 6.67) .. controls (8.89, 6.21) and (8.41, 5.95) .. (7.91, 5.95);
    \draw (7.91, 5.95) .. controls (7.67, 5.95) and (7.43, 5.95) .. (7.19, 5.95);
    \draw (6.66, 5.95) .. controls (6.27, 5.95) and (5.95, 5.63) .. (5.95, 5.24);
    \draw (5.95, 4.71) .. controls (5.95, 4.26) and (5.47, 3.99) .. (4.97, 3.99);
    \draw (4.97, 3.99) .. controls (4.64, 3.99) and (4.32, 3.99) .. (3.99, 3.99);
    \draw (3.99, 3.99) .. controls (3.05, 3.99) and (2.03, 3.82) .. 
          (2.03, 3.02) .. controls (2.03, 2.52) and (2.30, 2.04) .. (2.75, 2.04);
    \draw (3.28, 2.04) .. controls (3.73, 2.04) and (3.99, 2.52) .. (3.99, 3.02);
    \draw (3.99, 3.02) .. controls (3.99, 3.25) and (3.99, 3.49) .. (3.99, 3.73);
    \draw (3.99, 4.26) .. controls (3.99, 4.65) and (4.31, 4.97) .. (4.71, 4.97);
    \draw (5.24, 4.97) .. controls (5.47, 4.97) and (5.71, 4.97) .. (5.95, 4.97);
    \draw (5.95, 4.97) .. controls (6.19, 4.97) and (6.43, 4.97) .. (6.66, 4.97);
    \draw (7.19, 4.97) .. controls (7.43, 4.97) and (7.67, 4.97) .. (7.91, 4.97);
    \draw (7.91, 4.97) .. controls (8.45, 4.97) and (8.89, 4.53) .. 
          (8.89, 3.99) .. controls (8.89, 3.02) and (6.86, 3.02) .. (5.24, 3.02);
    \draw (4.71, 3.02) .. controls (4.56, 3.02) and (4.41, 3.02) .. (4.26, 3.02);
    \draw (3.73, 3.02) .. controls (3.28, 3.02) and (3.01, 2.53) .. (3.01, 2.04);
    \draw (3.01, 2.04) .. controls (3.01, 1.71) and (3.01, 1.38) .. (3.01, 1.06);
    \draw (3.01, 1.06) .. controls (3.01, 0.52) and (2.57, 0.08) .. 
          (2.03, 0.08) .. controls (1.54, 0.08) and (1.06, 0.34) .. (1.06, 0.79);
    \draw (1.06, 1.32) .. controls (1.07, 3.33) and (1.11, 5.36) .. (1.78, 7.26);
  \end{scope}

%% file: plots/0-friend_smooth_gaps_20_0.3_width_800.tikz
  \begin{scope}
    \draw (0.74, 6.60) .. controls (0.35, 6.60) and (0.09, 6.97) .. 
          (0.09, 7.38) .. controls (0.08, 7.90) and (0.56, 8.26) .. 
          (1.10, 8.38) .. controls (2.12, 8.60) and (3.29, 8.59) .. (4.40, 8.58);
    \draw (4.93, 8.58) .. controls (6.38, 8.57) and (7.82, 8.57) .. (9.27, 8.56);
    \draw (9.27, 8.56) .. controls (9.87, 8.55) and (9.87, 7.67) .. 
          (9.87, 6.92) .. controls (9.87, 6.21) and (9.87, 5.29) .. (9.53, 5.29);
    \draw (9.06, 5.29) .. controls (8.97, 5.29) and (8.88, 5.29) .. (8.79, 5.29);
    \draw (8.39, 5.29) .. controls (8.30, 5.29) and (8.21, 5.29) .. (8.12, 5.29);
    \draw (7.73, 5.29) .. controls (7.64, 5.29) and (7.56, 5.29) .. (7.47, 5.29);
    \draw (7.01, 5.29) .. controls (6.62, 5.29) and (6.62, 4.57) .. (6.62, 3.98);
    \draw (6.62, 3.98) .. controls (6.62, 3.00) and (6.19, 2.02) .. (5.31, 2.02);
    \draw (5.31, 2.02) .. controls (4.23, 2.01) and (3.14, 2.00) .. (2.05, 1.99);
    \draw (2.05, 1.99) .. controls (1.83, 1.99) and (1.62, 1.99) .. (1.40, 1.99);
    \draw (1.40, 1.99) .. controls (0.97, 1.98) and (0.83, 2.50) .. 
          (0.78, 2.99) .. controls (0.69, 4.00) and (0.71, 5.22) .. (0.73, 6.33);
    \draw (0.74, 6.86) .. controls (0.75, 7.25) and (1.70, 7.25) .. (2.44, 7.25);
    \draw (2.97, 7.25) .. controls (3.45, 7.25) and (3.92, 7.25) .. (4.40, 7.25);
    \draw (4.93, 7.25) .. controls (6.03, 7.25) and (7.27, 7.16) .. (7.27, 6.21);
    \draw (7.27, 5.75) .. controls (7.27, 5.59) and (7.27, 5.44) .. (7.27, 5.29);
    \draw (7.27, 5.29) .. controls (7.27, 4.70) and (7.27, 3.98) .. (6.88, 3.98);
    \draw (6.42, 3.98) .. controls (6.27, 3.98) and (6.12, 3.98) .. (5.97, 3.98);
    \draw (5.97, 3.98) .. controls (5.75, 3.98) and (5.53, 3.98) .. (5.31, 3.98);
    \draw (5.31, 3.98) .. controls (4.99, 3.98) and (4.66, 4.09) .. (4.66, 4.37);
    \draw (4.66, 4.83) .. controls (4.66, 4.92) and (4.66, 5.01) .. (4.66, 5.09);
    \draw (4.66, 5.49) .. controls (4.66, 5.64) and (4.66, 5.79) .. (4.66, 5.94);
    \draw (4.66, 5.94) .. controls (4.66, 6.38) and (4.66, 6.81) .. (4.66, 7.25);
    \draw (4.66, 7.25) .. controls (4.66, 7.47) and (4.66, 7.69) .. (4.66, 7.90);
    \draw (4.66, 7.90) .. controls (4.66, 8.13) and (4.66, 8.36) .. (4.66, 8.58);
    \draw (4.66, 8.58) .. controls (4.66, 9.21) and (5.94, 9.21) .. 
          (6.97, 9.21) .. controls (7.95, 9.21) and (9.27, 9.21) .. (9.27, 8.82);
    \draw (9.27, 8.29) .. controls (9.27, 7.29) and (9.27, 6.29) .. (9.27, 5.29);
    \draw (9.27, 5.29) .. controls (9.26, 4.48) and (9.26, 3.68) .. 
          (9.26, 2.87) .. controls (9.26, 2.21) and (9.25, 1.39) .. (8.84, 1.39);
    \draw (8.38, 1.39) .. controls (8.30, 1.39) and (8.21, 1.39) .. (8.12, 1.39);
    \draw (7.66, 1.39) .. controls (6.66, 1.38) and (5.31, 1.37) .. (5.31, 1.75);
    \draw (5.31, 2.28) .. controls (5.31, 2.63) and (5.31, 2.98) .. (5.31, 3.33);
    \draw (5.31, 3.33) .. controls (5.31, 3.48) and (5.31, 3.64) .. (5.31, 3.79);
    \draw (5.31, 4.25) .. controls (5.31, 4.53) and (4.98, 4.64) .. (4.66, 4.64);
    \draw (4.66, 4.64) .. controls (4.44, 4.64) and (4.23, 4.64) .. (4.01, 4.64);
    \draw (4.01, 4.64) .. controls (3.66, 4.64) and (3.32, 4.64) .. (2.97, 4.64);
    \draw (2.44, 4.64) .. controls (2.05, 4.64) and (2.05, 3.92) .. (2.05, 3.33);
    \draw (2.05, 3.33) .. controls (2.05, 2.97) and (2.05, 2.61) .. (2.05, 2.26);
    \draw (2.05, 1.73) .. controls (2.05, 0.72) and (3.65, 0.72) .. 
          (4.99, 0.72) .. controls (6.28, 0.72) and (7.93, 0.72) .. (7.93, 1.39);
    \draw (7.93, 1.39) .. controls (7.93, 2.69) and (7.93, 3.99) .. (7.93, 5.29);
    \draw (7.93, 5.29) .. controls (7.93, 5.65) and (7.63, 5.94) .. (7.27, 5.94);
    \draw (7.27, 5.94) .. controls (6.49, 5.94) and (5.71, 5.94) .. (4.93, 5.94);
    \draw (4.47, 5.94) .. controls (4.38, 5.94) and (4.29, 5.94) .. (4.21, 5.94);
    \draw (3.74, 5.94) .. controls (3.54, 5.94) and (3.36, 5.81) .. 
          (3.36, 5.62) .. controls (3.36, 5.35) and (3.70, 5.29) .. (4.01, 5.29);
    \draw (4.01, 5.29) .. controls (4.23, 5.29) and (4.44, 5.29) .. (4.66, 5.29);
    \draw (4.66, 5.29) .. controls (5.34, 5.29) and (5.97, 4.88) .. (5.97, 4.25);
    \draw (5.97, 3.72) .. controls (5.97, 3.51) and (5.79, 3.33) .. (5.58, 3.33);
    \draw (5.05, 3.33) .. controls (4.79, 3.33) and (4.53, 3.33) .. (4.27, 3.33);
    \draw (3.74, 3.33) .. controls (3.40, 3.33) and (3.05, 3.33) .. (2.70, 3.33);
    \draw (2.70, 3.33) .. controls (2.55, 3.33) and (2.40, 3.33) .. (2.25, 3.33);
    \draw (1.79, 3.33) .. controls (1.40, 3.33) and (1.40, 2.75) .. (1.40, 2.25);
    \draw (1.40, 1.72) .. controls (1.40, 0.26) and (3.29, 0.08) .. 
          (4.99, 0.08) .. controls (6.64, 0.08) and (8.58, 0.08) .. (8.58, 1.39);
    \draw (8.58, 1.39) .. controls (8.58, 2.69) and (8.58, 3.99) .. (8.59, 5.29);
    \draw (8.59, 5.29) .. controls (8.59, 5.93) and (8.55, 6.58) .. 
          (8.37, 7.19) .. controls (8.14, 7.91) and (6.30, 7.91) .. (4.93, 7.91);
    \draw (4.40, 7.90) .. controls (3.61, 7.90) and (2.70, 7.90) .. (2.70, 7.25);
    \draw (2.70, 7.25) .. controls (2.70, 7.10) and (2.70, 6.94) .. (2.70, 6.79);
    \draw (2.70, 6.33) .. controls (2.70, 5.77) and (2.70, 5.20) .. (2.70, 4.64);
    \draw (2.70, 4.64) .. controls (2.70, 4.29) and (2.70, 3.94) .. (2.70, 3.60);
    \draw (2.70, 3.07) .. controls (2.70, 2.79) and (3.04, 2.68) .. 
          (3.36, 2.68) .. controls (3.72, 2.68) and (4.01, 2.97) .. (4.01, 3.33);
    \draw (4.01, 3.33) .. controls (4.01, 3.68) and (4.01, 4.03) .. (4.01, 4.37);
    \draw (4.01, 4.83) .. controls (4.01, 4.92) and (4.01, 5.01) .. (4.01, 5.09);
    \draw (4.01, 5.49) .. controls (4.01, 5.64) and (4.01, 5.79) .. (4.01, 5.94);
    \draw (4.01, 5.94) .. controls (4.01, 6.48) and (3.33, 6.60) .. (2.70, 6.60);
    \draw (2.70, 6.60) .. controls (2.05, 6.60) and (1.39, 6.60) .. (0.74, 6.60);
  \end{scope}

%% file: plots/concord_intro.tex
\begin{tikzoverlay*}[width=0.95\textwidth]{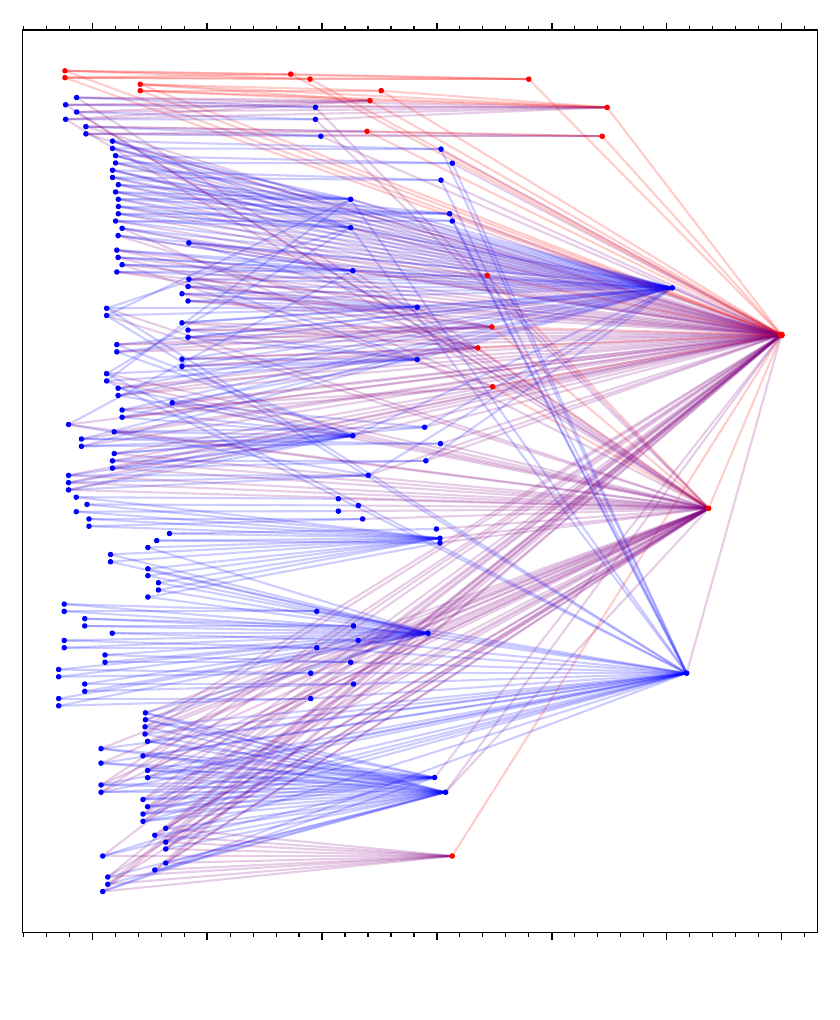}
  \draw (50.000000, 5.0) node[below] {\small Hyperbolic volume
    of $S^3 \setminus K$};
  \draw (93.019481, 8.670635) node[below] {$0$};

  \draw (79.349955, 8.670635) node[below] {$5$};

  \draw (65.680429, 8.670635) node[below] {$10$};

  \draw (52.010903, 8.670635) node[below] {$15$};

  \draw (38.341377, 8.670635) node[below] {$20$};

  \draw (24.671851, 8.670635) node[below] {$25$};

  \draw (11.002325, 8.670635) node[below] {$30$};

  \begin{scope}[shift={(93.01948052, 14.90599745)},
                xscale=-2.73390519, yscale=0.00511876]
  \end{scope}
  \node[fill=white] at (97.8, 81.9) {unknot};
              
\end{tikzoverlay*}

%% file: band_search.tex
\section{Searching for ribbon disks and concordances}
\label{sec: bands}

\begin{figure}
\centering
\includegraphics[width=0.45\textwidth]{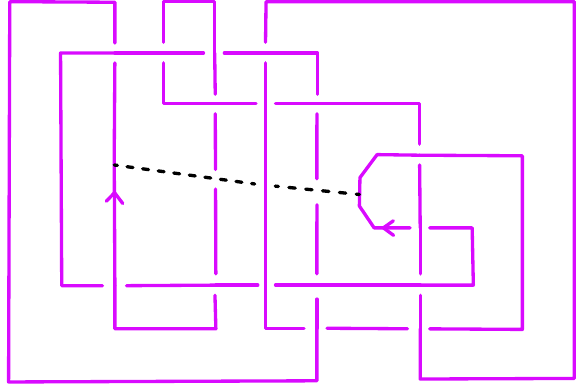} \hspace{1cm} %
\includegraphics[width=0.45\textwidth]{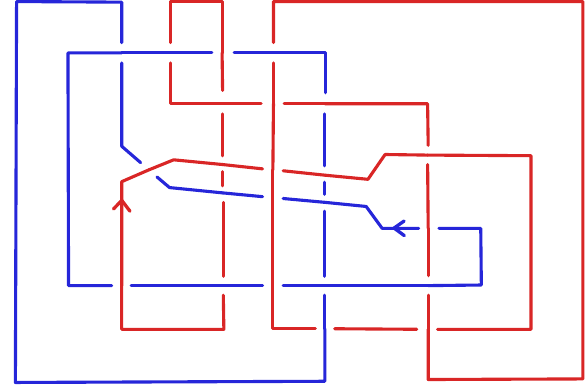}
\caption{A band move on the knot $K = 16n424914$ at left resulting in the
  link at right.  The band is added along the arc shown by the
  dotted line and has a half-twist so that the new link $L$ has two
  components.  In fact, $L$ is the unlink, showing that $K$ is ribbon. }
\label{fig: 16n424914}
\end{figure}

This section is devoted to describing the process we used to show:

\begin{theorem}\label{thm: ribbon}
  At least \NumRibbonKnots\ knots in \PS\ are ribbon and hence smoothly slice.
\end{theorem}
First, we recall a standard description for ribbon disks that we use
throughout. For a link $L_0$, a \emph{band move} is a 1-handle
attachment resulting in a link $L_1$ with one more component; see
Figure~\ref{fig: 16n424914}. A band move is determined by an arc with
endpoints in $L_0$ together with how the band twists around it. A band
move encodes an embedded surface in $S^3 \times I$ whose interesting
component is a pair of pants $P$ with waist attached to $L_0$ and
cuffs attached to $L_1$. Here, one arranges $P$ to have a single
critical point with respect to the $I$-coordinate, which has
index~1. When a knot $K$ has a sequence of $k$ band moves resulting in
an unlink $U$, it is ribbon: the cobordism in $S^3 \times I$ made up
of $k$ pants combines with disks capping off the components of $U$ in
$D^4$ to form a ribbon disk. Moreover, any ribbon disk can be
described in this way. (The reason a band move is required to
increase the number of components is that otherwise the corresponding
cobordism results in a non-orientable or otherwise non-planar
surface.)

\begin{figure}
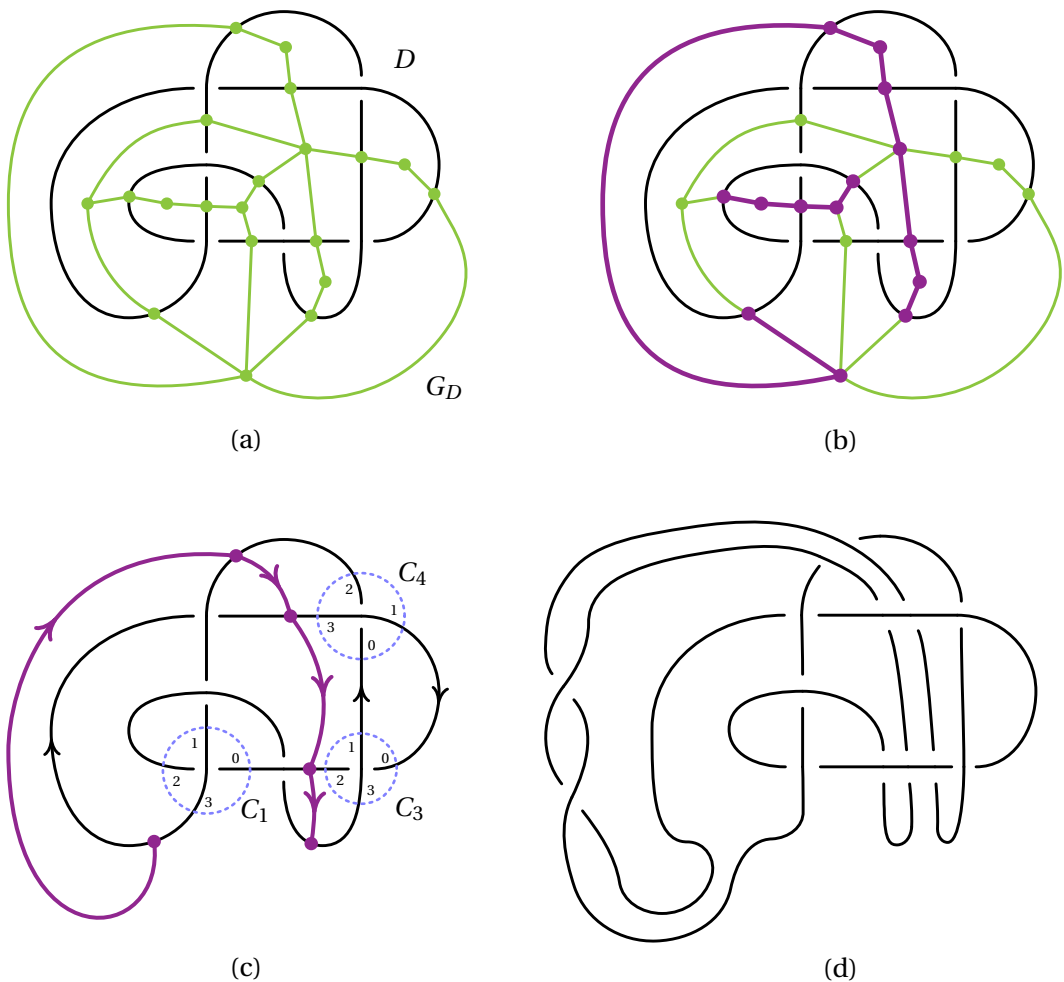

  \centering
  \input figures/diagrammatics

  \caption{Part (a) shows a diagram $D$ of the knot $K6a3$ with the
    subdivided dual graph $G_D$. Part (b) shows two simple paths in
    $G_D$ that can be used to build diagrammatic bands. Part (d) shows
    the result of doing one possible band move starting with the path
    shown in (c).  We can record the (oriented) path in (c) by the
    crossing input pairs to its left, that is
    $[(C_1, 3), (C_4, 2), (C_4, 3), (C_3, 2), (C_3, 3)]$, and add the
    over-under information and the number of half-twists ($+2$ in this
    case) to specify the band move in (d). In \cite{CodeAndData,
      SnapPy}, this is \texttt{Band([(1, 3), (4, 2), (4, 3), (3, 2),
      (3, 3)], [True, False, False], 2)}.}
  \label{fig: diag bands}
\end{figure}

\subsection{Diagrammatic bands}
For a knot $K$, there are infinitely-many isotopy classes of arcs that
can be used for a band move, to say nothing of the choice of the
twist. No bounds are known to exist for the complexity of the arcs
needed to encode a ribbon disk, but here we consider the following
finite set of possibilites. Given a diagram $D$ for a knot $K$, we
define a \emph{diagrammatic band} as follows. Consider the dual graph
to $D$, which has one vertex for every complementary region of the
planar projection of $D$ and one edge for each segment of $D$ between
two crossings. Modify this dual graph to create $G_D$ by adding a
vertex in the middle of each dual edge where it meets the original
$D$, as in Figure~\ref{fig: diag bands}(a). A diagrammatic band is one
built from a simple path in $G_D$ whose endpoints are on $K$, together
with the over-under information and the number of half-twists; see
Figure~\ref{fig: diag bands}(b, c, d) for an example.  For a given
path, the parity of the number of half-twists is determined by the
requirement that the band increases the number of components.  We used
diagrammatic bands because of technical challenges to considering
bands that cross out of and then back into the same complementary
region.

\begin{table}
  \centering
  \begin{tabular}{crl}
    \toprule
    length & \multicolumn{1}{c}{\# bands} & \% bands\\
    \midrule
    2 & \num{1684652} & 84.6 \\
    3 & \num{255129}  & 12.8 \\
    4 & \num{43474}   & \hphantom{1}2.2\\
    5 & \num{7490}    & \hphantom{1}0.4 \\
    6 & \num{100}     & \hphantom{1}0.005 \\
    7 & \num{16}      & \hphantom{1}0.0008 \\
    \bottomrule
  \end{tabular}
  \hspace{4em}
  \begin{tabular}{crc}
    \toprule
    half-twists & \multicolumn{1}{c}{\# bands} & \multicolumn{1}{c}{\% bands} \\
    \midrule
    0 & \num{1079431} & 54.2\hphantom{4} \\
    1 & \num{909006} & 45.7\hphantom{4} \\
    2 & \num{2424} &   \hphantom{4}0.12 \\
    3 & \num{0} &   \hphantom{4}0\hphantom{.12} \\
    \bottomrule
  \end{tabular}

  \caption{Statistics on the 1.99 million bands used in
    Section~\ref{sec: bands}.  Here, the length of a band is the
    number of times it interacts with the diagram, including its
    endpoints; the band used for Figure~\ref{fig: diag bands}(d) has
    length 5.  Bands with 3 half-twists were initially included in our
    search, but since nothing came of them, we subsequently restricted
    to at most 1 or 2. }
  \label{tab: band stats}
\end{table}

\subsection{Search strategy}
\label{sec: strategy}

Our search for ribbon disks was done in several stages. We started
with adding only a single band that moreover came from a path in $G_D$
that is minimal length among those joining its endpoints.  In
Figure~\ref{fig: diag bands}(b), the shorter arc is minimal length but
the longer one used in Figure~\ref{fig: diag bands}(c) is not. For
another example, Figure \ref{fig: 16n424914} depicts a band added
along a shortest path.  Initially, we limited ourselves to three
half-twists.

When no shortest path yielded a band giving the unlink, we checked all
simple paths in $G_D$ up to a certain length (usually 6). If that
failed to find a ribbon disk, we initiated a multi-band search,
starting with two bands and working up to four bands.

For the multi-band search, we first considered all links resulting
from $D$ by a single band move of the specified type.  We pruned this
collection of 2-component links by keeping only those where the
following necessary conditions for topological sliceness held:
\begin{enumerate}
\item \label{item: link 0}
  the linking number between every pair of components was zero,

\item \label{item: sig 0}
  the link signature was zero, and

\item \label{item: Fox-Milnor}
  it satisfied the multivariate Fox-Milnor test; that is, the
  Alexander nullity was one less than the number of components of the
  link, and the first non-vanishing Alexander polynomial satisfied the
  Fox-Milnor condition of \cite[Corollary 12.3.14]{Kawauchi1996}.
\end{enumerate}
For links whose exteriors were hyperbolic, we also removed duplicate
copies of the same link from the collection.  After pruning, for each
link $L_i$ we added a second band in various ways to create a pile of
3-component links. These links were collectively pruned by (\ref{item:
  link 0}, \ref{item: sig 0}, \ref{item: Fox-Milnor}) and consolidated
before possibly adding a third band, and so on.  Table~\ref{tab: band
  stats} gives data about the lengths and twists of the bands used.

As part of this process, we also searched for ribbon concordances
between knots in \PS\ to generate data for Section~\ref{sec: ribbon
  graph}.  Suppose a band move on a knot $K$ results in a split
two-component link where one component is an unknot and the other
component is a knot $K'$ different from $K$. Such a band then
prescribes a ribbon concordance from $K$ to $K'$ with one saddle point
and one minimum.

Throughout, identifying a link $L$ as the unlink was primarily done by
building an ideal triangulation $T$ of $S^3 \setminus L$, simplifying
it to $T'$ using Pachner moves and standard heuristics, computing a
presentation for $\pi_1(T')$ and simplifying that, and then finally
checking if said presentation had no relators.  At least with SnapPy
\cite{SnapPy}, we found this method both faster than using
diagrammatic simplification and having an even higher success rate.
For example, it correctly identifies the 382 ``really hard unknot
diagrams'' of \cite[Appendix~A]{ApplebaumEtAl}.

\subsection{Shaking  diagrams}
\label{sec: shaking}

Since we only used diagrammatic bands, the choice of starting diagram
is important.  We first tried the minimal diagram specified for each
knot in \cite{Burton2020}. When that failed, we generated a diverse
collection of possibly slightly larger diagrams and used these to
search for bands. Initially, we tried generating other diagrams using
random Reidemeister moves, but found the following method suggested by
Question 4 in \cite[\S 12]{DunfieldObeidinRudd2024} to be more
effective. Starting with a diagram $D$, we computed a triangulation
$T$ of its exterior and then used the method of
\cite{DunfieldObeidinRudd2024} to get another diagram $D'$ for the
knot. Such ``diagram shaking'' was used for less than 0.25\% of the
ribbon knots in Theorem~\ref{thm: ribbon}.

\subsection{Caching 2-component ribbon links}
\label{sec: cache}

For a collection of more than \num{146000} ribbon knots in \PS\ where
we needed at least two bands, we saved the 2-component link created by
the first band move.  This resulted in a database of \num{12143}
ribbon links included with \cite{CodeAndData}. Two of these links
appeared about \num{19000} times each ($L10n36 = 10^2_{104}$ and
$L10n32 = 10^2_{26}$) and more than 130 appeared at least 100 times,
but about \num{7200} appeared exactly once.  With this in hand, it is
possible to accelerate the search for ribbon disks, using SnapPy to
check if the result of adding a single band is a known ribbon link.
Moreover, certain of these links have hard to find ribbon disks. For
example, the link $R$ we called \verb|ribbon_2_18_86ba6dfc| initially
appeared only for the knot $17nh_{0000980}$.  However, we subsequently
found 66 other knots where one band move produced $R$, showing that
those knots were ribbon.  Overall, this technique identified 206
additional ribbon knots.

\subsection{Searching by connected sum with ribbon knots}
\label{sec: slice only}

Another method for detecting smoothly slice knots uses the fact that
knot concordance classes form a group under connected sum: for a knot
$K$ and a smoothly slice knot $J$, if $K \# J$ is smoothly slice, then
$K$ is as well.  Using this, we proceeded as follows: for a knot $K$,
take the connected sum of $K$ with a known ribbon knot $R$, and then
attempt to show that $K \# R$ is ribbon using Section~\ref{sec:
  strategy}. (A knot $K$ is smoothly slice if and only if there is a
such ribbon knot $R$ \cite[Lemma~2.5]{Teichner2011}, so in theory this
is a universal test.)

We did this search with two specific ribbon knots: the connected sum
of a trefoil $T_{2,3}$ with its mirror $T_{2,-3}$, and the connected
sum of two copies of the figure-8 knot $K4a1$. The major advantage of
using a ribbon knot of the form $A \# B$ is that its connected sum
with $K$ can be viewed as adding $A$ and $B$ separately in different
places along $K$, ideally where there are then obvious band moves
that decrease the number of crossings.  The motivating example comes
from \cite[\S 8]{HeraldKirkLivingston2010}, where $K12a990$ is shown
to be smoothly slice using $T_{2,3} \# T_{2,-3}$.

Fix an orientation on a knot $K$ and suppose its diagram has a bigon
region whose two sides have parallel orientations.  If we form the
connected sum of $K$ with a trefoil knot with oppositely signed
crossings near this bigon, then we can do a band move that simplifies
the resulting 2-component link, as in
Figure~\ref{adding_in_trefoil_oriented}.  If we repeat this process
with a mirror image of the trefoil next to a different bigon with
crossings of opposite sign, this amounts to taking the connected sum
with $T_{2,3} \# T_{2,-3}$ and then adding two bands.  This results in
a diagram for a 3-component link $L$ with two fewer crossings than that
of $K$.  We then searched for ribbon disks for $L$ as in
Section~\ref{sec: strategy}.  Figure~\ref{19nh_051162051} shows an
example where this strategy works.

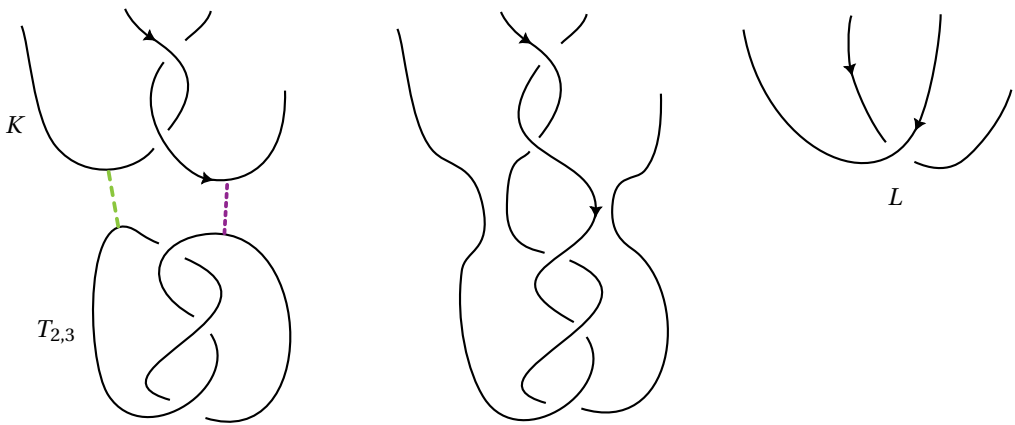
\begin{figure}
  \centering
  \begin{tikzpicture}[nmdstd]
    \definecolor{connsum}{cmyk}{0.5,0.0,1.0,0.0}
    \colorlet{bandsum}{-connsum}
    \begin{tikzoverlay*}[width=0.9\textwidth]{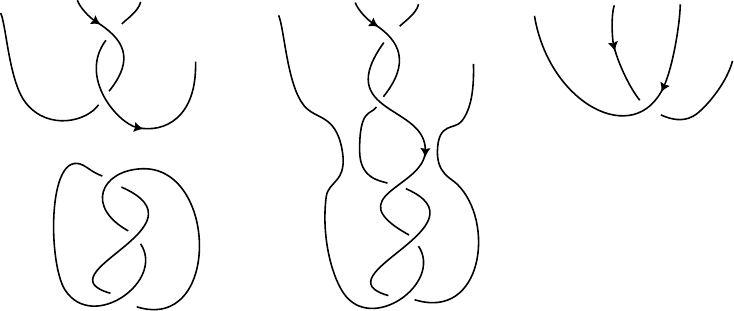}
      \node[left] at (1.5,30.5) {$K$};
      \node[left] at (6.7, 9.3) {$T_{2, 3}$};
      \node[right] at (86.4,23.2) {$L$};
      \begin{scope}[line width=1.4]
        \draw[color=connsum, dashed] (9.9,20.1) -- (8.9,25.85);
        \draw[color=bandsum, dotted] (20.6,19.4) -- (20.9,24.8);
      \end{scope}
    \end{tikzoverlay*}
  \end{tikzpicture}
  \caption{We perform the connected sum of the oriented knot $K$ 
    at left with $T_{2, 3}$ along the dark dotted line, and then do an
    untwisted band move along the light dashed line.  The result is
    shown in the middle.  After simplifying, we get the
    2-component link $L$ at right which has one fewer crossing than
    $K$.}
\label{adding_in_trefoil_oriented}
\end{figure}

\begin{figure}
  \centering
  \begin{tikzpicture}[nmdstd]
    \definecolor{connsum}{cmyk}{0.5,0.0,1.0,0.0}
    \colorlet{bandsum}{-connsum}
    \begin{tikzoverlay*}[width=0.6\textwidth]{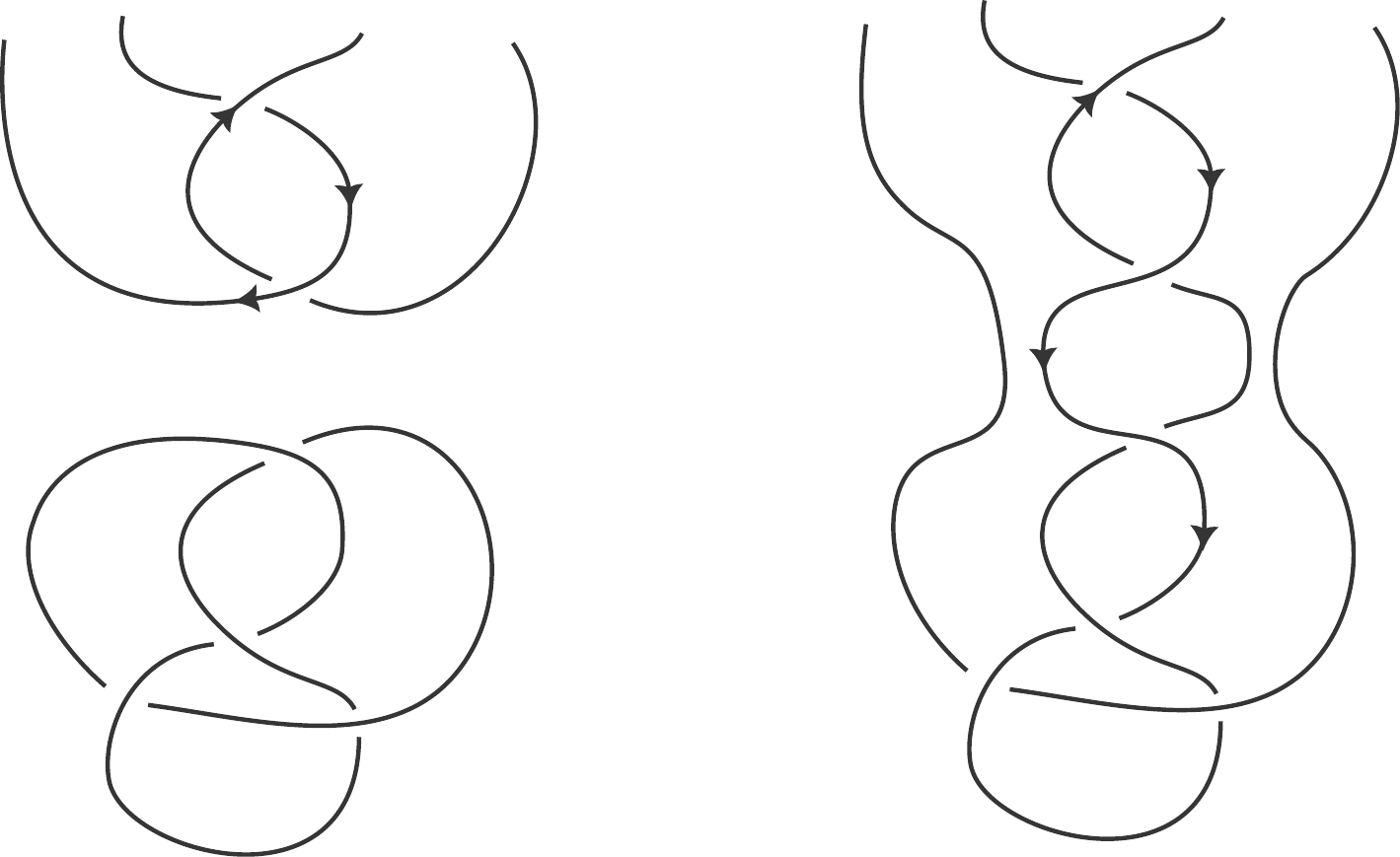}
      \node[left] at (1.1,46.2) {$K$};
      \node[left] at (6.1,11.0) {$K4a1$};
      \node[right] at (94.9,35.7) {$L$};
      \begin{scope}[line width=1.4]
        \draw[color=connsum, dashed] (10.6,29.85) -- (9.3,40.1);
        \draw[color=bandsum, dotted] (27.3,31.1) -- (27.8,39.1);
      \end{scope}
    \end{tikzoverlay*}
  \end{tikzpicture}
  \caption{Performing the connected sum of $K$ with $K4a1$ along the
    dark dotted line and then adding an untwisted band along the light
    dashed line results in the 2-component link $L$ at left.}
\label{adding_4_1s}
\end{figure}
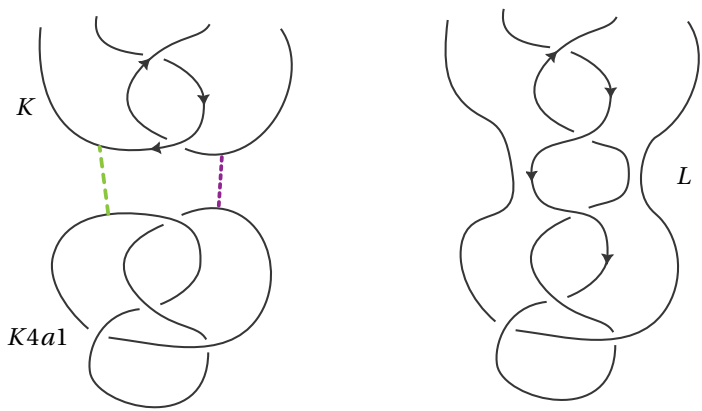

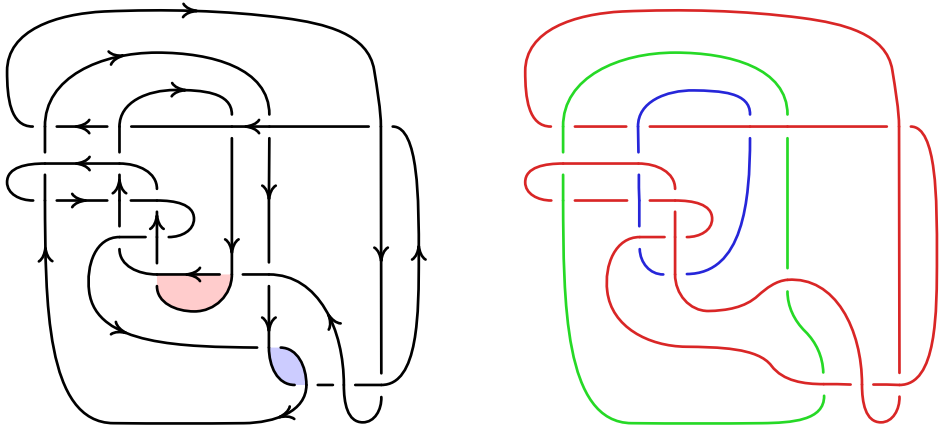
\begin{figure}
\centering
\begin{tikzpicture}[nmdstd, line width=1.01, scale=0.85]
\input figures/19nh_051162051.tex
\end{tikzpicture}
\caption{For the knot $K = 19nh_{051162051}$ at left, consider the two
  shaded bigons, one with positive and one with negative crossings.
  Taking the connected sum with $T_{2,3} \# T_{2,-3}$ and adding two
  bands as in Figure~\ref{adding_in_trefoil_oriented} and its mirror
  image yields the 3-component link $L$ at right.  As $L$ is an
  unlink, $K$ is smoothly slice.}
\label{19nh_051162051}
\end{figure}

The other knot we added was $K4a1 \# K4a1$, which is ribbon since
$K4a1$ is its own mirror image. As with the trefoils, each $K4a1$ was
added along with a band, as in Figure~\ref{adding_4_1s}. To ensure
this is a valid band move, the edges of the bigon must be oriented
anti-parallel, as in the figure.

Initially, we found 446 new slice knots in \PS\ by adding
$T_{2,3} \# T_{2,-3}$ and 67 slice knots by adding $K4a1 \#
K4a1$. However, by searching harder for ribbon disks, we eventually
found them for all 513 of these knots, as discussed in the next
subsection.

\subsection{Obscure ribbon disks}

For certain knots in \PS, we had reason to suspect that they were
ribbon, even though our initial search did not turn up any ribbon
disks. This includes the 513 smoothly slice knots from the previous
subsection, as well as 41 knots that are 0-friends with a ribbon knot
by the method of Section~\ref{sec: friends}.

For these knots, we expanded the parameters of our search (the number
of distinct diagrams tried, the number and length of bands checked,
etc).  In all but one case, namely $18nh_{00000601}$ from
Theorem~\ref{thm: mystery}, we eventually found a ribbon disk.  We
devoted more computational resources to these suspicious knots than we
could afford to do for all the unknown knots in Theorem~\ref{thm: real
  main}.  Based on the success rate of the expanded search as compared
to that of the original search, we expect there are
hundreds of other similarly obscure ribbon disks among the
\NumSmoothSliceUnknowns\ knots in \PS\ whose smooth slice status is
unknown.

The knot $19nh_{051162051}$ of Figure~\ref{19nh_051162051} is one
example where we initially discovered it was slice by
Section~\ref{sec: slice only} but eventually found a ribbon
disk. Another way we could show that knots were smoothly slice
(without a priori knowing that they are ribbon) was by showing that
their knot trace was diffeomorphic to that of a ribbon knot (see
Section~\ref{sec: friend app}).  Some knots where we initially used
the shared knot trace method but later found a ribbon disk are
$17nh_{0016322}$, $17nh_{0026540}$, and $17nh_{0298397}$. The knot
$17nh_{0016322}$ is depicted in Figure \ref{fig: disk the obscure}(a).

\begin{figure}
  \centering
  \begin{tikzpicture}[nmdstd, line width=1.00001, scale=0.9]
    \input figures/obscure_ribbon.tex
  \end{tikzpicture}
  \caption{The knot $17nh_{0016322}$ in (a) was initially shown to be
    smoothly slice by finding a ribbon knot with a diffeomorphic knot
    trace.  We suspected that the knot $19nh_{000077044}$ in (b) was
    smoothly slice because it shared a 0-surgery with a ribbon knot.
    For both knots, we eventually found a ribbon disk using an
    expanded band search.  }
\label{fig: disk the obscure}
\end{figure}
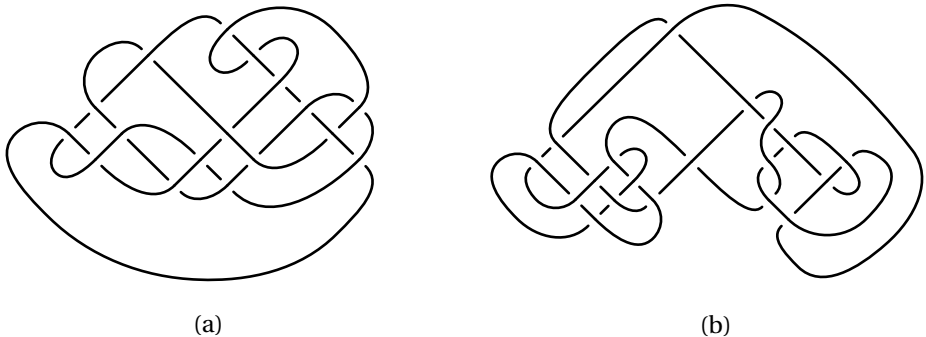

For Section~\ref{sec: friends}, we also identified knots we suspected
were smooth slice because they are $0$-friends with a ribbon knot
(this is a priori weaker than having the same knot trace).  For these
knots, we also eventually found ribbon disks for all
but $18nh_{00000601}$.  In particular, we found ribbon disks for
$19nh_{000077044}$, $19nh_{000187109}$, and $19nh_{003361975}$.  The
knot $19nh_{000077044}$ is depicted in Figure \ref{fig: disk the
  obscure}(b).

  
\subsection{Fusion number}

\begin{table}
  \centering
  \begin{tabular}{c r l} 
    \toprule
    \# bands & \# knots  & \multicolumn{1}{c}{\%}\\
    \midrule
    1 & \NumKnotsNeedingOneBand & 76.5 \\ 
    2 & \NumKnotsNeedingTwoBands  & 23.4 \\
    3 & \NumKnotsNeedingThreeBands & \hphantom{2}0.14 \\
    4 & \NumKnotsNeedingFourBands &  \hphantom{2}0.0045 \\
    \bottomrule
  \end{tabular}
  \caption{The number of bands in the ribbon disks we used in
    Theorem~\ref{thm: ribbon}. It is an upper bound on the fusion
    number for each knot.}
  \label{tab: fusion}
\end{table}

The \emph{fusion number} of a ribbon knot is the minimal number of
1-handles needed in a ribbon disk, or equivalently the number of band
moves needed to produce the unlink. In our search, we looked for ribbon
disks needing up to four bands; the vast majority of the ribbon disks
we found only needed one band, with the vast majority of the remainder
only needing 2 bands, and so on; see Table~\ref{tab: fusion}.  Only
$\NumKnotsNeedingFourBands$ of the ribbon disks we found used four
bands, so we did not search further than that.

There has been much research exploring bounds on the fusion number by
computable knot invariants, including the torsion order of knot Floer
homology \cite{JuhaszMillerZemke2020}, and the $X$-torsion order in
Khovanov homology \cite{Sarkar2020, Schuetz2024}. Using KnotJob
\cite{KnotJob}, we computed the $X$-torsion order for all the knots in
\PS, for the fields $\Q$, $\F_3$, and $\F_5$. For our ribbon knots,
$\RibbonKnotsWithXTorsionOrderTwo$ had $X$-torsion order two, and the
rest had $X$-torsion order one.  The ribbon knots with $X$-torsion
order two must have fusion number at least two. We found 2-band
ribbon disks for all but one of them, namely $18nh_{00098373}$, and
hence the rest have fusion number exactly two.

\subsection{Related work}
\label{sec: ribbon comp}

As mentioned in Section~\ref{sec: prior}, Owens and Swenton studied
the 1.2~billion nontrivial prime alternating knots with at most 21
crossings, finding some $\num{664633}$ ribbon disks
\cite{OwensSwenton2023}. They heavily exploited that when such a knot
is smoothly slice, it has a ``bifactorizable diagram'', a notion that
is defined in terms of Goeritz matrices. They then focused on
``algorithmic bands'' which are short (length at most~3 in our
language with at most 1 half-twist) and yield another bifactorizable
diagram; in hard cases, they allowed slightly longer ``escape bands''.
For prime alternating knots with at most 19 crossings, in the end we
completely determined which are ribbon in Theorem~\ref{thm: alt} using
our generic approach; this includes the 86 knots that were
unresolved by \cite{OwensSwenton2023}.  However, we had to really
struggle to find ribbon disks for some knots (e.g.~$19ah_{29063095}$,
$19ah_{40083815}$, $19ah_{02944101}$, and $19ah_{39877460}$) that are
straightforward with the techniques of \cite{OwensSwenton2023}.

The paper \cite{GukovEtAl2023} uses machine learning techniques to
search for ribbon disks.  There, the authors used the same class of
diagrammatic bands as we do.  However, rather than try all such bands
up to some fixed complexity, they use reinforcement learning or
Bayesian optimization to generate a collection of bands to add.  For
knots with at most 14 crossings, \cite{GukovEtAl2023} arrived at the
same list of \num{1705} ribbon knots as in Theorem~\ref{thm: real
  main}.

\subsection{Ribbon certificates}
\label{sec: cert}

For each ribbon knot in Theorem~\ref{thm: ribbon}, we saved a ``ribbon
certificate'' in \cite{CodeAndData}.  For such a knot $K$, this is
sequence $(D_1, b_1, D_2, b_2, \ldots, b_k, D_{k+1})$ where:
\begin{enumerate}
\item $D_1$ is a diagram for $K$.
\item Each $D_i$ is a link diagram and $b_i$ is a diagrammatic band
  for $D_i$.  Here, $D_i$ is recorded by a PD code and $b_i$ is encoded
  as in the caption to Figure~\ref{fig: diag bands}.
  
\item $D_{i+1}$ is a diagram for the link resulting from adding 
  $b_i$ to $D_i$.
  
\item $D_{k + 1}$ is the unlink or one of the links from
  Section~\ref{sec: cache}.
\end{enumerate}
To check a certificate, one needs to verify certain links are
isotopic.  For example, $D_1$ may not be the standard diagram for $K$
because of shaking as described in Section~\ref{sec: shaking}.
Similarly, $D_{i+1}$ is typically not the diagram initially
generated by adding $b_i$ to $D_i$; rather it, is a simplified and
possibly shaken version of the same.  Most of the time these links
have hyperbolic exteriors, in which case SnapPy can easily check the
needed link equivalences.

%% file: figures/diagrammatics.tex
\begin{tikzpicture}[nmdstd, scale=0.52, line width=1.1pt, font=\small]
\definecolor{dualgraph}{cmyk}{0.5,0.0,1.0,0.0}
\colorlet{arc}{-dualgraph}
\newcommand{\noderadius}{0.12}
\tikzset{%
  arc in dual/.style={color=arc, line width=1.75}
}

%
%

\newcommand{\outsidepath}{%
  (F8) .. controls (0, -2.5) and (-1, 0) .. (-1,3)
       .. controls (-1, 5.5) and (0, 8) .. (M5)
     }

\newcommand{\stevedoreknot}{%
  \draw[add coordinate={0.5}{M1}]
     (3.67, 2.04) .. controls (2.86, 2.04) and (2.04, 2.32) .. 
     (2.04, 3.02) .. controls (2.04, 3.82) and (3.06, 3.99) .. (4.00, 3.99);
  \draw[add coordinate={0.5}{M2}]
     (4.00, 3.99) .. controls (5.02, 3.99) and (5.95, 3.33) .. (5.95, 2.37);
  \draw[add coordinate={0.4}{M3}]
     (5.95, 1.71) .. controls (5.95, 0.90) and (6.23, 0.08) .. 
     (6.93, 0.08) .. controls (7.74, 0.08) and (7.91, 1.10) .. (7.91, 2.04);
  \draw[add coordinate={0.6}{M4}]
     (7.91, 2.04) .. controls (7.91, 3.23) and (7.91, 4.43) .. (7.91, 5.62);
  \draw[add coordinate={0.7}{M5}]
     (7.91, 6.28) .. controls (7.91, 7.25) and (6.98, 7.91) .. 
     (5.95, 7.91) .. controls (4.87, 7.91) and (4.00, 7.03) .. (4.00, 5.95);
  \draw[add coordinate={0.5}{M6}]
     (4.00, 5.95) .. controls (4.00, 5.41) and (4.00, 4.87) .. (4.00, 4.32);
  \draw[add coordinate={0.45}{M7}]
     (4.00, 3.66) .. controls (4.00, 3.12) and (4.00, 2.58) .. (4.00, 2.04);
  \draw[add coordinate={0.2}{M8}]
     (4.00, 2.04) .. controls (4.00, 0.95) and (3.12, 0.08) .. 
     (2.04, 0.08) .. controls (0.73, 0.08) and (0.08, 1.54) .. 
     (0.08, 3.02) .. controls (0.08, 4.77) and (1.79, 5.95) .. (3.67, 5.95);
  \draw[add coordinate={0.5}{M9}]
     (4.33, 5.95) .. controls (5.52, 5.95) and (6.72, 5.95) .. (7.91, 5.95);
  \draw[add coordinate={0.65}{M10}]
     (7.91, 5.95) .. controls (8.99, 5.95) and (9.87, 5.08) .. 
     (9.87, 3.99) .. controls (9.87, 2.97) and (9.21, 2.04) .. (8.24, 2.04);
  \draw[add coordinate={0.5}{M11}]
     (7.58, 2.04) .. controls (7.04, 2.04) and (6.50, 2.04) .. (5.95, 2.04);
  \draw[add coordinate={0.5}{M12}]
     (5.95, 2.04) .. controls (5.41, 2.04) and (4.87, 2.04) .. (4.33,
     2.04);

  \coordinate (F1) at (1, 3);
  \coordinate (F2) at (3, 3);
  \coordinate (F3) at (7, 1);
  \coordinate (F4) at (4.9, 2.9);
  \coordinate (F5) at (6.5, 4.4);
  \coordinate (F6) at (6, 7);
  \coordinate (F7) at (9, 4);
  \coordinate (F8) at (5, -1.4);
}

\newcommand{\dualedges}{
    \draw (F1) -- (M1);
    \draw (F1) .. controls (2, 5) and (3, 5) .. (M6);
    \draw (F1) .. controls (1, 1.5) and (2, 0.5) .. (M8);
    \draw (F2) -- (M1);
    \draw (F2) -- (M7);
    \draw (F3) -- (M11);
    \draw (F3) -- (M3);
    \draw (F4) -- (M2);
    \draw (F4) -- (M12);
    \draw (F4) -- (M7);
    \draw (F5) -- (M11);
    \draw (F5) -- (M9);
    \draw (F5) -- (M2);
    \draw (F5) -- (M4);
    \draw (F5) -- (M6);
    \draw (F7) -- (M4);
    \draw (F7) -- (M10);
    \draw (F6) -- (M9);
    \draw (F6) -- (M5);
    \draw \outsidepath;
    \draw (F8) -- (M8);
    \draw (F8) -- (M12);
    \draw (F8) -- (M3);
    \draw (F8) .. controls (6.5, -2.5) and (9, -2) .. (10.2, 0)
              .. controls (11, 1.5) and (10.25, 2.25) .. (M10);
  }

%
%
  
\begin{scope}
  \begin{scope}[color=black]
    \stevedoreknot
  \end{scope}
  \begin{scope}[color=dualgraph]
    \dualedges
    \foreach \i in {1,...,12} {\filldraw (M\i)
      circle [radius=\noderadius];}
    \foreach \i in {1,...,8} {\filldraw (F\i)
      circle [radius=\noderadius];}
  \end{scope}
  \node at (9, 6.7) {$D$};
  \node at (10, -1.7) {$G_D$};
  \node[below] at (5, -2.5) {(a)};
  \node[below] at (5, -16) {(c)};
\end{scope}

%
%

\begin{scope}[shift={(15, 0)}]
  \begin{scope}[color=black]
    \stevedoreknot
  \end{scope}
  \begin{scope}[color=dualgraph]
    \dualedges
    \foreach \i in {1,...,12} {\filldraw (M\i)
      circle [radius=\noderadius];}
    \foreach \i in {1,...,8} {\filldraw (F\i)
      circle [radius=\noderadius];}
  \end{scope}
  \begin{scope}[arc in dual]
    \draw (M3) -- (F3) -- (M11) -- (F5) -- (M9) -- (F6) -- (M5);
    \draw \outsidepath;
    \draw (F8) -- (M8);
    \foreach \C in {M3, F3, M11, F5, M9, F6, M5, F8, M8} {
      \filldraw (\C) circle [radius=\noderadius];
    }
    \draw (M1) -- (F2) -- (M7) -- (F4) -- (M2);
    \foreach \C in {M1, F2, M7, F4, M2} {
      \filldraw (\C) circle [radius=\noderadius];
    }
  \end{scope}
  \node[below] at (5, -2.5) {(b)};
  \node[below] at (5, -16) {(d)};
\end{scope}

%
%

\begin{scope}[shift={(0, -13.5)}]
  \begin{scope}
  \draw[add coordinate={0.5}{M1}]
     (3.67, 2.04) .. controls (2.86, 2.04) and (2.04, 2.32) .. 
     (2.04, 3.02) .. controls (2.04, 3.82) and (3.06, 3.99) .. (4.00, 3.99);
  \draw[add coordinate={0.5}{M2}]
     (4.00, 3.99) .. controls (5.02, 3.99) and (5.95, 3.33) .. (5.95, 2.37);
  \draw[add coordinate={0.4}{M3}]
     (5.95, 1.71) .. controls (5.95, 0.90) and (6.23, 0.08) .. 
     (6.93, 0.08) .. controls (7.74, 0.08) and (7.91, 1.10) .. (7.91, 2.04);
  \draw[add coordinate={0.6}{M4}, mid arrow=0.6]
     (7.91, 2.04) .. controls (7.91, 3.23) and (7.91, 4.43) .. (7.91, 5.62);
  \draw[add coordinate={0.7}{M5}]
     (7.91, 6.28) .. controls (7.91, 7.25) and (6.98, 7.91) .. 
     (5.95, 7.91) .. controls (4.87, 7.91) and (4.00, 7.03) .. (4.00, 5.95);
  \draw[add coordinate={0.5}{M6}]
     (4.00, 5.95) .. controls (4.00, 5.41) and (4.00, 4.87) .. (4.00, 4.32);
  \draw[add coordinate={0.45}{M7}]
     (4.00, 3.66) .. controls (4.00, 3.12) and (4.00, 2.58) .. (4.00, 2.04);
  \draw[add coordinate={0.2}{M8}, mid arrow=0.56]
     (4.00, 2.04) .. controls (4.00, 0.95) and (3.12, 0.08) .. 
     (2.04, 0.08) .. controls (0.73, 0.08) and (0.08, 1.54) .. 
     (0.08, 3.02) .. controls (0.08, 4.77) and (1.79, 5.95) .. (3.67, 5.95);
  \draw[add coordinate={0.5}{M9}]
     (4.33, 5.95) .. controls (5.52, 5.95) and (6.72, 5.95) .. (7.91, 5.95);
  \draw[add coordinate={0.65}{M10}, mid arrow=0.57]
     (7.91, 5.95) .. controls (8.99, 5.95) and (9.87, 5.08) .. 
     (9.87, 3.99) .. controls (9.87, 2.97) and (9.21, 2.04) .. (8.24, 2.04);
  \draw[add coordinate={0.6}{M11}]
     (7.58, 2.04) .. controls (7.04, 2.04) and (6.50, 2.04) .. (5.95, 2.04);
  \draw[add coordinate={0.5}{M12}]
     (5.95, 2.04) .. controls (5.41, 2.04) and (4.87, 2.04) .. (4.33,
     2.04);
     
  \coordinate (F1) at (1, 3);
  \coordinate (F2) at (3, 3);
  \coordinate (F3) at (7, 1);
  \coordinate (F4) at (4.9, 2.9);
  \coordinate (F5) at (6.5, 4.4);
  \coordinate (F6) at (6, 7);
  \coordinate (F7) at (9, 4);
  \coordinate (F8) at (5, -1.4);   
  \end{scope}

  \coordinate (C3) at (7.9, 2.05);
  \coordinate (C4) at (7.9, 5.95);
  \coordinate (C1) at (4.00, 2.00);
  \begin{scope}[dotted, color=blue!50, line width=1]
    \draw[radius=0.9] (C3) circle;
    \draw[radius=1.1] (C4) circle;
    \draw[radius=1.1] (C1) circle;
  \end{scope}
  \begin{scope}[font=\tiny]
    \node at ($(C1) + (23:0.8)$) {0};
    \node at ($(C1) + (113:0.76)$) {1};
    \node at ($(C1) + (200:0.8)$) {2};
    \node at ($(C1) + (-86:0.8)$) {3};

    \node at ($(C3) + (23:0.67)$) {0};
    \node at ($(C3) + (113:0.6)$) {1};
    \node at ($(C3) + (210:0.6)$) {2};
    \node at ($(C3) + (-65:0.6)$) {3};

    \node at ($(C4) + (10:0.85)$) {1};
    \node at ($(C4) + (113:0.76)$) {2};
    \node at ($(C4) + (200:0.8)$) {3};
    \node at ($(C4) + (-70:0.8)$) {0};

  \end{scope}

  \begin{scope}[font=\small]
    \node at ($(C1) + (-40:1.6)$) {$C_1$};
    \node at ($(C3) + (-40:1.6)$) {$C_3$};
    \node at ($(C4) + (40:1.7)$) {$C_4$};
  \end{scope}

  \begin{scope}[color=arc, line width=1.5pt]
    \draw[mid arrow=0.6]  (M11) to[bend left=10] (M3);
    \draw[mid arrow=0.55] (M9) to[bend left] (M11);
    \draw[mid arrow=0.65] (M5) to[bend left] (M9);
    \draw[mid arrow=0.7]  (M8) .. controls (3, -2.5) and (-1, -3) ..  (-1,2.5)
                               .. controls (-1, 4)  and (0, 8) .. (M5);
    \foreach \C in {M3, M11, M9, M5, M8} {
      \filldraw (\C) circle [radius=\noderadius];
    }
  \end{scope}
\end{scope}

%
%

\begin{scope}[shift={(12.3, -16)}, scale=1.28]
  \begin{scope}
    \draw (5.27, 5.10) .. controls (4.57, 5.10) and (3.82, 4.94) .. 
          (3.82, 4.34) .. controls (3.82, 3.84) and (4.37, 3.59) .. (4.93, 3.59);
    \draw (5.62, 3.59) .. controls (6.03, 3.59) and (6.45, 3.59) .. (6.86, 3.59);
    \draw (6.86, 3.59) .. controls (7.03, 3.59) and (7.19, 3.59) .. (7.35, 3.59);
    \draw (7.35, 3.59) .. controls (7.53, 3.59) and (7.70, 3.59) .. (7.88, 3.59);
    \draw (7.88, 3.59) .. controls (7.99, 3.59) and (8.10, 3.59) .. (8.21, 3.59);
    \draw (8.69, 3.59) .. controls (9.41, 3.59) and (9.87, 4.32) .. 
          (9.87, 5.10) .. controls (9.87, 5.95) and (9.17, 6.62) .. (8.32, 6.62);
    \draw (8.32, 6.62) .. controls (8.02, 6.62) and (7.71, 6.62) .. (7.40, 6.62);
    \draw (7.40, 6.62) .. controls (7.22, 6.62) and (7.04, 6.62) .. (6.85, 6.62);
    \draw (6.85, 6.62) .. controls (6.43, 6.62) and (6.01, 6.62) .. (5.59, 6.62);
    \draw (4.91, 6.62) .. controls (3.52, 6.62) and (2.30, 5.67) .. 
          (2.30, 4.35) .. controls (2.30, 3.32) and (2.30, 2.08) .. 
          (3.00, 2.09) .. controls (3.35, 2.09) and (3.58, 1.73) .. 
          (3.47, 1.38) .. controls (3.34, 0.95) and (2.95, 0.66) .. 
          (2.50, 0.66) .. controls (2.10, 0.66) and (1.72, 0.90) .. 
          (1.66, 1.28) .. controls (1.57, 1.82) and (1.26, 2.28) .. (0.93, 2.71);
    \draw (0.52, 3.25) .. controls (0.08, 3.82) and (0.11, 4.61) .. (0.55, 5.19);
    \draw (0.55, 5.19) .. controls (0.83, 5.56) and (1.05, 5.99) .. 
          (1.05, 6.46) .. controls (1.04, 6.86) and (1.32, 7.20) .. 
          (1.69, 7.39) .. controls (2.29, 7.70) and (2.96, 7.81) .. 
          (3.62, 7.91) .. controls (4.29, 8.01) and (4.99, 8.07) .. (5.61, 7.79);
    \draw (5.61, 7.79) .. controls (6.06, 7.59) and (6.56, 7.36) .. (6.73, 6.94);
    \draw (6.98, 6.30) .. controls (7.15, 5.86) and (7.25, 4.73) .. (7.32, 3.93);
    \draw (7.38, 3.25) .. controls (7.41, 2.96) and (7.43, 2.68) .. (7.42, 2.39);
    \draw (7.42, 2.39) .. controls (7.41, 2.21) and (7.31, 2.04) ..
          (7.14, 2.03)  .. controls (7.02, 2.03) and (6.88, 2.05) .. 
          (6.88, 2.75) .. controls (6.88, 2.80) and (6.87, 2.88) ..
          (6.87, 3.25);
    \draw (6.86, 3.93) .. controls (6.85, 4.66) and (6.08, 5.10) .. (5.27, 5.10);
  \end{scope}
  \begin{scope}
    \draw (7.90, 3.25) .. controls (7.92, 2.96) and (7.94, 2.67) .. 
          (7.93, 2.38) .. controls (7.92, 2.23) and (8.00, 2.08) .. 
          (8.13, 2.08) .. controls (8.36, 2.08) and (8.45, 2.94) .. (8.45, 3.59);
    \draw (8.45, 3.59) .. controls (8.45, 4.48) and (8.4, 5.38) .. (8.4, 6.35);
    \draw (5.25, 6.62) to[bend left=18] (5.6, 7.60);
    \draw (6.4, 8.15)  to[bend left=4]  (6.80, 8.20);
    \draw (6.80, 8.20).. controls (7.90, 8.20) and (8.4, 7.40) .. (8.40, 6.90);
    \draw (5.25, 6.62) .. controls (5.26, 6.23) and (5.27, 5.83) .. (5.27, 5.44);
    \draw (5.27, 4.76) .. controls (5.27, 4.37) and (5.27, 3.98) .. (5.27, 3.59);
    \draw (5.27, 3.59) .. controls (5.28, 3.29) and (5.28, 3.00) .. 
          (5.28, 2.70) .. controls (5.29, 2.39) and (5.00, 2.17) .. 
          (4.67, 2.15) .. controls (4.20, 2.14) and (3.95, 1.64) .. 
          (3.85, 1.15) .. controls (3.71, 0.50) and (3.07, 0.14) .. 
          (2.39, 0.11) .. controls (1.67, 0.08) and (1.00, 0.51) .. 
          (0.77, 1.19) .. controls (0.58, 1.78) and (0.45, 2.42) .. (0.72, 2.98);
    \draw (0.72, 2.98) .. controls (1.04, 3.63) and (1.21, 4.39) .. (0.76, 4.93);
    \draw (0.33, 5.45) .. controls (0.17, 5.65) and (0.19, 6.00) .. 
          (0.22, 6.30) .. controls (0.25, 6.75) and (0.44, 7.17) .. 
          (0.73, 7.52) .. controls (1.19, 8.09) and (2.18, 8.24) .. 
          (3.02, 8.37) .. controls (4.01, 8.53) and (5.04, 8.57) .. (5.95, 8.13);
    \draw (5.95, 8.13) .. controls (6.50, 7.87) and (6.99, 7.48) .. (7.26, 6.93);
    \draw (7.55, 6.31) .. controls (7.70, 6.01) and (7.79, 4.73) .. (7.85, 3.93);
  \end{scope}
\end{scope}

\end{tikzpicture}

%% file: figures/19nh_051162051.tex
  \begin{scope}[scale=0.65]
    \fill[fill=red!20]
         (5.43, 3.68) .. controls (5.43, 3.19) and (5.03, 2.79) .. 
         (4.54, 2.79) .. controls (4.09, 2.79) and (3.65, 3.00) ..
         (3.65, 3.40) -- 
         (3.65, 3.68) .. controls (4.15, 3.68) and (4.64, 3.68) ..
         (5.14, 3.68) -- cycle;

    \fill[blue!20]
         (6.59, 1.92) .. controls (6.99, 1.91) and (7.20, 1.47) ..
         (7.20, 1.02) --
         (6.92, 1.02) .. controls (6.51, 1.02) and (6.31, 1.47) ..
         (6.31, 1.92) -- cycle;
    
    \draw[mid arrow=0.53] (0.70, 7.24) .. controls (0.16, 7.24) and (0.09, 7.94) .. 
          (0.09, 8.57) .. controls (0.09, 9.57) and (1.27, 9.95) .. 
          (2.40, 10.01) .. controls (3.83, 10.07) and (5.27, 10.04) .. 
          (6.69, 9.85) .. controls (7.63, 9.72) and (8.67, 9.48) .. 
          (8.81, 8.60) .. controls (8.88, 8.15) and (8.97, 7.63) .. (8.97, 7.24);
    \draw[mid arrow=0.55] (8.97, 7.24) .. controls (8.97, 5.26) and (8.98, 3.28) .. (8.98, 1.30);
    \draw (8.98, 0.73) .. controls (8.99, 0.42) and (8.82, 0.12) .. 
          (8.54, 0.12) .. controls (8.17, 0.12) and (8.09, 0.59) .. (8.09, 1.02);
    \draw[mid arrow=0.5] (8.09, 1.02) .. controls (8.09, 2.35) and (7.50, 3.68) .. (6.31, 3.68);
    \draw (6.31, 3.68) .. controls (6.10, 3.68) and (5.90, 3.68) .. (5.69, 3.68);
    \draw[mid arrow=0.55] (5.14, 3.68) .. controls (4.64, 3.68) and (4.15, 3.68) .. (3.65, 3.68);
    \draw (3.65, 3.68) .. controls (3.20, 3.68) and (2.76, 3.89) .. (2.76, 4.29);
    \draw (2.76, 4.84) .. controls (2.76, 5.05) and (2.76, 5.25) .. (2.76, 5.46);
    \draw[mid arrow=0.85] (2.76, 5.46) .. controls (2.76, 5.67) and (2.76, 5.87) .. (2.76, 6.08);
    \draw (2.76, 6.62) .. controls (2.76, 6.83) and (2.76, 7.03) .. (2.76, 7.24);
    \draw[mid arrow=0.6] (2.76, 7.24) .. controls (2.76, 7.83) and (3.42, 8.12) .. 
          (4.09, 8.12) .. controls (4.72, 8.12) and (5.43, 8.07) .. (5.43, 7.53);
          \draw[mid arrow=0.85] (5.43, 6.96) .. controls (5.43, 5.87) and (5.43, 4.78) .. (5.43, 3.68);
          
    \draw (5.43, 3.68) .. controls (5.43, 3.19) and (5.03, 2.79) .. 
    (4.54, 2.79) .. controls (4.09, 2.79) and (3.65, 3.00) .. (3.65, 3.40);

    \draw (3.65, 3.95) .. controls (3.65, 4.16) and (3.65, 4.37) .. (3.65, 4.58);
    \draw[mid arrow=0.9] (3.65, 4.58) .. controls (3.65, 4.78) and (3.65, 4.99) .. (3.65, 5.19);
    \draw (3.65, 5.74) .. controls (3.65, 6.14) and (3.21, 6.35) .. (2.76, 6.35);
    \draw[mid arrow=0.6] (2.76, 6.35) .. controls (2.17, 6.35) and (1.58, 6.35) .. (0.99, 6.35);
    \draw (0.99, 6.35) .. controls (0.56, 6.35) and (0.09, 6.27) .. 
          (0.09, 5.90) .. controls (0.09, 5.62) and (0.39, 5.46) .. (0.71, 5.46);
    \draw[mid arrow=0.6] (1.27, 5.46) .. controls (1.67, 5.46) and (2.07, 5.46) .. (2.47, 5.46);
    \draw (3.03, 5.46) .. controls (3.23, 5.46) and (3.44, 5.46) .. (3.65, 5.46);
    \draw (3.65, 5.46) .. controls (4.07, 5.46) and (4.53, 5.38) .. 
          (4.53, 5.02) .. controls (4.53, 4.74) and (4.24, 4.58) .. (3.93, 4.58);
    \draw (3.38, 4.58) .. controls (3.17, 4.58) and (2.97, 4.58) .. (2.76, 4.58);
    \draw[mid arrow=0.5] (2.76, 4.58) .. controls (2.20, 4.58) and (1.96, 3.92) .. 
          (2.04, 3.27) .. controls (2.20, 1.96) and (4.29, 1.94) .. (6.03, 1.92);
    \draw (6.59, 1.92) .. controls (6.99, 1.91) and (7.20, 1.47) .. (7.20, 1.02);
    \draw[mid arrow=0.1, mid arrow=0.9] (7.20, 1.02) .. controls (7.20, 0.38) and (6.45, 0.11) .. 
          (5.72, 0.10) .. controls (4.69, 0.09) and (3.66, 0.09) .. 
          (2.62, 0.10) .. controls (1.00, 0.11) and (1.00, 3.05) .. (0.99, 5.46);
    \draw (0.99, 5.46) .. controls (0.99, 5.67) and (0.99, 5.87) .. (0.99, 6.08);
    \draw (0.99, 6.62) .. controls (0.99, 6.83) and (0.99, 7.03) .. (0.99, 7.24);
    \draw[mid arrow=0.4] (0.99, 7.24) .. controls (0.98, 8.43) and (2.32, 9.02) .. 
          (3.65, 9.02) .. controls (4.95, 9.02) and (6.32, 8.65) .. (6.32, 7.53);
    \draw[mid arrow=0.5] (6.32, 6.96) .. controls (6.31, 5.96) and (6.31, 4.96) .. (6.31, 3.97);
    \draw[mid arrow=0.75] (6.31, 3.40) .. controls (6.31, 2.91) and (6.31, 2.41) .. (6.31, 1.92);
    \draw (6.31, 1.92) .. controls (6.31, 1.47) and (6.51, 1.02) .. (6.92, 1.02);
    \draw (7.47, 1.02) .. controls (7.59, 1.02) and (7.71, 1.02) .. (7.83, 1.02);
    \draw (8.36, 1.02) .. controls (8.57, 1.02) and (8.78, 1.02) .. (8.98, 1.02);
    \draw[mid arrow=0.55](8.98, 1.02) .. controls (9.87, 1.02) and (9.87, 2.74) .. 
          (9.87, 4.13) .. controls (9.87, 5.48) and (9.87, 7.24) .. (9.25, 7.24);
    \draw (8.68, 7.24) .. controls (7.89, 7.24) and (7.11, 7.24) .. (6.32, 7.24);
    \draw[mid arrow=0.67] (6.32, 7.24) .. controls (6.02, 7.24) and (5.72, 7.24) .. (5.43, 7.24);
    \draw (5.43, 7.24) .. controls (4.63, 7.24) and (3.84, 7.24) .. (3.04, 7.24);
    \draw[mid arrow=0.65] (2.47, 7.24) .. controls (2.07, 7.24) and (1.67, 7.24) .. (1.27, 7.24);
  \end{scope}

  %
  %
  \begin{scope}[shift={(8, 0)}, scale=0.65,]
  \definecolor{linkcolor0}{rgb}{0.85, 0.15, 0.15}
  \definecolor{linkcolor1}{rgb}{0.15, 0.15, 0.85}
  \definecolor{linkcolor2}{rgb}{0.15, 0.85, 0.15}
  \begin{scope}[color=linkcolor0]
    \draw (0.70, 7.24) .. controls (0.16, 7.24) and (0.09, 7.94) .. 
          (0.09, 8.57) .. controls (0.09, 9.57) and (1.27, 9.95) .. 
          (2.40, 10.01) .. controls (3.83, 10.07) and (5.27, 10.04) .. 
          (6.69, 9.85) .. controls (7.63, 9.72) and (8.67, 9.48) .. 
          (8.81, 8.60) .. controls (8.88, 8.15) and (8.97, 7.63) .. (8.97, 7.24);
    \draw (8.97, 7.24) .. controls (8.97, 5.26) and (8.98, 3.29) .. (8.98, 1.31);
    \draw (8.98, 0.74) .. controls (8.99, 0.43) and (8.82, 0.12) .. 
          (8.54, 0.12) .. controls (8.17, 0.12) and (8.09, 0.60) .. (8.09, 1.03);
    \draw (8.09, 1.03) .. controls (8.09, 2.35) and (7.46, 3.64) .. (6.32, 3.55);
    \draw (6.32, 3.55) .. controls (6.05, 3.53) and (5.82, 3.35) .. 
          (5.62, 3.17) .. controls (5.31, 2.88) and (4.87, 2.81) .. 
          (4.44, 2.80) .. controls (3.99, 2.80) and (3.65, 3.21) .. (3.65, 3.68);
    \draw (3.65, 3.68) .. controls (3.65, 3.98) and (3.65, 4.28) .. (3.65, 4.58);
    \draw (3.65, 4.58) .. controls (3.65, 4.78) and (3.65, 4.99) .. (3.65, 5.19);
    \draw (3.65, 5.74) .. controls (3.65, 6.14) and (3.22, 6.35) .. (2.78, 6.35);
    \draw (2.78, 6.35) .. controls (2.18, 6.35) and (1.59, 6.35) .. (0.99, 6.35);
    \draw (0.99, 6.35) .. controls (0.56, 6.35) and (0.09, 6.27) .. 
          (0.09, 5.90) .. controls (0.09, 5.62) and (0.39, 5.46) .. (0.71, 5.46);
    \draw (1.27, 5.46) .. controls (1.69, 5.46) and (2.10, 5.46) .. (2.51, 5.46);
    \draw (3.05, 5.46) .. controls (3.25, 5.46) and (3.45, 5.46) .. (3.65, 5.46);
    \draw (3.65, 5.46) .. controls (4.07, 5.46) and (4.53, 5.38) .. 
          (4.53, 5.02) .. controls (4.53, 4.74) and (4.24, 4.58) .. (3.93, 4.58);
    \draw (3.40, 4.58) .. controls (3.20, 4.58) and (3.00, 4.58) .. (2.81, 4.58);
    \draw (2.81, 4.58) .. controls (2.24, 4.58) and (1.96, 3.92) .. 
          (2.04, 3.27) .. controls (2.14, 2.42) and (3.01, 1.95) .. 
          (3.91, 1.94) .. controls (4.66, 1.93) and (5.62, 1.92) .. 
          (5.94, 1.48) .. controls (6.22, 1.11) and (6.71, 1.04) .. (7.18, 1.04);
    \draw (7.18, 1.04) .. controls (7.39, 1.04) and (7.61, 1.04) .. (7.82, 1.03);
    \draw (8.36, 1.03) .. controls (8.57, 1.03) and (8.78, 1.03) .. (8.98, 1.02);
    \draw (8.98, 1.02) .. controls (9.87, 1.02) and (9.87, 2.73) .. 
          (9.87, 4.13) .. controls (9.87, 5.48) and (9.87, 7.24) .. (9.25, 7.24);
    \draw (8.68, 7.24) .. controls (7.89, 7.24) and (7.11, 7.24) .. (6.32, 7.24);
    \draw (6.32, 7.24) .. controls (6.02, 7.24) and (5.73, 7.24) .. (5.43, 7.24);
    \draw (5.43, 7.24) .. controls (4.64, 7.24) and (3.85, 7.24) .. (3.05, 7.24);
    \draw (2.49, 7.24) .. controls (2.08, 7.24) and (1.68, 7.24) .. (1.27, 7.24);
  \end{scope}
  \begin{scope}[color=linkcolor1]
    \draw (3.37, 3.68) .. controls (3.05, 3.68) and (2.82, 3.96) .. (2.81, 4.29);
    \draw (2.80, 4.84) .. controls (2.80, 5.05) and (2.80, 5.25) .. (2.80, 5.46);
    \draw (2.80, 5.46) .. controls (2.79, 5.67) and (2.79, 5.87) .. (2.79, 6.08);
    \draw (2.78, 6.62) .. controls (2.78, 6.83) and (2.77, 7.03) .. (2.77, 7.24);
    \draw (2.77, 7.24) .. controls (2.76, 7.83) and (3.43, 8.12) .. 
          (4.10, 8.11) .. controls (4.72, 8.10) and (5.44, 8.07) .. (5.43, 7.53);
    \draw (5.43, 6.96) .. controls (5.43, 5.42) and (5.26, 3.70) .. (3.93, 3.69);
  \end{scope}
  \begin{scope}[color=linkcolor2]
    \draw (6.32, 3.27) .. controls (6.32, 2.91) and (6.48, 2.57) .. 
          (6.74, 2.31) .. controls (7.00, 2.05) and (7.16, 1.70) .. (7.17, 1.32);
    \draw (7.19, 0.76) .. controls (7.20, 0.16) and (6.42, 0.11) .. 
          (5.72, 0.10) .. controls (4.69, 0.09) and (3.66, 0.09) .. 
          (2.62, 0.10) .. controls (1.00, 0.11) and (1.00, 3.05) .. (0.99, 5.46);
    \draw (0.99, 5.46) .. controls (0.99, 5.67) and (0.99, 5.87) .. (0.99, 6.08);
    \draw (0.99, 6.62) .. controls (0.99, 6.83) and (0.99, 7.03) .. (0.99, 7.24);
    \draw (0.99, 7.24) .. controls (0.98, 8.43) and (2.32, 9.02) .. 
          (3.65, 9.02) .. controls (4.95, 9.02) and (6.32, 8.65) .. (6.32, 7.53);
    \draw (6.32, 6.96) .. controls (6.32, 5.92) and (6.32, 4.88) .. (6.32, 3.83);
  \end{scope}
  \end{scope}

%% file: figures/obscure_ribbon.tex
  \begin{scope}[rotate=45, scale=0.5]
    \draw (2.24, 6.36) .. controls (2.24, 5.42) and (2.53, 4.45) .. (3.33, 4.45);
    \draw (3.33, 4.45) .. controls (3.69, 4.45) and (4.06, 4.45) .. (4.42, 4.45);
    \draw (4.42, 4.45) .. controls (4.69, 4.45) and (4.97, 4.45) .. (5.24, 4.45);
    \draw (5.77, 4.45) .. controls (6.41, 4.45) and (7.05, 4.45) .. (7.68, 4.45);
    \draw (7.68, 4.45) .. controls (8.21, 4.45) and (8.78, 4.54) .. 
          (8.78, 4.99) .. controls (8.78, 5.36) and (8.36, 5.54) .. (7.95, 5.54);
    \draw (7.42, 5.54) .. controls (7.01, 5.54) and (6.60, 5.72) .. 
          (6.60, 6.08) .. controls (6.60, 6.52) and (7.17, 6.62) .. (7.68, 6.62);
    \draw (7.68, 6.62) .. controls (8.89, 6.62) and (9.87, 5.65) .. 
          (9.87, 4.45) .. controls (9.87, 3.40) and (9.67, 2.27) .. (8.78, 2.27);
    \draw (8.78, 2.27) .. controls (8.50, 2.27) and (8.22, 2.27) .. (7.95, 2.27);
    \draw (7.42, 2.27) .. controls (6.49, 2.27) and (5.51, 2.56) .. (5.51, 3.36);
    \draw (5.51, 3.36) .. controls (5.51, 3.72) and (5.51, 4.09) .. (5.51, 4.45);
    \draw (5.51, 4.45) .. controls (5.51, 5.45) and (5.51, 6.45) .. (5.51, 7.45);
    \draw (5.51, 7.98) .. controls (5.51, 8.49) and (4.98, 8.81) .. 
          (4.42, 8.81) .. controls (3.82, 8.81) and (3.33, 8.32) .. (3.33, 7.71);
    \draw (3.33, 7.71) .. controls (3.33, 7.44) and (3.33, 7.16) .. (3.33, 6.89);
    \draw (3.33, 6.36) .. controls (3.33, 5.81) and (3.33, 5.26) .. (3.33, 4.71);
    \draw (3.33, 4.18) .. controls (3.33, 3.67) and (3.86, 3.36) .. (4.42, 3.36);
    \draw (4.42, 3.36) .. controls (4.69, 3.36) and (4.97, 3.36) .. (5.24, 3.36);
    \draw (5.77, 3.36) .. controls (6.41, 3.36) and (7.05, 3.36) .. (7.68, 3.36);
    \draw (7.68, 3.36) .. controls (8.24, 3.36) and (8.78, 3.05) .. (8.78, 2.54);
    \draw (8.78, 2.01) .. controls (8.78, 1.49) and (8.24, 1.18) .. (7.68, 1.18);
    \draw (7.68, 1.18) .. controls (6.08, 1.18) and (4.42, 1.71) .. (4.42, 3.10);
    \draw (4.42, 3.63) .. controls (4.42, 3.81) and (4.42, 4.00) .. (4.42, 4.18);
    \draw (4.42, 4.71) .. controls (4.42, 5.64) and (4.14, 6.62) .. (3.33, 6.62);
    \draw (3.33, 6.62) .. controls (2.97, 6.62) and (2.61, 6.62) .. (2.24, 6.62);
    \draw (2.24, 6.62) .. controls (1.72, 6.62) and (1.15, 6.72) .. 
          (1.15, 7.17) .. controls (1.15, 7.53) and (1.56, 7.71) .. (1.98, 7.71);
    \draw (2.51, 7.71) .. controls (2.69, 7.71) and (2.88, 7.71) .. (3.07, 7.71);
    \draw (3.60, 7.71) .. controls (4.24, 7.71) and (4.87, 7.71) .. (5.51, 7.71);
    \draw (5.51, 7.71) .. controls (6.51, 7.71) and (7.68, 7.71) .. (7.68, 6.89);
    \draw (7.68, 6.36) .. controls (7.68, 6.08) and (7.68, 5.81) .. (7.68, 5.54);
    \draw (7.68, 5.54) .. controls (7.68, 5.26) and (7.68, 4.99) .. (7.68, 4.71);
    \draw (7.68, 4.18) .. controls (7.68, 4.00) and (7.68, 3.81) .. (7.68, 3.63);
    \draw (7.68, 3.10) .. controls (7.68, 2.82) and (7.68, 2.55) .. (7.68, 2.27);
    \draw (7.68, 2.27) .. controls (7.68, 2.00) and (7.68, 1.72) .. (7.68, 1.44);
    \draw (7.68, 0.91) .. controls (7.68, 0.10) and (6.60, 0.09) .. 
          (5.65, 0.09) .. controls (4.17, 0.08) and (2.86, 0.90) .. 
          (1.86, 1.99) .. controls (0.75, 3.18) and (0.09, 4.72) .. 
          (0.08, 6.35) .. controls (0.07, 7.51) and (0.16, 8.81) .. 
          (1.15, 8.81) .. controls (1.76, 8.81) and (2.24, 8.32) .. (2.24, 7.71);
    \draw (2.24, 7.71) .. controls (2.24, 7.44) and (2.24, 7.16) .. (2.24, 6.89);
  \end{scope}

  \node[below] at (0.1, 1) {(a)};
  \node[below] at (7.5, 1) {(b)};

  \begin{scope}[shift={(7.5, -0.7)}, rotate=45, scale=0.5]
    \draw (3.00, 5.88) .. controls (2.65, 5.88) and (2.48, 6.27) .. (2.48, 6.67);
    \draw (2.48, 6.67) .. controls (2.48, 6.86) and (2.48, 7.05) .. (2.48, 7.23);
    \draw (2.48, 7.75) .. controls (2.48, 8.19) and (2.48, 8.63) .. (2.48, 9.07);
    \draw (2.48, 9.07) .. controls (2.48, 9.87) and (3.78, 9.87) .. 
          (4.87, 9.87) .. controls (5.93, 9.87) and (7.27, 9.87) .. (7.27, 9.32);
    \draw (7.27, 8.76) .. controls (7.27, 7.80) and (7.27, 6.84) .. (7.27, 5.88);
    \draw (7.27, 5.88) .. controls (7.27, 5.69) and (7.27, 5.50) .. (7.27, 5.31);
    \draw (7.27, 4.84) .. controls (7.27, 4.65) and (7.27, 4.47) .. (7.27, 4.28);
    \draw (7.27, 4.28) .. controls (7.27, 3.84) and (7.27, 3.40) .. (7.27, 2.96);
    \draw (7.27, 2.40) .. controls (7.27, 2.13) and (7.42, 1.88) .. 
          (7.67, 1.88) .. controls (8.00, 1.88) and (8.07, 2.30) .. (8.07, 2.68);
    \draw (8.07, 2.68) .. controls (8.07, 3.41) and (8.07, 4.28) .. (7.55, 4.28);
    \draw (7.03, 4.28) .. controls (6.93, 4.28) and (6.82, 4.28) .. (6.71, 4.28);
    \draw (6.20, 4.28) .. controls (5.84, 4.28) and (5.67, 3.88) .. (5.67, 3.48);
    \draw (5.67, 3.48) .. controls (5.67, 3.21) and (5.67, 2.95) .. (5.67, 2.68);
    \draw (5.67, 2.68) .. controls (5.67, 1.80) and (6.39, 1.08) .. 
          (7.27, 1.08) .. controls (8.03, 1.08) and (8.87, 1.23) .. 
          (8.87, 1.88) .. controls (8.87, 2.28) and (8.70, 2.68) .. (8.35, 2.68);
    \draw (7.83, 2.68) .. controls (7.64, 2.68) and (7.46, 2.68) .. (7.27, 2.68);
    \draw (7.27, 2.68) .. controls (6.83, 2.68) and (6.39, 2.68) .. (5.95, 2.68);
    \draw (5.39, 2.68) .. controls (4.91, 2.68) and (4.88, 2.04) .. 
          (4.88, 1.48) .. controls (4.88, 0.51) and (6.06, 0.21) .. 
          (7.17, 0.15) .. controls (8.33, 0.08) and (9.52, 0.51) .. 
          (9.67, 1.55) .. controls (9.81, 2.58) and (9.87, 3.63) .. 
          (9.84, 4.67) .. controls (9.82, 5.67) and (9.67, 6.68) .. 
          (9.36, 7.64) .. controls (9.09, 8.52) and (8.21, 9.04) .. (7.27, 9.04);
    \draw (7.27, 9.04) .. controls (5.77, 9.05) and (4.26, 9.06) .. (2.76, 9.06);
    \draw (2.24, 9.07) .. controls (2.13, 9.07) and (2.03, 9.07) .. (1.92, 9.07);
    \draw (1.41, 9.07) .. controls (1.05, 9.07) and (0.88, 8.67) .. 
          (0.88, 8.27) .. controls (0.88, 7.83) and (1.24, 7.48) .. (1.69, 7.48);
    \draw (1.69, 7.48) .. controls (1.95, 7.48) and (2.21, 7.48) .. (2.48, 7.48);
    \draw (2.48, 7.48) .. controls (2.67, 7.48) and (2.85, 7.48) .. (3.04, 7.48);
    \draw (3.56, 7.48) .. controls (3.83, 7.48) and (4.08, 7.32) .. 
          (4.08, 7.07) .. controls (4.08, 6.75) and (3.66, 6.67) .. (3.28, 6.67);
    \draw (3.28, 6.67) .. controls (3.09, 6.67) and (2.91, 6.67) .. (2.72, 6.67);
    \draw (2.24, 6.67) .. controls (2.13, 6.67) and (2.03, 6.67) .. (1.92, 6.67);
    \draw (1.41, 6.67) .. controls (0.62, 6.67) and (0.09, 7.43) .. 
          (0.09, 8.27) .. controls (0.09, 9.03) and (0.23, 9.87) .. 
          (0.89, 9.87) .. controls (1.33, 9.87) and (1.69, 9.51) .. (1.69, 9.07);
    \draw (1.69, 9.07) .. controls (1.69, 8.63) and (1.69, 8.19) .. (1.69, 7.75);
    \draw (1.69, 7.23) .. controls (1.69, 7.05) and (1.69, 6.86) .. (1.69, 6.67);
    \draw (1.69, 6.67) .. controls (1.69, 5.91) and (1.83, 5.07) .. 
          (2.48, 5.07) .. controls (2.93, 5.07) and (3.28, 5.43) .. (3.28, 5.88);
    \draw (3.28, 5.88) .. controls (3.28, 6.06) and (3.28, 6.25) .. (3.28, 6.43);
    \draw (3.28, 6.91) .. controls (3.28, 7.10) and (3.28, 7.29) .. (3.28, 7.48);
    \draw (3.28, 7.48) .. controls (3.28, 7.91) and (3.64, 8.27) .. 
          (4.08, 8.27) .. controls (4.88, 8.27) and (4.88, 7.13) .. (4.88, 6.16);
    \draw (4.88, 5.60) .. controls (4.88, 4.66) and (4.88, 3.48) .. (5.39, 3.48);
    \draw (5.95, 3.48) .. controls (6.31, 3.48) and (6.48, 3.88) .. (6.48, 4.28);
    \draw (6.48, 4.28) .. controls (6.48, 4.72) and (6.83, 5.07) .. (7.27, 5.07);
    \draw (7.27, 5.07) .. controls (7.65, 5.07) and (8.07, 5.15) .. 
          (8.07, 5.48) .. controls (8.07, 5.72) and (7.82, 5.88) .. (7.55, 5.88);
    \draw (6.99, 5.88) .. controls (6.29, 5.88) and (5.58, 5.88) .. (4.88, 5.88);
    \draw (4.88, 5.88) .. controls (4.44, 5.88) and (4.00, 5.88) .. (3.56, 5.88);
  \end{scope}

%% file: HKL.tex
\section{Topological obstructions}
\label{sec: HKL}

The main topological slice obstruction we used was that of Herald,
Kirk, and Livingston \cite{HeraldKirkLivingston2010}, which uses
twisted Alexander polynomials and the theory developed by Kirk and
Livingston \cite{KirkLivingston1999a} to provide a
computationally-efficient form of the Casson-Gordon invariants
\cite{CassonGordonOrsay}.  We introduce several refinements of the
methods of \cite{HeraldKirkLivingston2010} in Sections~\ref{sec: HKL
  refinement}, \ref{sec: poly is a norm}, \ref{sec: HKL with q=2}, and
\ref{sec: q^e}; we refer to these broadly as \emph{HKL tests}.  The
main result of this section is:
\begin{theorem}
  \label{thm: HKL obs summary}
  Of the knots in \PS, at least \NumObsByAnyHKLTest\ (57.1\%) are
  obstructed from being topologically slice by an HKL test. 
\end{theorem}
The contributions of the individual techniques is given in
Section~\ref{sec: HKL params}.

\subsection{Basic setup} \label{sec: HKL setup}

We first outline how these invariants work, following
\cite{HeraldKirkLivingston2010, KirkLivingston1999a,
  FriedlVidussi2011}.  For a knot $K$ in $S^3$, let $E_K$ be its
exterior $S^3 \setminus \nu(K)$, and let $B_m$ be the $m$-fold cyclic
cover of $S^3$ branched over $K$ for some $m \geq 2$.  Consider
$H_1(B_m; \Z)$ and a homomorphism $\chi \maps H_1(B_m; \Z) \to \Z/q\Z$
for some integer $q \geq 2$.  (We assume here that $H_1(B_m; \Z)$ is
finite, which can easily be arranged or checked directly.)
By identifying $\Z/q\Z$ with the multiplicative group generated by
$\zeta_q = e^{2 \pi i/q}$ in $\Q(\zeta_q)$, we can see $\rho$ as a representation
of $H_1(B_m; \Z)$ on $\Q(\zeta_q)$, which induced 
down from $\pi_1(\mbox{$m$-fold cyclic cover of $E_K$)}$ to $\pi_1(E_K)$,
gives a representation
$\rho_\chi \maps \pi_1 (E_K) \to \GL{m}{\big(\Q(\zeta_q)\big)}$.  If
$Z_K$ is the $0$-framed Dehn surgery on $K$, then $\rho_\chi$ extends
to $\pi_1(Z_K)$ and the criterion will be in terms of the twisted
Alexander polynomial of $(Z_K, \rho_\chi)$, which is a Laurent
polynomial in $\Q(\zeta_q)[t^{\pm 1}]$ that we denote as $\Delta_\chi$ but
is $\Delta_{Z_K, (\Q[\zeta_q[t^{\pm 1}])^m}$ in
\cite{HeraldKirkLivingston2010}, $\Delta_1(Z_K; \rho_\chi)$ in
\cite{KirkLivingston1999a}, and $\Delta_{Z_K}^{\alpha(m, \chi)}$ in
\cite{FriedlVidussi2011}.

The basic idea is that if $K$ is topologically slice, then $B_m$ is
the boundary of the 4-manifold $N_m$ obtained by taking the $m$-fold
cyclic cover of $D^4$ branched over the slice disk.  If $\chi$ extends
to $H_1(N_m; \Z)$, possibly after enlarging the range to $\Z/q^k\Z$,
then $\Delta_\chi(t)$ must be a \emph{norm}, that is, equal to
$f(t) \fbar(t)$ up to units in $\Q(\zeta_q)[t^{\pm 1}]$; here $\fbar$
is the result of applying the involution to $\Q(\zeta_q)[t^{\pm 1}]$
that interchanges $t$ and $t^{-1}$ and is complex conjugation on
$\Q(\zeta_q)$.  To understand when $\chi$ might extend to $N_m$, one
induces the notion of an \emph{invariant metabolizer}, a subgroup $A$
of $H := H_1(B_m; \Z)$ where $A$ is invariant under the
$\Z/m\Z$-action on $H$, the orders satisfy $\abs{A}^2 = \abs{H}$, and
$A$ is its own orthogonal complement under the nonsingular linking
form $\lk \maps H \otimes H \to \Q/\Z$.  This is relevant because the
kernel of $H \to H_1(N_m; \Z)$ is necessarily such an invariant
metabolizer.  The following is a reformulation of the results of
Section 6 of \cite{KirkLivingston1999a} as in
\cite{FriedlVidussi2011}:

\begin{theorem}[\cite{KirkLivingston1999a}]
  \label{thm: main KL}
  Suppose $m$ and $q$ are prime powers where $q$ is odd.  If $K$ is
  topologically slice, then there exists an invariant metabolizer
  $A \subset H$ such that every
  $\chi \maps H \to \Z/q\Z$ with $A \subset \ker(\chi)$
  gives a $\Delta_\chi$ that is a norm.
\end{theorem}
To use this as a slice obstruction, one considers all possible
invariant metabolizers $A$ and for each one finds at least one $\chi$
for which $\Delta_\chi$ is not a norm; the precise tests we used are
discussed in Section~\ref{sec: HKL implement}.  As in
\cite{HeraldKirkLivingston2010}, we do not actually compute the form
$\lk$, and so work with the possibly larger class of $A$ that are
invariant and have the correct size.

\subsection{The original HKL  test}
\label{sec: HKL implement}

Until Section~\ref{sec: q^e}, we will only consider the special case
of Theorem~\ref{thm: main KL} when $q$ is prime.  We also assume that
$\gcd(m, q) = 1$ as otherwise it turns out that $H^1(B_m; \F_q)$ will
be zero.  Since $q$ is prime, any $\chi \maps H \to \Z/q\Z$ can be
viewed as an element of $H^1(B_m; \F_q)$.  Since $m$ and $q$ are
coprime, the polynomial $x^m - 1$ in $\F_q[x]$ is separable, and, by
the theory of rational canonical form for matrices \cite[\S
12.2]{DummitFoote}, the irreducible representations of $\Z/m\Z$ on
finite-dimensional $\F_q$-vector spaces correspond to the irreducible
factors of $x^m - 1$ in $\F_q[x]$. Indeed, given an irreducible
representation $\psi$ of $\Z/m\Z$ on such a vector space $V$, the
corresponding factor of $x^m - 1$ is the characteristic polynomial of
the $\psi$-action of the preferred generator of $\Z/m\Z$.

Switching to the language of modules over $\F_q[\Z/m\Z]$, decompose
$\Hbar = H_1(B_m; \F_q)$ into irreducible submodules as
\[
  \Hbar = \bigoplus_{i = 1}^k V_i^{e_i} \mtext{where the
    $V_i$ are distinct and all $e_i \geq 1$.}
\]
Let $\Abar$ be the image of a metabolizer $A \subset H_1(B_m; \Z)$ in
$\Hbar$; by Lemma 8.2 of \cite{HeraldKirkLivingston2010}, it is a
\emph{proper} subspace of $\Hbar$.  
By Schur's lemma, there are $\Z/m\Z$-invariant subspaces
$W_i$ of $V_i^{e_i}$ so that $\Abar = \bigoplus_{i = 1}^k W_i$.

For an $\F_q$-vector space $W$, we denote the dual vector space of
linear functionals $W \to \F_q$ by $W^*$, and further identify 
$H^1(B_m; \F_q)$ with $\Hbar^*$. When all the multiplicities $e_i = 1$,
we have the following principle that underlies the applications in
Section~10 of \cite{HeraldKirkLivingston2010}:

\begin{lemma}
  \label{lem: basic HKL}
  Suppose $\Hbar$ decomposes into $\bigoplus_{i = 1}^k V_i$ where
  the $V_i$ are distinct irreducibles.  Suppose each $V_i^*$ contains
  a $\chi_i$ where $\Delta_{\chi_i}$ is not a norm (here $\chi_i$ is
  defined as $0$ on the other $V_j$).  Then $K$ is not topologically
  slice.
\end{lemma}

\begin{proof}
As $\Abar$ is a proper subspace of $\Hbar$ and is the direct sum of a
subset of the $V_i$, after reindexing the $V_i$ we may assume $\Abar$
is contained in $\bigoplus_{i = 2}^k V_i$.  Then $\Abar$ is contained
in $\ker(\chi_1)$, and, since $\Delta_{\chi_1}$ is not a norm,
Theorem~\ref{thm: main KL} implies that $K$ is not topologically
slice.
\end{proof}

%
%
%

\subsection{A refined HKL test}
\label{sec: HKL refinement}

We now analyze the image $\Abar$ of the metabolizer more carefully,
leading to Theorem~\ref{thm: imp HKL} which improves Lemma~\ref{lem:
  basic HKL}. As in \cite[\S 8]{HeraldKirkLivingston2010}, for any
prime $q$ we consider the localization $H_{(q)} = H \otimes \Z_{(q)}$
of $H$, which concretely is the subgroup of elements whose order is a
power of $q$; hence $H = \bigoplus \setdef{H_{(q)}}{\mbox{$q$ prime}}$
where each $H_{(q)}$ is $\Z/m\Z$-invariant.  The subgroups $H_{(q)}$
are orthogonal under the linking form $\lk$, and hence the restriction
of $\lk$ to each $H_{(q)}$ is nondegenerate.  Moreover, the
metabolizer similarly decomposes into the sum of
$A_{(q)} \leq H_{(q)}$, and one can check that
$\abs{A_{(q)}}^2 = \abs{H_{(q)}}$.  Here is the key fact that
strengthens Lemma 8.2 of \cite{HeraldKirkLivingston2010} and allows us
to go beyond Lemma~\ref{lem: basic HKL}:

\begin{lemma}[\cite{Sawin2024}]
  \label{lem: small meta}
  As $\F_q$-vector spaces, $\dim \Abar \leq \frac{1}{2} \dim \Hbar$.
  Moreover, $H_{(q)}/A_{(q)} \cong A_{(q)}$.
\end{lemma}

\begin{proof}
For any group $N$, define $N^{\vee} = \Hom(N, \Q/\Z)$. Set
$G = H_{(q)}$.  Since $\lk$ on $G$ is nondegenerate, the natural map
$\iota \maps G \to G^{\vee}$ defined by
$g \mapsto \lk(g, \cdotspaced)$ is injective.  For any $M \leq G$, the
surjection $G \surjects G/M$ induces an injection
$(G/M)^\vee \injects G^\vee$, and we identify $(G/M)^\vee$ with its
image.  Note that $M^\perp$ is the preimage of $(G/M)^\vee$ under
$\iota$.  Taking $M$ to be the metabolizer $A_{(q)}$, from
$M = M^\perp$ we get that $M \cong (G/M)^\vee$; since $N^\vee \cong N$
for any finite group, we have $M \cong G/M$ proving the second claim.

Recall that $\rank N$ is the minimal number of generators of a group
$N$; when $N$ is a $q$-group,
$\rank N = \dim_{\F_q}\big(N \otimes \Z/q\Z\big)$.  Let $\Gbar = G/qG$
and $\Mbar$ be the image of $M$ in $\Gbar$.  Note that
$\rank(G/M) = \dim \Gbar - \dim \Mbar$.  Since $M \cong G/M$, we have
\[
  \rank M + \dim \Mbar = \dim \Gbar.
\]
Since $\dim \Mbar \leq \rank M$, we have
$2 \dim \Mbar \leq \dim \Gbar$ as needed.
\end{proof}

\begin{theorem}
  \label{thm: imp HKL}
  Suppose $\Hbar$ decomposes into distinct irreducibles as
  $\bigoplus_{i = 1}^{k} V_i^{e_i}$.  We say that $V_i$ is
  \emph{weakly obstructed} when there is some $\chi \in (V_i^{e_i})^*$
  where $\Delta_\chi$ is not a norm; it is \emph{strongly obstructed}
  if for every proper invariant subspace $W \subset V_i^{e_i}$ there is a
  $\chi \in (V_i^{e_i})^*$ with $W \subset \ker \chi$ where
  $\Delta_\chi$ is not a norm. (When $e_i = 1$, these notions
  coincide.)  Define
  \[
    \delta(V_i) = \begin{cases}
      \dim V_i^{e_i} &  \text{if $V_i$ is strongly obstructed} \\
      \dim V_i      &  \text{if $V_i$ is weakly obstructed} \\
      0             &  \text{otherwise}
    \end{cases}
  \]

  If $\sum_{i = 0}^k \delta(V_i) >  \frac{1}{2} \dim \Hbar$, then $K$ is not
  topologically slice.
\end{theorem}

\begin{proof}
As noted above, $\Abar$ is the direct sum of invariant subspaces
$A_i \subset V_i^{e_i}$.  If $\dim A_i \geq \delta(V_i)$ for all $i$,
then
\[
  \dim \Abar = \sum_{i = 0}^k \dim A_i \geq \sum_{i = 0}^k
  \delta(V_i) > \frac{1}{2} \dim \Hbar
\]  
contradicting Lemma~\ref{lem: small meta}.  Therefore, we have an $i$
with $0 \leq \dim A_i < \delta(V_i)$. If $V_i$ is strongly obstructed
so that $\delta(V_i) = \dim V_i^{e_i}$, this means
$A_i \neq V_i^{e_i}$; if $V_i$ is weakly obstructed so
$\delta(V_i) = \dim(V_i)$, we get $A_i = \{0\}$. In either case, there
is a $\chi \in (V_i^{e_i})^*$ vanishing on all of $\Abar$, and
Theorem~\ref{thm: main KL} applies to show $K$ is not topologically
slice.
\end{proof}

\subsection{Checking if a polynomial is a norm}
\label{sec: poly is a norm}

The paper \cite{HeraldKirkLivingston2010} does not give a general
procedure for deciding when $\Delta_\chi$ is a norm; rather, the
authors used various techniques (e.g.~a form of Gauss's lemma over
$\Z[\zeta_q]$) to show that the $\Delta_\chi$ appearing in their
computations were not norms.  Below in Proposition~\ref{prop: norm
  test}, we give a general criterion that is easy to implement
provided one can factor polynomials with coefficients in a number
field, which is a well-studied problem \cite{Cohen1993, SageMath}.

We work in the setting of a subfield $F$ of $\C$ that is invariant
under complex conjugation.  Let $R = F[t, t^{-1}]$ be its ring of
Laurent polynomials. Consider the involutive ring isomorphism that
combines complex conjugation with $t \mapsto 1/t$; that is, if
$f(t) = \sum_n a_n t^n$, set $\fbar(t) = \sum_n \abar_n t^{-n}$
where the $a_n \in F$ and the sum is actually finite.  Recall that
$f \in R$ is a norm when there is a $g \in R$ with $f = u g \gbar$ for
some unit $u$ of $R$ (here $u = a t^n$ for some $a \in F$). We have
the following test:
\begin{proposition}
  \label{prop: norm test}
  Suppose $f \in R$ is $u p_1^{e_1} p_2^{e_2} \cdots p_k^{e_k}$ where
  $u \in R$ is a unit and the $p_i \in R$ are non-associate
  irreducibles.  Then $f$ is a norm if and only if for every $i$
  either
  \begin{enumerate}
  \item \label{item: self-assoc}
    $\pbar_i$ is an associate of $p_i$ and $e_i$ is even, or
  \item \label{item: interchanged}
    $\pbar_i$ is an associate of some other $p_j$ and $e_i = e_j$.
  \end{enumerate}
\end{proposition}

\begin{proof}
Recall that $R$ is a principal ideal domain (it is a localization of
the PID $F[t]$) and hence has unique factorization. To avoid messing
with units, we work with ideals rather than elements. If $f$ has the
required factorization we can write
\[
(f) = (q_1)^{2 d_1} \cdots (q_\ell)^{2 d_\ell}
\big((r_1)(\rbar_1)\big)^{e_1} \cdots \big((r_m)(\rbar_m)\big)^{e_m}
\]
where each $(q_i) = (\qbar_i)$ and all of the $(q_i)$, $(r_i)$, and
$(\rbar_i)$ are distinct.  We then take $(g) = (q_1)^{d_1} \cdots (q_\ell)^{d_\ell}
(r_1)^{e_1} \cdots (r_m)^{e_m}$ to show $(f)$ is a norm.

Conversely, suppose $f$ is a norm, that is $(f) = (g)(\gbar)$, and
consider a factorization $(g) = (q_1)^{d_1} \cdots (q_k)^{d_k}$ where
the $(q_i)$ are distinct prime ideals. As $h \mapsto \hhbar$ is a ring
isomorphism, we have that $(\qbar_1)^{d_1} \cdots (\qbar_k)^{d_k}$ is
such a factorization of $(\gbar)$.  Consider the factorization of
$(f)$ gotten by combining those for $(g)$ and $(\gbar)$ and grouping
common factors.  If $(\qbar_i) = (q_i)$, then $(q_i)$ occurs in $(f)$
to the power $2 d_i$ as needed for (\ref{item: self-assoc}). If
$(\qbar_i) = (q_k)$ for some $k \neq j$, then $(q_i)$ and $(\qbar_i)$
both occur in $(f)$ to the power $d_i + d_k$ as needed for (\ref{item:
  interchanged}).  Finally, if $(\qbar_i) \neq (q_k)$ for all $k$,
then $(q_i)$ and $(\qbar_i)$ both occur in $(f)$ to the power
$d_i$, and so $(f)$ has the required factorization.
\end{proof}

\subsection{Allowing $q = 2$}
\label{sec: HKL with q=2}

We now explain why Theorem~\ref{thm: main KL} excludes $q = 2$ and
give a variant that can be applied even in this case.  As noted in
Section~\ref{sec: HKL setup}, one may have to enlarge the range of
$\chi$ to $\Z/q^k\Z$ to extend it over the 4-manifold $N_m$ bounded by
$B_m$.  Thus while $\Delta_\chi$ is in $\Q(\zeta_{q})[t^{\pm 1}]$, if
the knot is topologically slice, what we learn from Proposition~6.1 of
\cite{KirkLivingston1999a} is that $\Delta_\chi$ is a norm in the
larger ring $\Q(\zeta_{q^k})[t^{\pm 1}]$ where $k$ is unknown.  By
Lemma~6.4 of \cite{KirkLivingston1999a}, for odd $d$ the latter
condition is equivalent to being a norm in the original
$\Q(\zeta_{q})[t^{\pm 1}]$, giving Theorem~\ref{thm: main KL}.
However for $q = 2$, we have examples such as $f(t) = t^4 + 4 t^2 + 1$
which is irreducible in $\Q[t^{\pm 1}] = \Q(\zeta_2)[t^{\pm 1}]$ and
so not a norm there, but $f = g \gbar$ for
$g = t^2 + (\sqrt{2} i) t + 1$ in $\Q(\zeta_8)[t^{\pm 1}]$.

If the knot $K$ is ribbon, then $H_1(B_m ; \Z) \to H_1(N_m; \Z)$ is
onto and there is no need to extend $\chi$.  In fact, one has:

\begin{theorem}
  \label{thm: ribbon obs}
  Suppose $m$ is an odd prime power.  If $K$ is topologically homotopy
  ribbon, then there exists an invariant metabolizer $A \subset H$
  such that every $\chi \maps H \to \Z/2^k\Z$ with
  $A \subset \ker(\chi)$ gives a $\Delta_\chi$ that is a norm.
\end{theorem}

\begin{proof}
Recall that $K$ is topologically homotopy ribbon when it bounds a
locally flat disk $F \subset D^4$ so that the inclusion
$S^3 \setminus K \hookrightarrow D^4 \setminus F$ is surjective on
fundamental groups.  (Any ribbon knot is topologically homotopy ribbon
\cite[Lemma 3.1]{Gordon1981}.)  Consequently, the inclusion
$B_m \to N_m$ gives a surjection $H \to H_1(N_m; \Z)$; the kernel of
this is $A$.  As $A \subset \ker(\chi)$, we get that $\chi$ extends to $H_1(N_m; \Z)$
and hence $\Delta_\chi$ is a norm by Proposition~6.1 of
\cite{KirkLivingston1999a}.
\end{proof}

\begin{table}
  \centering
  {\small
  \begin{tabular}{llll}
    $17nh_{0630889}$ & $18nh_{09292518}$ & $18nh_{13798702}$ & $18nh_{25872205}$ \\
    $19nh_{000130563}$ & $19nh_{000130564}$ & $19nh_{001561948}$ & $19nh_{001746199}$ \\
    $19nh_{001785287}$ & $19nh_{015088058}$ & $19nh_{020746102}$ & $19nh_{026824671}$ \\
    $19nh_{032393076}$ & $19nh_{035320248}$ & $19nh_{035487682}$ & $19nh_{055867647}$ \\
    $19nh_{068872115}$ & $19nh_{083547570}$ & $19nh_{109593374}$ & $19nh_{144186247}$ \\
    $19nh_{144186248}$ & $19nh_{146789683}$ & $19nh_{150081216}$ & $19nh_{169852405}$ \\
  \end{tabular}
  }
  \caption{We do not know whether these 24 knots are smoothly or
    topologically slice, but they are not ribbon, or even
    topologically homotopy ribbon, by Theorem~\ref{thm: ribbon obs}
    using $m = 3$ and $k = 1$.  They all have
    $\Delta_K = t^4-2 t^3+3 t^2-2 t+1$. }
  \label{tab: not ribbon}
\end{table}
\begin{remark}
  Theorem~\ref{thm: ribbon obs} gives a ribbon obstruction that is not
  obviously a topological slice obstruction.  We found
  \NumNotHomotopyRibbon\ knots in \PS\ where Theorem~\ref{thm: ribbon
    obs} applies, but we have been unable to determine whether the
  knot is smoothly or topologically slice.  These are listed in
  Table~\ref{tab: not ribbon}, and are thus candidates for knots that
  are slice but not ribbon.
\end{remark}

To get a genuine topological slice obstruction when $q = 2$, we need a
criterion for showing that $f \in \Q[t^{\pm 1}]$ is not a norm in
$\Q(\zeta_{2^k})[t^{\pm 1}]$ for all $k$; equivalently, setting
$\Q(\zeta_{2^\infty}) = \bigcup_{k} \Q(\zeta_{2^k})$, we want to know
whether $f$ is a norm in $\Q(\zeta_{2^\infty})[t^{\pm 1}]$.
Proposition~\ref{prop: 2 case} below gives the needed test, allowing
us to apply Lemma~\ref{lem: basic HKL} and Theorem~\ref{thm: imp HKL}
when $q = 2$.

For use
in Section~\ref{sec: q^e}, we consider a slightly more general
question.  Suppose $K = \Q(\zeta_{2^n})$ and $f \in K[t^{\pm 1}]$ is
irreducible. Let $L_f \subset \C$ be the
splitting field of $f$ and take $\Gal(f) = \Gal(L_f/K)$.  For any
finite group $G$, we set:
\begin{equation}
  \label{eq: galois bound}
  d(G) = \max \setdef{k \geq 0}{\mbox{$\exists H \trianglelefteq G$ with
      quotient $\Z/2^k\Z$ or $\Z/2\Z \oplus \Z/2^{k - 1}\Z$}}
\end{equation}
and further define $d(f) = d(\Gal(f))$.

Our test is the following, where the case $K = \Q$ corresponds to $n =
1$, not $n = 0$:

\begin{proposition}
  \label{prop: 2 case}
  Let $K = \Q(\zeta_{2^n})$ with $n \geq 1$. Suppose
  $f \in K[t^{\pm1}]$ factors as
  $u p_1^{e_1} p_2^{e_2} \cdots p_k^{e_k}$ where $u$ is a unit and the
  $p_i$ are non-associate irreducibles.  Then $f$ is a norm in
  $\Q(\zeta_{2^\infty})[t^{\pm 1}]$ if and only if for every $i$ one
  of the following holds:
  \begin{enumerate}
  \item
    \label{a: sym even}
      $\pbar_i$ is an associate of $p_i$ and $e_i$ is even, or
    \item
      \label{b: sym odd}
      $\pbar_i$ is an associate of $p_i$ with $e_i$ odd, and $p_i$
      is a norm in $\Q(\zeta_{2^\ell})[t^{\pm 1}]$ with
      $\ell = n + d(p_i) + 1$, where $d(p_i)$ is from (\ref{eq: galois
        bound}), or
    \item
      \label{c: non sym}
      $\pbar_i$ is an associate of some other $p_j$ and $e_i = e_j$.
    \end{enumerate}
\end{proposition}

The following will be needed for Proposition~\ref{prop: 2 case}:

\begin{lemma}
  \label{lem: subfields}
  Every finitely-generated subfield $L$ of $\Q(\zeta_{2^\infty})$ is
  Galois over $\Q$ and has $[L:\Q] = 2^k$ for some $k \geq 0$.
  Moreover, the field $\Q(\zeta_{2^\infty})$ has exactly three
  subfields $L$ with $[L : \Q] = 2^k$ for each $k \geq 1$.  These are
  $\Q(\zeta_{2^{k + 1}})$, $\Q\big(\cos(\pi/2^{k+1})\big)$, and
  $\Q\big(\sin(\pi/2^{k+1}) i\big)$.  All such $L$ are contained in
  $\Q(\zeta_{2^{k+2}})$ and have $\Gal(L/\Q)$ either $\Z/2^k\Z$ or
  $\Z/2\Z \oplus \Z/2^{k - 1}\Z$.  Finally, if
  $M \subset \Q(\zeta_{2^\infty})$ is finitely-generated and contains
  $L$, then $M/L$ is Galois with group $\Z/2^\ell\Z$ or
  $\Z/2\Z \oplus \Z/2^{\ell - 1}\Z$ where $[M : L] = 2^\ell$.
\end{lemma}

\begin{proof}
Any finitely-generated $L$ is contained in some
$\Q(\zeta_{2^\ell})$.  Recall
\[
  G =
\Gal(\Q(\zeta_{2^\ell})/\Q) \cong (\Z/2^\ell\Z)^\times \cong \Z/2\Z
\oplus \Z/2^{\ell - 2}\Z
\]
and that any subfield $L$ of $\Q(\zeta_{2^\ell})$ corresponds to a
subgroup $H \leq G$.  As $G$ is abelian, $H$ is normal and so $L/\Q$
is Galois.  As $[L : \Q] = \abs{G/H}$ divides
$\abs{G} = 2^{\ell - 1}$, we have proved the first claim.

Now fix $k \geq 1$.  We first analyze the subfields $L$ of
$\Q(\zeta_{2^\ell})$ of degree $2^k$ over $\Q$ when $\ell \geq k + 2$.
Let $H \leq G$ be the subgroup corresponding to $L$.  Now, $G/H$ is
generated by two elements, one of which has order at most two, and so
is either $\Z/2^k\Z$ or $\Z/2\Z \oplus \Z/2^{k - 1}\Z$.  Up to
automorphisms of the target group, you can check that there are two
epimorphisms from $G$ onto $\Z/2^k\Z$ and one onto
$\Z/2\Z \oplus \Z/2^{k - 1}\Z$.  Thus $\Q(\zeta_{2^\ell})$ has exactly
three subfields of degree $2^k$ over $\Q$.

For any $m > \ell$, we similarly have $\Q(\zeta_{2^m})$ has three
subfields of degree $2^k$ over $\Q$, and therefore these subfields
must be the ones inherited from
$\Q(\zeta_{2^\ell}) \subset \Q(\zeta_{2^m})$.  Since any finitely-generated
subfield of $\Q(\zeta_{2^\infty})$ is contained in some
$\Q(\zeta_{2^m})$, we have proven that $\Q(\zeta_{2^\infty})$ has
exactly three subfields of degree $2^k$, specifically those coming
from $M = \Q(\zeta_{2^{k+2}})$. To identify them concretely, first note that
each corresponds to the fixed field of one of the three involutions 
in $\Gal(M/\Q)$.  If $\zeta = \zeta_{2^{k+2}}$,
these are $\zeta \mapsto \zetabar$, $\zeta \mapsto -\zeta$, and
$\zeta \mapsto -\zetabar$, whose fixed fields are
$\Q\big(\cos(\pi/2^{k+1}) = \zeta + \zetabar \big)$,
$\Q\big(\zeta_{2^{k + 1}} = \zeta^2\big)$, and
$\Q\big(\sin(\pi/2^{k+1}) i = \zeta - \zetabar\big)$,
respectively.

The final claims about $M/L$ now follow from our knowing $\Gal(M/\Q)$ is
$\Z/2^m\Z$ or $\Z/2\Z \oplus \Z/2^{m - 1}\Z$ and that the only
quotients of these groups have the same form.
\end{proof}

\begin{corollary}
  \label{cor: galois bound}
  Let $K = \Q(\zeta_{2^n})$ for $n \geq 1$ and suppose
  $f \in K[t^{\pm 1}]$ is irreducible.  If $f$ is a norm over
  $\Q(\zeta_{2^\infty})$ then it is a norm over $\Q(\zeta_{2^\ell})$
  with $\ell = n + d(f) + 1$.
\end{corollary}
  
\begin{proof}
Set $R = \Q(\zeta_{2^\infty})[t^{\pm 1}]$.  Suppose $f = u g \gbar$
with $g \in R$ and $u \in R$ a unit, i.e.~$u = a t^d$ for nonzero
$a \in \Q(\zeta_{2^\infty})$.  Multiplying by units in $R$, we
can assume that $f$ and $g$ are genuine polynomials in
$\Q(\zeta_{2^\infty})[t]$ with nonzero constant terms.  Then $f = a
t^n g \gbar$ where $n = \deg(g)$ and $a \in \Q(\zeta_{2^\infty})$.

Let $L_f$ be the splitting field of $f$ over $K$.  Writing
$g(t) = \prod^n_{i = 1} (t - \theta_i)$, we have $\theta_i \in L_f$ so
$g \in L_f[t^{\pm 1}]$. Now
$\gbar = \prod^n_{i = 1} (1/t - \thetabar_i)$, so each $1/\thetabar_i$
is a root of $f$ and hence $\thetabar_i$ is also in $L_f$.  Thus
$\gbar \in L_f[t^{\pm 1}]$ as well.  Solving for $a$ in
$f = a t^n g \gbar$, we see that $a \in L_f$ as well.  Thus $f = a t^n
g \gbar$ shows that $f$ is a norm over the field
$M = L_f \cap \Q(\zeta_{2^\infty})$.

Now $M$ is Galois over $K$ as both $L_f$ and $\Q(\zeta_{2^\infty})$
are, and so $M$ corresponds to a normal subgroup $H$ of $\Gal(f)$.  By
the last claim of Lemma~\ref{lem: subfields}, $\Gal(M/K)$ is
$\Z/2^k\Z$ or $\Z/2\Z \oplus \Z/2^{k - 1}\Z$ for some $k \geq 0$.  As $\Gal(M/K)$ is also
$\Gal(f)/H$, we have $k \leq d(f)$ by (\ref{eq: galois bound}).  So
$[M : \Q] = [M : K][K : \Q] = 2^{k + (n - 1)} \leq 2^{n + d(f) - 1}$.
Again by Lemma~\ref{lem: subfields}, we have
$M \subset \Q(\zeta_{2^{\ell}})$ where $\ell = n + d(f) + 1$, and so $f$
is a norm over $\Q(\zeta_{2^{\ell}})$ as needed.
\end{proof}

We now prove the main result of this subsection.

\begin{proof}[Proof of Proposition~\ref{prop: 2 case}]
The ``only if'' direction is straightforward and left to you. For the
other, suppose $f$ is a norm in some $\Q(\zeta_{2^\ell})[t^{\pm
  1}]$. By Proposition~\ref{prop: norm test}, as ideals in
$\Q(\zeta_{2^\ell})[t^{\pm 1}]$ we have:
\[
(f) = (q_1)^{2 d_1} \cdots (q_\ell)^{2 d_\ell}
\big((r_1)(\rbar_1)\big)^{e_1} \cdots \big((r_m)(\rbar_m)\big)^{e_m}
\]
where each $(q_i) = (\qbar_i)$ and all of the $(q_i)$, $(r_i)$, and
$(\rbar_i)$ are distinct.  Consider the action of
$G = \Gal\big(\Q(\zeta_{2^\ell})/K\big)$ on these ideals, which
permutes them amongst themselves.  Since the absolute Galois group
$\Gal\big(\Q(\zeta_{2^\ell})/\Q\big)$ is abelian and complex
conjugation is one of its elements, we have
$\overline{\sigma \cdot (g)} = \sigma \cdot (\gbar)$ for all
$\sigma \in G$ and $g \in \Q(\zeta_{2^\ell})[t^{\pm 1}]$.  So the
$G$-orbit of $(q_1)$ consists only of certain other $(q_i)$; by
uniqueness of factorizations, we have $d_i = d_1$ in this case. The
product of the ideals in the $G$-orbit of $(q_1)$ gives an ideal $(p)$
generated by a $p \in K[t^{\pm 1}]$ which is irreducible there, and
indeed $p^{2 d_1}$ is a factor of $f$ in case (\ref{a: sym even}).
The same is true for the other orbits of the $(q_i)$.

Turning to $(r_1)$, there are two cases.  If $(\rbar_1)$ is in a
different $G$ orbit than $(r_1)$, the same is true of any other
$(r_j)$ or $(\rbar_j)$ in its orbit. Relabeling, we can assume the $G$
orbit of $(r_1)$ is all $(r_i)$ with $1 \leq i \leq k$; here
$e_i = e_1$ for such $i$.  The product of these $(r_i)$ gives $(p)$
with $p \in K[t^{\pm 1}]$ irreducible.  Now
$(\pbar) = (\rbar_1) (\rbar_2) \cdots (\rbar_k)$ is distinct from
$(p)$, so $p^{e_1}$ is a factor of $f$ of in case (\ref{c: non sym}).
Finally, consider when $(\rbar_1)$ is in the same $G$ orbit as
$(r_1)$. In this case, if $(p)$ is the product of all the ideals in
the orbit of $(r_1)$, we get the usual factor $p^{e_1}$ of $f$.  If
$e_1$ happens to be even, we are in case (\ref{a: sym even}), and
otherwise we apply Corollary~\ref{cor: galois bound} to land in case
(\ref{b: sym odd}).
\end{proof}

\begin{remark}
  Computing the Galois group of a polynomial is a highly nontrivial
  task.  This is especially the case when the base field is not $\Q$,
  where the only implementation we are aware of is in Magma
  \cite{Magma}.  However, the following result allows us to avoid
  computing the Galois group in practice.  For example, if
  $f \in \Q[t]$ has $\deg(f) \leq 14$, it shows we need only worry
  about $f$ being a norm over $\Q\big(\zeta_{2^5}\big)$ which has
  degree 16 over $\Q$.
\end{remark}

\begin{lemma}
  Suppose $f \in K[t]$ is irreducible, where $K$ has 
  characteristic $0$. For any even degree $\leq 30$, the following is
  an upper bound on $d(f)$:
  \begin{center}
    \begin{tabular}{c|ccc|ccc|ccc|ccc|ccc}
      \toprule
      $\deg(f)$ & $2$ & $4$ & $6$ & $8$ & $10$ & $12$ & $14$ & $16$ & $18$ & $20$ & $22$ & $24$
      & $26$ & $28$ & $30$ \\ \midrule
      $d(f) \leq$ & $1$ & $2$ & $2$ & $3$ & $3$ & $3$ & $2$ & $4$ & $4$ & $4$ & $2$ & $4$ & $3$ & $3$ & $4$ \\
      \bottomrule
    \end{tabular}
  \end{center}
\end{lemma}

\begin{proof}
Let $n = \deg(f)$.  Since $K$ has characteristic $0$ and $f$ is
irreducible, the roots of $f$ are distinct. Considering the action of
$\Gal(f)$ on the $n$ roots, we get $\Gal(f) \leq S_n$ and moreover,
$\Gal(f)$ acts transitively on $\{1, 2, \ldots, n\}$.  Using the
classification of transitive permutation groups in low degrees
\cite[Table~1]{HoltRoyle2020}, we used \cite{GAP} to compute $d(G)$
for each of them and took maximums to get the values in the table.
Here $d(G)$ can easily be computed from
the abelianization of $G$, and doing so for all 37,256 relevant
transitive groups takes about a minute.
\end{proof}

\subsection{Quotients of nonprime order}
\label{sec: q^e}

So far, we have considered $\chi \maps H \to \Z/q\Z$ where $q$ is a
prime. However, Theorem~\ref{thm: main KL} also permits the case of
$\chi \maps H \to \Z/q^e\Z$ where $e > 1$.  Now $\Z/q^e\Z$ is not a
field, so to follow the discussion leading to Lemma~\ref{lem: basic
  HKL}, we would need to think about $(\Z/q^e\Z)[\Z/m\Z]$-modules and
their attendant quirks.  To sidestep this issue for small enough $m$
and $q^e$, one can work directly with a triangulation of $B_m$ and an
associated presentation of $\pi_1B_m$.  One then uses this
presentation to describe $H$ and find all $\chi \maps H \to \Z/q^e\Z$
directly.  One then enumerates all possible metabolizers $A_{(q)}$ of
$H_{(q)}$ that satisfy Lemma~\ref{lem: small meta} and sees whether
each can be ruled out by some $\chi$.  This direct approach is much
less computationally efficient than in the field case where we can
exploit the decomposition of $\Hbar$ into irreducibles.
However, we were still able to use it to obstruct
\NumObsByDirectHKLTest\ examples where we had no success using just
$\chi$ taking values in fields. In those cases, $m = 2$ and
$q = 3, 5, 7$ or $(m, q) \in \{(3, 2), (3, 7)\}$.  In fact, $(2, 3)$
worked in \num{3380} cases and $(3, 2)$ worked in \num{438} cases,
with the other values of $(m, q)$ accounting for just 15 cases.  (If
one knows the $\Z/m\Z$-action on $H$, then one could restrict to the
possible metabolizers that are $\Z/m\Z$-invariant. We did not take the
care necessary to do so, especially as when $m = 2$ the action is just
by $\pm 1$ and hence adds no restrictions.)

\subsection{Effective methods and parameters}
\label{sec: HKL params}

As stated in Theorem~\ref{thm: HKL obs summary}, in total we showed
that \NumObsByAnyHKLTest\ (57.1\%) knots in \PS\ are not topologically
slice using an HKL test.  The original Lemma~\ref{lem: basic HKL}
accounted for the vast majority, namely \NumObsByBasicHKLTest\
(55.2\%).  Theorem~\ref{thm: imp HKL} or Proposition~\ref{prop: 2
  case} was used for \NumObsByFancyHKLTest\ (1.9\%) knots, and the
method of Section~\ref{sec: q^e} was used for \NumObsByDirectHKLTest\
(0.03\%) knots.


For the successful obstructions, the degree $m$ of the cover used was
primarily $2$ (84.1\%), $3$ (10.6\%), $5$ (2.6\%), $7$ (1.4\%), $11$
(0.8\%), and $31$ (0.3\%), with the remaining values of $m$ accounting
for 0.2\% of the successful HKL tests; the largest $m$ was $71$.  The
value of $q$ used had median $7$, mean $14.3$, and $\sigma =
13.6$. The largest $q$ was $661$, though 99.6\% had $q < 50$.  We used
$q = 2$, and hence Proposition~\ref{prop: 2 case}, for
\NumObsByHLKWhereQIsTwo\ knots, which is 3.0\% of the successful HKL
tests.

\subsection{Implementation}

Our implementation of these topological slice obstructions was
incorporated into the software package SnapPy \cite{SnapPy} and is
available when one uses it in SageMath \cite{SageMath}.  The initial
version was released in 2021 as part of SnapPy 3.0 and was restricted
to the test of Lemma~\ref{lem: basic HKL}, and also required $m$ to be
prime (not just a prime power); this is the version referred to in
\cite{ManolescuPiccirillo2023, OwensSwenton2023}.  An enhanced version
that uses Theorem~\ref{thm: imp HKL}, allows $m$ to be a prime power,
and uses Proposition~\ref{prop: 2 case} when $q = 2$ will be part of
SnapPy 3.3 and is included in \cite{CodeAndData}.  One difference
compared to \cite{HeraldKirkLivingston2010} is that we use the default
simplified presentation for $\pi_1 E_K$ that SnapPy provides, rather
than a Wirtinger presentation coming from a knot diagram; the former
has many fewer generators but much longer relators.

\subsection{Possible refinements}

We end this section by outlining two possible refinements.  Given how
effective this technique is overall, we expect either would likely
obstruct topological sliceness of some of the
\approxremainingknotstop\  knots where that property is unknown.

First, one could greatly reduce the number of metabolizers to be ruled
out by computing the linking form $\lk$ on $H_1(B_m; \Z)$. To
illustrate the idea, consider the common situation where $m = 2$ and
$H_{(q)} = (\Z/q\Z)^2 = \Hbar$.  Here, the generator of $\Z/2\Z$ acts
on all of $\Hbar$ by $v \mapsto -v$, so $\Hbar$ decomposes as $V_1^2$
where $\dim V_1 = 1$.  To apply Theorem~\ref{thm: imp HKL}, we will
need to strongly obstruct $V_1^2$, and so we need some $q + 1$
different $\Delta_\chi$ to not be norms.  Now
$\lk$ must be the standard hyperbolic bilinear form on $\Hbar$, that
is, there is a basis $\{x, y\}$ of $\Hbar$ with
$\lk(x, x) = \lk(y, y) = 0$ and $\lk(x, y) = 1/q$ (see e.g.~Chapter 7
of \cite{Aschbacher2000}, specifically (21.2) and note $\Hbar$ must
have Witt index 1 as $\dim \Abar = 1$).  Thus, $\pair{x}$ and
$\pair{y}$ are the only possibilities for $\Abar$ and so we need only
check two $\Delta_\chi$.  (A caveat is that many ribbon knots in this
situation have all $\Delta_\chi$ being nonnorms, in which case
Theorem~\ref{thm: imp HKL} would succeed regardless.)  Direct
computation of $\lk$ from a triangulation was done in
\cite{BudneyBurton2020} using Regina \cite{Regina}.  When $m$ is
small, Regina succeeds in computing $\lk$ for the $B_m$ that are of
interest here.  However, it would require delicate book-keeping to
``line up'' the answer for $\lk$ with the computation of
$\Delta_\chi$, so we do not pursue this here. (Another approach to
find $\lk$ would be to use \cite{Nosaka2022}, which requires a Seifert
surface.)

Second, one can often determine enough of $\lk$ to greatly restrict
the possible $\Abar$ just from the decomposition of $\Hbar$ into
irreducibles.  For an $\F_q[\Z/m\Z]$-module $V$, let $V^*$ be its dual
or contragredient; in coordinates, if
$\psi \maps \Z/m\Z \to \GL{n}{\F_q}$ describes $V$ then
$\psi^*(g) = (\psi(g^{-1}))^t$.  If $V$ and $W$ are irreducible
invariant subspaces of $\Hbar$, we get a map $V \to W^*$ of
$\F_q[\Z/m\Z]$-modules by $v \mapsto \lk(v, \cdotspaced)$; by Schur's
lemma, this implies that $V$ and $W$ are orthogonal unless
$V \cong W^*$.  Suppose that $\Hbar$ decomposes as
$\bigoplus_{i = 1}^k V_i^{e_i}$ where no $V_i$ is self-dual and so
isomorphic to $V_i^*$.  Then $\lk$ vanishes on each $V_i^{e_i}$, and
we can rewrite the decomposition as an \emph{orthogonal} sum
$\bigoplus_{i = 1}^{k/2} W_i$ where
$W_i = V_i^{e_i} \oplus (V_i^*)^{e_i}$ and nondegeneracy forces each
nonzero $v \in V_i$ to pair nontrivially with some $w \in V_i^*$.  For
example, if
$\Hbar = V_1 \oplus V_1^* \oplus V_2 \oplus V_2^* \oplus V_3 \oplus
V_3^*$ where all $V_i$ have the same dimension, then Theorem~\ref{thm:
  imp HKL} requires us to find nonnorm $\Delta_\chi$ for four of these
factors to obstruct slicing.  However, we claim that if there are
nonnorm $\Delta_\chi$ for both $V_i$ and $V_i^*$ for a single $i$ then
the knot cannot be topologically slice.  This is because any $\Abar$
will have to meet $V_i \oplus V_i^*$ in $\{0\}$, $V_i$, or $V_i^*$,
and so $\Abar$ is in the kernel of at least one of the two $\chi$.

%% file: smooth.tex
\section{Smooth slice obstructions}
\label{sec: smooth}

We computed many different smooth slice obstructions coming from Khovanov
homology, knot Floer homology, and gauge theory, yielding: 

\begin{theorem}
  \label{thm: smooth summary}
  Of the knots in \PS, some \NumObsBySmoothInvariants\ (8.6\%) are
  obstructed from being smoothly slice by the invariants listed in
  Table~\ref{tab: smooth obs}.  
\end{theorem}

\begin{remark}
  Note that Theorem~\ref{thm: HKL obs summary}, which uses the
  topological HKL tests, covers 6.6~times more knots than
  Theorem~\ref{thm: smooth summary}.  Moreover, Theorem~\ref{thm:
    smooth summary} obstructs only \NumObsBySmoothInvButNotHKL\ knots
  (0.3\% of \PS) that are not included in Theorem~\ref{thm: HKL obs
    summary}.
\end{remark}

\begin{table}
  \centering
  \begin{tabular}{rrl}
    \toprule
    \#$K \in \PSM$ & \#$K \in \HKLC$ & smooth slice obstruction\\
    \midrule
    \num{195155} & \num{6776} & One of the $s$-invariants $\{s_\Q, s_{\F_2}, s_{\F_3}\}$  \\
    \num{60929} &  \num{2998} & The Steenrod square $\svec^{\Sq^1_o}$ on $\Kh_\odd^*$\\
    \num{890}   &  \num{838} & The LEO invariant $\svectil_c$ \\ 
    \num{6} & \num{6} & An $s$-invariant for the $\mathfrak{sl}_3$-version of $\Kh^*$ \\
    \num{76885} & {13} & Goeritz bifactor.~test for alternating knots \\
    \midrule
    \NumObsBySmoothInvariants & \NumObsBySmoothInvButNotHKL & total distinct knots obstructed \\
    \bottomrule
  \end{tabular}
  
  \caption{Above is the breakdown of which smooth slice obstructions
    were used to prove Theorem~\ref{thm: smooth summary}.  The first
    column gives the number of knots in \PS\ obstructed by each
    invariant.  For the second column, $\HKLC \subset \PSM$ is the
    complement of the knots shown not to be topologically slice by
    Theorem~\ref{thm: HKL obs summary}; hence this column represents
    the new information gleaned from Theorem~\ref{thm:
      smooth summary}. Each knot is counted only once, for the
    invariant that obstructs it that is highest on the table.  The LEO
    invariant $\svectil_c$ necessarily obstructs anything handled by
    $\{s_\Q, s_{\F_2}, s_{\F_3}, \svec^{\Sq^1_o}\}$, for a total of
    \num{256974} knots in \PS\ and \num{10612} knots in $\HKLC$.  }
  \label{tab: smooth obs}
\end{table}

The set of invariants used to prove Theorem~\ref{thm: smooth summary} were:

\begin{enumerate}
\item
  \label{item: s}
  Rasmussen's $s$-invariant \cite{Rasmussen2010}, over the fields
  $\Q$, $\F_2$, and $\F_3$.  We denote these by $s_\Q$, $s_{\F_2}$, and,
  $s_{\F_3}$, each of which takes values in $\Z$ and is $0$ on all
  smoothly slice knots.
  
\item
  \label{item: Sq1 odd}
  The $\Sq^1$-invariant in odd Khovanov homology from 
  \cite[\S 5.6]{SarkarScadutoStoffregen2020}.  We denote it by
  $\svec^{\Sq^1_o}$, and it takes values in $\Z^4$ and is $0$ on all
  smoothly slice knots.
  
\item
  \label{item: LEO}
  The invariant $\svectil_c$ from \cite[\S
  7]{DunfieldLipshitzSchuetz}, which is defined in terms of local
  equivalence classes of local even-odd (LEO) triples that mix
  together the original and odd flavors of Khovanov homology.  This
  takes values in $\Z^2$ and is $0$ on all smoothly slice knots.

\item \label{item: sl3} Various $s$-invariants associated to the
  $\mathfrak{sl}_3$-Khovanov homology \cite{Schuetz2025}.  These are
  quite difficult to compute, so we did so only for
  \NumSLThreeObsComputed\ knots, including all those unresolved by
  other methods.
  
\item
  \label{item: bifac}
  For alternating knots, we employed the criterion of
  \cite[Lemma~3]{LewarkMcCoy2019}, which involves bifactorability of
  the Goeritz matrices and is based on Donaldson's theorem.
\end{enumerate}

All computations of the Khovanov-based invariants in (\ref{item:
  s}--\ref{item: sl3}) were done using Sch\"utz's KnotJob
\cite{KnotJob}. By Section~6 of \cite{DunfieldLipshitzSchuetz}, the
LEO invariant $\svectil_c$ is a stronger smooth slice invariant than
any of $s_\Q, s_{\F_2}, s_{\F_3}, \svec^{\Sq^1_o}$; that is, if any of
those four invariants obstruct smooth slicing, so does $\svectil_c$.
For alternating knots in \PS, the obstruction (\ref{item: bifac})
specific to that setting was the only one of (\ref{item:
  s}--\ref{item: bifac}) that obstructed smooth slicing.  By
\cite[Prop.~6.30]{DunfieldLipshitzSchuetz}, none of (\ref{item:
  s}--\ref{item: LEO}) can obstruct smooth slicing for an alternating
knot with signature $0$, so this is to be expected.
  
Other invariants we looked at included:
\begin{itemize}
\item The invariants $\tau$, $\nu$, and $\epsilon$ from knot Floer
  homology, computed via \cite{HFKCalculator}. These obstructed
  essentially the same set of knots as did
  $(s_\Q, s_{\F_2}, s_{\F_3})$.  Specifically, every knot obstructed
  by one of $(\tau, \nu, \epsilon)$ was also obstructed by one of
  $(s_\Q, s_{\F_2}, s_{\F_3})$, with the converse holding 99.96\% of
  the time.

\item The $\Sq^1$-invariant from the original Khovanov homology
  \cite{LipshitzSarkar2014a} obstructs exactly the same knots in \PS\
  as $(s_\Q, s_{\F_2}, s_{\F_3})$.  The same is true for the
  $s^\Z$-invariant of \cite{Schuetz2022a}.  In general, neither of
  these invariants can give more information than the LEO invariant
 $\svectil_c$ \cite{DunfieldLipshitzSchuetz}.

\item The higher $\Sq^2$-invariant of \cite{LipshitzSarkar2014a} is
  very difficult to compute for a 19 crossing knot, and the analogues
  for odd Khovanov homology from \cite{Schuetz2022b} are even harder.
  (This is because there is no known Bar-Natan--style ``scanning
  algorithm'' for these invariants; compare \cite{Schuetz2022c}, which
  applies even to $\svectil_c$.)  While these invariants can provide
  slice obstructions not seen by $\svectil_c$, we only computed
  $\Sq^2$ from \cite{LipshitzSarkar2014a} for about \num{2700} knots
  where all the other Khovanov invariants vanished.  Unfortunately,
  $\Sq^2$ also vanished for these knots.  The same was true for the
  roughly \num{130} knots where we computed the versions of $\Sq^2$
  for $\Kh_\odd^*$.
\end{itemize}

%% file: friends.tex
\section{Zero-friends}
\label{sec: friends}

The 0-surgery on a knot $K$ in $S^3$ is the unique Dehn surgery where
$b_1 > 0$.  A pair of knots in $S^3$ are \emph{0-friends} when their
0-surgeries are homeomorphic. In this section, we study some 0-friends
of the knots in \PS.  The motivation for this is twofold.  First, and
most tantalisingly, if $K$ and $K'$ are 0-friends where $K$ is
smoothly slice and $K'$ is not, then there is an exotic smooth
structure on $S^4$, disproving the smooth 4-dimensional Poincar\'e
Conjecture; see \cite[\S 1]{ManolescuPiccirillo2023}.  Second, even
with that conjecture unresolved, in favorable circumstances one can
show that certain 0-friends must have the same smooth slice status;
for example, this was used to show that the Conway knot is not
smoothly slice \cite{Piccirillo2020}.  In Section~\ref{sec: friend
  app}, we use 0-friends to determine the smooth slice status of 25
knots in \PS. Section~\ref{sec: mystery} gives 4 pairs of 0-friends
where we can only determine smooth sliceness of one of the knots;
these are thus candidates for constructing exotic smooth 4-spheres.

\subsection{Finding 0-friends}
\label{sec: friend finder}

There is an algorithm to determine when two 3-manifolds are
homeomorphic (see e.g.~\cite{Kuperberg2019}).  While the known
algorithms are unimplementable, in practice one can still usually
decide this, especially in the presence of hyperbolic geometry
\cite{HodgsonWeeks1994}.  With 114 exceptions (0.003\%), the knots in
\PS\ have hyperbolic 0-surgeries, so the first question is how many
pairs of 0-friends there are in \PS\ itself.  The answer is just 101
pairs.\footnote{The common 0-surgery is hyperbolic for 98 of these 101
  pairs of 0-friends.  There are only two nontrivial groups of mutual
  0-friends, both with four knots: $\{K11n97$, $16n86947$,
  $18nh_{00000399}$, $18nh_{00000400}\}$ and $\{K11n67$,
  $18nh_{00000254}$, $19nh_{000000631}$, $19nh_{000000632}\}$.  See
  \cite{CodeAndData} for all 186 knots involved.}  This is consistent
with \cite{KegelWeiss2024}, which found only 6 such pairs among all
$\approx 313,000$ prime knots with $\leq 15$ crossings.  Note also
that the 0-friend used in \cite{Piccirillo2020} is much more
complicated than the Conway knot itself. Therefore, to usefully study
0-friends of knots in \PS, we need a technique that can generate them
from a single given knot $K$.

We used the following method from
\cite[\S9.3]{DunfieldObeidinRudd2024} to generate more than
\num{500000} 0-friends of the knots in \PS; we explain in
Section~\ref{sec: hyp rules} why this should generate most of the
interesting small-to-medium 0-friends.  One example pair of 0-friends
is shown in Figure~\ref{fig: besties}.

\begin{enumerate}
\item
  \label{item: geod}
  For each knot $K$ whose 0-surgery $Z_K = S^3_0(K)$ is hyperbolic, we
  looked at all closed geodesics $\gamma$ of length at most $3$ that
  generated $H_1(Z_K; \Z) \cong \Z$, using \cite{HodgsonWeeks1994,
    SnapPy}. (The reason we stopped at length $3$ is explained
  in Section~\ref{sec: hyp rules}.)
  
\item
  \label{item: drill}
  For each $\gamma$, we considered the exterior
  $X = Z_K \setminus \gamma$, which must itself be hyperbolic. The
  hyperbolic structure on $X$ determines a finite list of Dehn
  fillings on it that might be $S^3$, and this can be used to
  identify whether $X$ is the exterior of some unknown knot $K'$ in
  $S^3$.  (See e.g.~\cite{Dunfield2020b} for more on this.)

\item For each knot exterior $X$ identified in Step~\ref{item: drill},
  we used \cite{DunfieldObeidinRudd2024} to find an actual knot
  diagram for $K'$.
  
\end{enumerate}

\begin{remark}
  This method uses numerical computations about the hyperbolic
  structures that are not rigorous.  However, each pair of 0-friends
  in our data must be genuine since filling, drilling, and finding
  knot diagrams are all combinatorial operations that are implemented
  at the level of manipulating triangulations of the relevant
  3-manifolds.  Such numerical issues instead mean that our list of
  0-friends may not be complete according to the above description;
  still, we expect it contains 99\% of such 0-friends.
\end{remark}

\subsection{Overview of 0-friends}

We identified at least one 0-friend for 546,717 (14.1\%) of the knots
in \PS. (For these knots, we found an average of 1.5 0-friends, though
we were not careful about avoiding duplicates.) The crossing
number of the diagram of the 0-friend $K'$ had mean 225.7 and median
154, with standard deviation $\sigma$ of 219.2. (We make no claim that
these are minimal crossing diagrams, though we suspect very few can be
simplified to reduce the number of crossings by even 50\%.)  This makes
most of these diagrams much too large to allow computing Khovanov
homology and hence most of the invariants of Section~\ref{sec:
  smooth}.  In contrast, the exteriors $E_{K'}$ are only modestly more
complicated than the original $E_K$.  The number of tetrahedra in the
triangulation of $E_{K'}$ had mean and median 39.0 as compared to 31.0
for $E_{K}$, an increase of only 25\% (here $\sigma = 7.91$ and
$\sigma = 3.84$ respectively).

In the end, we were able to determine the smooth slice status for
78,507 0-friends of 56,319 knots in \PS; we denote these 0-friend
knots by \ZF\ and show information about their crossing number and
hyperbolic volume in Figure~\ref{fig: ZF basic}.  Some 82.0\% of the
knots in \ZF\ are ribbon and 18.0\% are not smoothly slice.  (For all
the knots in \PS\ where we found a 0-friend, we had 71.5\% ribbon and
28\% not smoothly slice, with 0.4\% unknown; this contrasts with \PS\
as a whole where 42.2\% are ribbon, 57.5\% are not smoothly lice, and
0.3\% unknown.)  Initially, based on the computations described in
Sections~\ref{sec: bands}--\ref{sec: smooth}, there were 36 knots in
\PS\ of unknown smooth slice status where this was determined for a
0-friend.  For 28 of these knots in \PS, we can use the 0-friend to
determine the smooth slice status of the original; see
Theorem~\ref{thm: slice via friends} for more and Section~\ref{sec:
  mystery} for the 8 mysterious remaining knots.

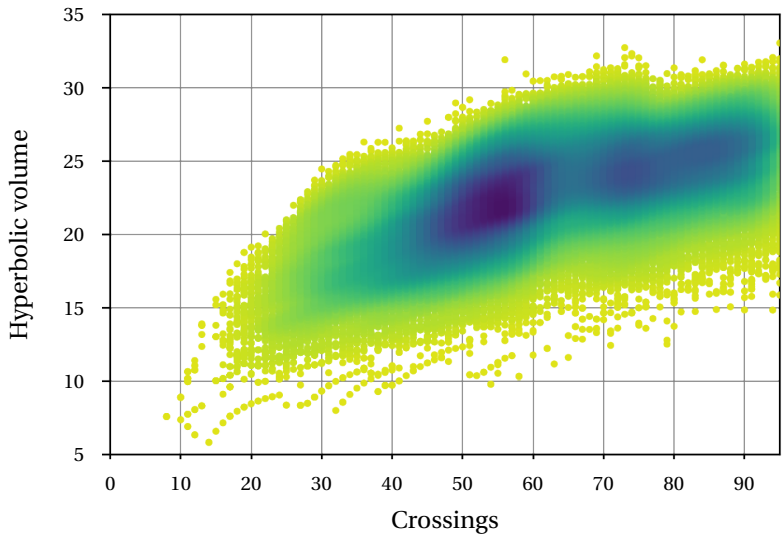
\begin{figure}
  \centering
  \begin{tikzpicture}[font=\scriptsize]
    \newcommand{\nmdsubfigurewidth}{9.75cm}
    \tikzset{axis label/.style={font=\footnotesize}}
    \tikzset{annotate plot/.style={font=\scriptsize}}
    \begin{scope}
      \input plots/friends_cross_hyp_medium.tex
    \end{scope}
  \end{tikzpicture}
  \caption{ For each 0-friend $K'$ in \ZF, this plot compares the
    crossing number of the simplest known diagram for $K'$ with the
    volume of the hyperbolic structure on $E_{K'}$.  The latter has
    mean 22.4, median 22.8, and $\sigma = 3.5$.  This is about 9\% more
    than the corresponding original knots in \PS; compare also
    Figure~\ref{fig: PS basic}.  }
  \label{fig: ZF basic}
\end{figure}

\begin{figure}
  \centering
  \begin{tikzpicture}[font=\scriptsize]
    \newcommand{\nmdsubfigurewidth}{5.7cm}
    \tikzset{axis label/.style={font=\footnotesize}}
    \tikzset{annotate plot/.style={font=\scriptsize}}
    \begin{scope}
      \input plots/hyp_dehn_0_filling.tex
    \end{scope}
    \begin{scope}[shift={(7.0, 0)}]
      \input plots/core_len_violin.tex
    \end{scope}
  \end{tikzpicture}
  \caption{%
    There are 93 knots $K$ in \PS\ whose exterior $E_K$ is hyperbolic
    but $Z_K$ is not hyperbolic.  At left, we see the portion of such
    exceptional knots declines exponentially in the number of
    crossings, ending at $0.0015\%$ ($1$ in $67{,}500$) for
    19-crossing knots.  For \PS, every time $Z_K$ was hyperbolic the
    core of the Dehn surgery solid torus was isotopic to a geodesic.
    The violin plot at right shows how the length of that geodesic
    changes as the number of crossings increases; the dotted line
    passes through the medians of each sample and the top and bottom
    lines are the maxes and mins.  }
  \label{fig: Z_K hyp}
\end{figure}
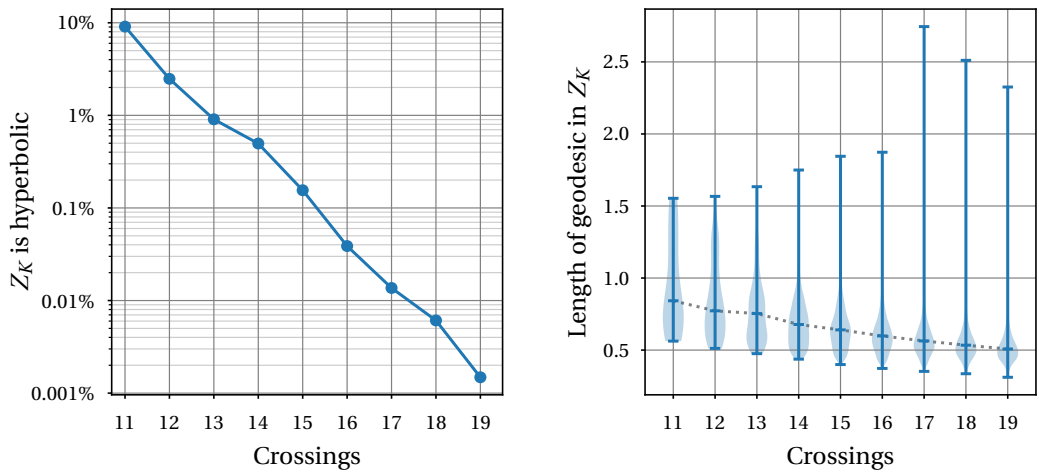

\subsection{Hyperbolic biases}
\label{sec: hyp rules}

Our approach in Section~\ref{sec: friend finder} may seem to be one of
convenience: we restrict to when $Z_K$ is hyperbolic and then using
that hyperbolic structure to guide the search for the 0-friend.

\FloatBarrier

We already noted that $Z_K$ is hyperbolic for all but 114 knots $K$ in
\PS; this includes the 21 knots in \PS\ where $E_K$ itself is not
hyperbolic.  Moreover, the portion of such exceptional knots appears
to decline exponentially in the number of crossings, as shown in the
left plot of Figure~\ref{fig: Z_K hyp}.  Thus insisting that $Z_K$ be
hyperbolic requires that we skip very few knots.

However, while any $K'$ that is a 0-friend of $K$ gives a
corresponding knot in $Z_K$ that generates $H_1(Z_K; \Z)$, in
Step~\ref{item: geod} we restrict to those knots in $Z_K$ that are
\emph{isotopic to geodesics} in its hyperbolic structure.  Remarkably,
whenever $Z_K$ is hyperbolic for $K$ in \PS, the core of the Dehn
surgery solid torus in $Z_K$ is indeed isotopic to a geodesic.  We
know of no geometric heuristic that suggests this pattern will not
continue; if anything, we would expect it to become more marked if
that were possible.  Thus, when looking for a 0-friend, considering
just the exteriors of geodesics in $Z_K$ seems to be a very mild
restriction, or possibly not one at all!

\begin{figure}
  \centering
  \begin{tikzpicture}[font=\scriptsize]
  \newcommand{\nmdsubfigurewidth}{6.5cm}
  \tikzset{axis label/.style={font=\footnotesize}}
  \tikzset{annotate plot/.style={font=\scriptsize}}
  \begin{scope}
    \input plots/orig_core_len.tex
  \end{scope}
  \begin{scope}[shift={(7, 0)}]
    \input plots/friend_core_len.tex
  \end{scope}  
  \end{tikzpicture}
  \caption{ A closed geodesic $\gamma$ in a hyperbolic manifold $M$
    has a length $L \in \R_{>0}$ and also a twist $\theta \in \R/(2\pi\Z)$
    recording the holonomy of going around $\gamma$; these are
    combined into its \emph{complex length}
    $L + i \theta \in \C/(2\pi i \Z)$.  The above shows the complex
    lengths for the corresponding geodesics in $Z_K$ for each pair of
    0-friends $K$ in \PS\ and $K'$ in \ZF. The complex length depends
    on the orientation of $M$ but not of $\gamma$; as the original
    choice of $K$ or $\Kbar$ in \cite{Burton2020} is somewhat
    arbitrary, we have normalized $\theta$ modulo $\pi$ rather than
    $2 \pi$ in both cases. }
  \label{fig: geodesics}
\end{figure}
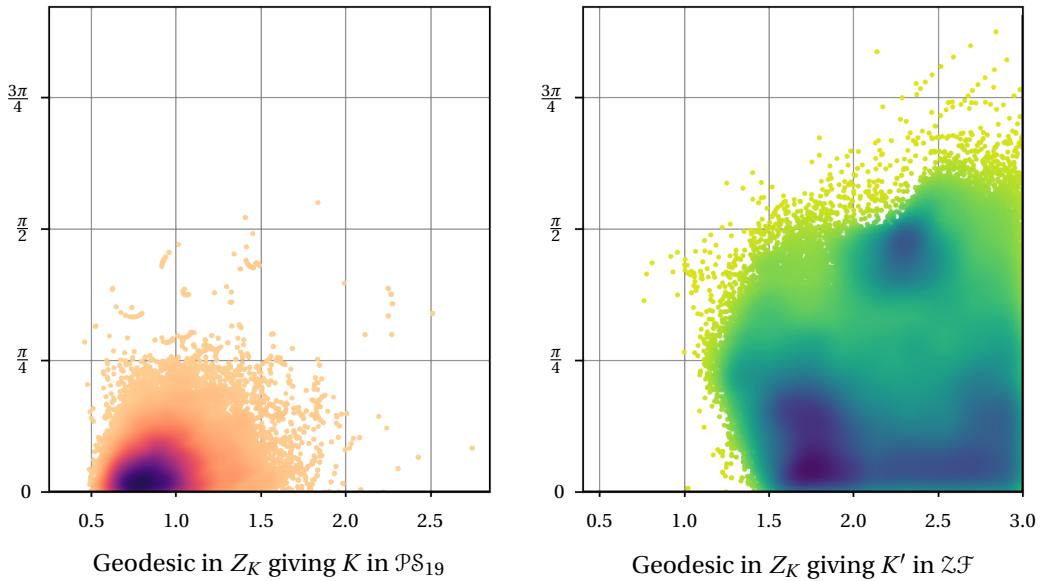

\begin{figure}
  \centering
  \begin{tikzpicture}[font=\scriptsize]
    \newcommand{\nmdsubfigurewidth}{10.0cm}
    \tikzset{axis label/.style={font=\footnotesize}}
    \tikzset{annotate plot/.style={font=\scriptsize}}
    \begin{scope}
      \input plots/geod_length_vs_crossings.tex
    \end{scope}
  \end{tikzpicture}
  \caption{This plot explores how the length of the geodesic in $Z_K$
    that gives the 0-friend $K'$ relates to the number of crossings of
    $K'$.  Specifically, the 0-friends were grouped by crossing as
    indicated in the upper-right and then a smoothed histogram is
    shown for the lengths of the corresponding geodesics.  The groups
    had between 2,659 and 25,316 knots each, where the size increases
    with the crossing number.  Here, we include all 0-friends found
    for knots in \PS, not just those in \ZF.}
  \label{fig: yield}
\end{figure}
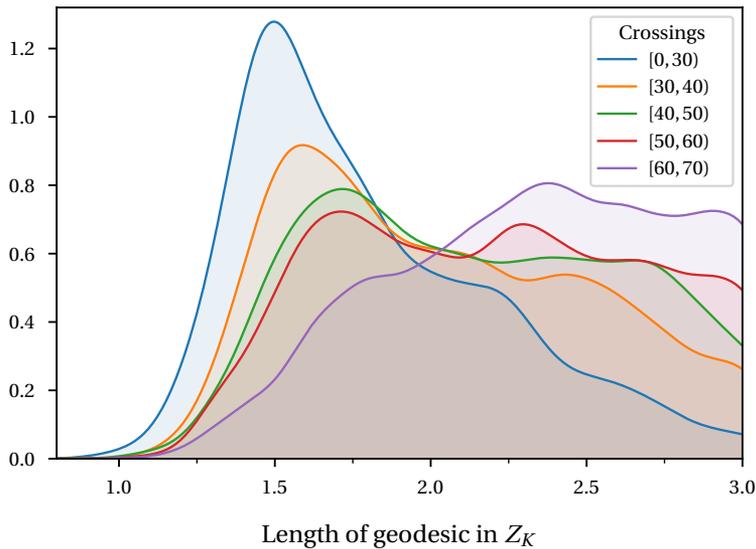

It remains to explain why we chose to limit our search to geodesics in
$Z_K$ of length at most 3.  For the whole of \PS, the length of the
core geodesic in $Z_K$ had mean $0.54$ and median $0.51$, with
$\sigma = 0.11$.  No geodesic was longer than $2.75$, and fewer than
$0.01\%$ were longer than $1.5$.  Moreover, as the number of crossings
increased, the median length decreased as shown in the right plot of
Figure~\ref{fig: Z_K hyp}, though a number of high outliers remain.
Thus a cutoff of $3$ is at least big enough to capture all of the
0-friends within \PS.  On the other hand, the right plot in
Figure~\ref{fig: geodesics} makes it clear there should be many
0-friends of knots in \PS\ coming from geodesics in $Z_K$ of length
more than 3.  The question is then how many ``small-to-medium''
0-friends we are missing, where concretely we mean those with a a
diagram of at most 50 crossings.  Figure~\ref{fig: yield} suggests that
we have found roughly 90\% of those with less than 30 crossings,
say 75\% of those with less than 40 crossings, and perhaps 60\% of
those with less than 50 crossings.  On the other hand, we likely have
much less than 50\% of those with crossing number in [50, 70).

It would be computationally feasible to raise the cutoff. As the
number of geodesics grows exponentially in the length, this gets
increasingly difficult, but considering those of length up to 4 (or
perhaps 5) should be possible in most of these examples.


\subsection{Hyperbolic aside} We next discuss an intriguing pattern in
the hyperbolic geometry of these examples; it is not directly related
to the main goal of using 0-friends to understand the knots in \PS, so
you may want to skip ahead to Section~\ref{sec: friend app}.  Our
motivation here was to try to predict in advance which
geodesics in $Z_K$ give $K'$ with a diagram with relatively few
crossings, but we gained little purchase on this goal.

We start by considering the difference in volume between $E_K$ and
$Z_K$.  By Thurston, one always has $\Vol(Z_K) < \Vol(E_K)$, and
presuming we are in the situation that $E_K$ is the complement of a
closed geodesic $\gamma$ in $Z_K$, the difference
$\Delta \Vol = \Vol(E_K) - \Vol(Z_K)$ is typically controlled by the
length $L$ of $\gamma$, with $\Delta \Vol \approx \frac{\pi}{2} L$.
The latter is always true asymptotically for sufficiently large Dehn
fillings on $E_K$ by \cite{NeumannZagier1985}.  Even for smaller Dehn
fillings, as long as one can deform the hyperbolic structure on $E_K$
to that of $Z_K$ via hyperbolic cone manifolds, it is usually a good
approximation because of Sch\"afli's formula; see e.g.~\cite[\S
A.3]{AgolStormThurston2007}.  Here, this pattern comes through very
strongly in Figure~\ref{fig: DVol}. While for any knot $K$ with a
diagram with $c$ crossings one has
$\Vol(E_K) < v_{\mathit{oct}} \cdot c < 3.67 c$, there is no converse to
this and Figure~\ref{fig: ZF basic} shows only a weak correlation
between $\Vol(E_{K'})$, and the crossing number of the diagram for
$K'$.

Less expected is the relationship between $\Vol(E_{K})$ and $L$ shown
in Figure~\ref{fig: huh?}, namely that $\Vol(E_{K})$ is roughly linear
in $1/L$.  Intriguingly, this holds only for the original $K$ in \PS,
not the 0-friends $K'$ in \ZF.  We now re-frame this phenomenon to
remove reference to $Z_K$.  Consider a horotorus $T$ that cuts off a
cusp neighborhood $C$ of $E_K$; the Euclidean structure on $T$ is
unique up to dilation.  Let $\mu$ and $\lambda$ be the Euclidean
translations corresponding to the meridian and longitude curves on
$T$; one defines the \emph{cusp shape} as $\tau_K = \lambda/\mu$,
where by convention orientations are chosen so $\Im(\tau_K) > 0$.  A
simple calculation using Lemma 4.2 of \cite{NeumannZagier1985}
suggests that $L \approx \Re(2 \pi i/\tau_K)$.  (To see this, apply
the lemma in the case $p = 0$ and $q = 1$, using that
$v(u) = \tau_K u + O(u^3)$ by their equation (12), and
$v(u) = 2 \pi i$ by their equation (32).)  This would imply
$2 \pi/L \approx \abs{\tau_K}^2/\Im(\tau_K)$, suggesting that we
compare $\Vol(E_K)$ with $\abs{\tau_K}^2/\Im(\tau_K)$.  Doing so
results in the striking plot in Figure~\ref{fig: cusp huh?}.  This
brings to mind \cite{DaviesJuhaszLackenbyTomasev2024}, which relates
$\Re(\tau_K)$ to $\sigma_K$, since plausibly slice knots have
$\sigma_K = 0$, but we do not know a direct connection to
Figure~\ref{fig: cusp huh?}.

\pagebreak

\enlargethispage{3cm}

\begin{figure}
  \centering
  \begin{tikzpicture}[font=\scriptsize]
    \newcommand{\nmdsubfigurewidth}{6.0cm}
    \tikzset{axis label/.style={font=\footnotesize}}
    \tikzset{annotate plot/.style={font=\scriptsize}}
    \begin{scope}
      \input plots/volume_change_vir.tex
    \end{scope}
    \begin{scope}[shift={(7.75, 0)}]
      \input plots/volume_change_ratio.tex
    \end{scope}
  \end{tikzpicture}

  \vspace{-0.2cm}
  
  \caption{At left, for each $K'$ in \ZF, we plot
    $\Delta \Vol = \Vol(E_{K'}) - \Vol(Z_{K'})$ against the length
    $L$ of the corresponding geodesic in $Z_{K'}$ and we do the same
    for its corresponding 0-friend $K$ in \PS. (From Figure~\ref{fig:
      geodesics}, we see that typically $L > 1.5$ for $K'$ and
    $L < 1.5$ for $K$.  The dotted line is the prediction
    $\Delta \Vol = \frac{\pi}{2}L$.  At right, the ratio
    $\Delta \Vol/ \pi L$ is plotted instead; compare this with the
    figure in \cite[\S A.3]{AgolStormThurston2007} which is for much
    smaller manifolds than those considered here.}
  \label{fig: DVol}
\end{figure}

\begin{figure}

  \vspace{-0.0cm}
  
  \centering
  \begin{tikzpicture}[font=\scriptsize]
    \newcommand{\nmdsubfigurewidth}{6.0cm}
    \tikzset{axis label/.style={font=\footnotesize}}
    \tikzset{annotate plot/.style={font=\scriptsize}}
    \begin{scope}
      \input plots/core_vs_volume_two_stories.tex
    \end{scope}
    \begin{scope}[shift={(7.5, 0)}]
      \input plots/one_on_core_vs_volume.tex
    \end{scope}
  \end{tikzpicture}

  \vspace{-0.2cm}

  \caption{At left, for each pair of 0-friends $K$ in \PS\ and $K'$ in
    \ZF, we plot the volume of $E_K$ and $E_K'$ against the length $L$
    of the corresponding geodesics in $Z_K = Z_K'$; the
    orange-red-purple blob where mostly $L \in [0.5, 1.5]$ corresponds
    to $K$ in \PS\ and the green-blue blob where mostly
    $L \in [1.5, 3]$ corresponds to $K'$ in \ZF.  Only the \PS\ blob
    shows much structure, and that structure becomes even more evident
    at right where the $x$-axis is now $1/L$ rather than $L$.}
  \label{fig: huh?}
\end{figure}
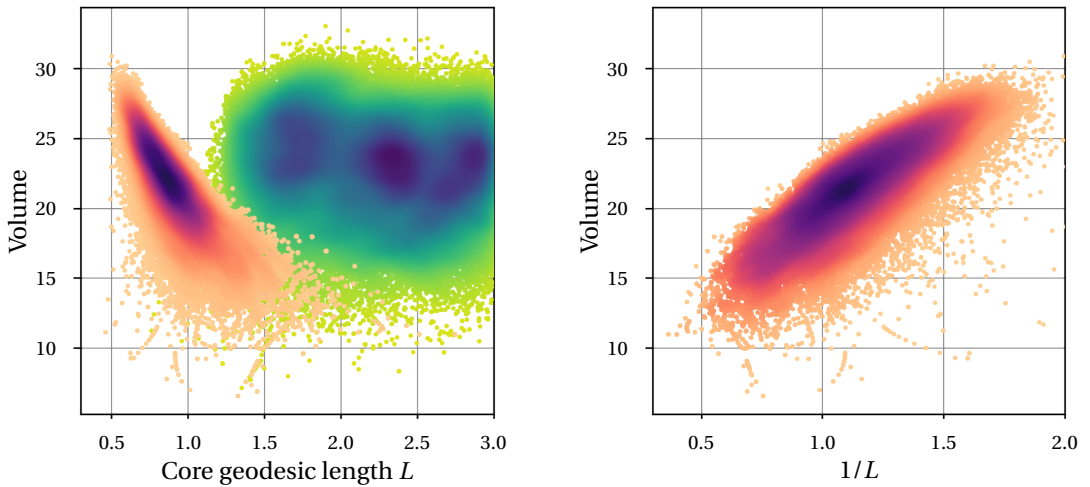

\FloatBarrier

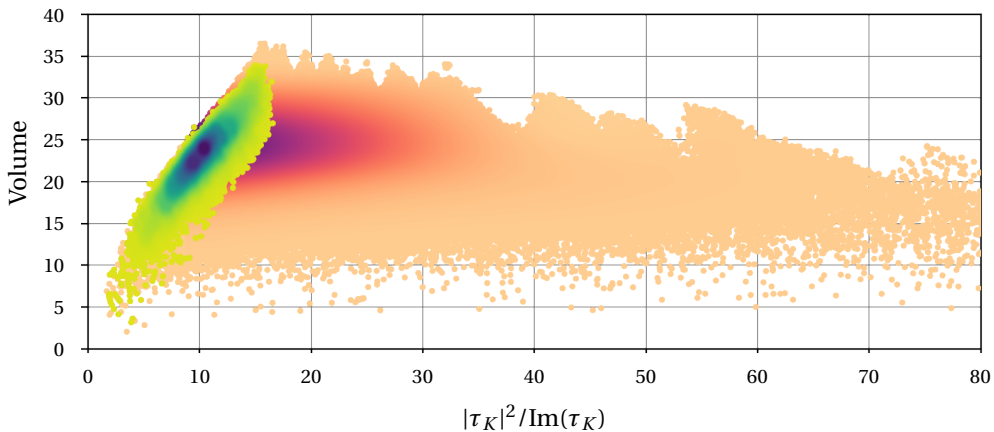
\begin{figure}[h]
  \centering
  \begin{tikzpicture}[font=\scriptsize]
    \newcommand{\nmdsubfigurewidth}{6.5cm}
    \tikzset{axis label/.style={font=\footnotesize}}
    \tikzset{annotate plot/.style={font=\scriptsize}}
    \begin{scope}
      \input plots/hyperbolic_aside.tex
    \end{scope}
  \end{tikzpicture}
  \caption{This plot compares the hyperbolic volume of $E_K$ with the
    shown quantity based on the cusp shape $\tau_K$.  The smaller
    green-blue blob on the left corresponds to hyperbolic knots in
    \PS\ with at most 16 crossings.  The larger peach-purple blob
    corresponds to all hyperbolic knots with at most 16 crossings. The
    densest part of the larger blob is hidden behind the smaller one.
  }
  \label{fig: cusp huh?}
\end{figure}

\subsection{Applications}
\label{sec: friend app}

To explain what the 36 exceptional 0-friends say about the
corresponding knots in \PS, we need to introduce the framework of
red-blue-green (RBG) links, following \cite[\S
3]{ManolescuPiccirillo2023}.

\begin{definition}
  \label{def: RBG link}
  An \emph{RBG link} $L = R \cup B \cup G \subset S^3$ is a 3-component
  link, with rational framings $r, b, g$
  respectively, such that there exist homeomorphisms $\psi_B \maps S^3_{r,
    g}(R \cup G) \to S^3$ and $\psi_G \maps S^3_{r, b}(R \cup B) \to
  S^3$ and such that $H_1(S^3_{r, b, g}(L); \Z) \cong \Z$.
\end{definition}
An RBG link describes a pair of 0-friends $K_B$ and $K_G$ where
$K_B$ is the image of $B$ under $\psi_B$ and $K_G$ is the image of $G$
under $\psi_G$; their common 0-surgery is $S^3_{r, b, g}(L)$.
Conversely, any pair of 0-friends can be encoded by an RBG link
\cite[Theorem~1.2]{ManolescuPiccirillo2023}.  An RBG link for the
0-friends in Figure~\ref{fig: besties} is shown in Figure~\ref{fig:
  RBG link}.

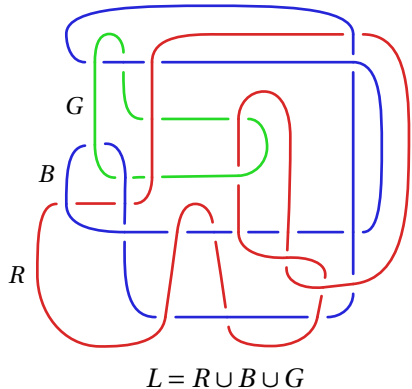
\begin{figure}
  \centering
  \begin{tikzpicture}[line cap=round, line join=round,
                      line width=1.01, font=\footnotesize]
    \definecolor{linkcolor0}{rgb}{0.85, 0.15, 0.15}
    \definecolor{linkcolor1}{rgb}{0.15, 0.15, 0.85}
    \definecolor{linkcolor2}{rgb}{0.15, 0.85, 0.15}
    \begin{scope}[scale=0.5]
      \input plots/RBG_smooth_gaps_20_0.3_width_800.tikz
      \node[left] at (0.05, 2) {$R$};
      \node[left] at (0.85, 4.7) {$B$};
      \node[left] at (1.575, 6.5) {$G$};
      \node[below] at (5, -0.2) {\small $L = R \cup B \cup G$};
    \end{scope}
  \end{tikzpicture}
  \caption{%
    An RBG link $L$ for the 0-friends $K$ and $K'$ in Figure~\ref{fig:
      besties}.  With $r = 1$, $b=g=0$, one has $K' = K_B$ and
    $K = K_G$.  The link $L$ was found from $(K, K')$ using the method
    of Section~\ref{sec: drill RBG}.  }
  \label{fig: RBG link}
\end{figure}

We say that an RBG link is \emph{super-special} when the two
sublinks $R \cup B$ and $R \cup G$ are both Hopf links
and the framings satisfy $b = g = 0$ and $r \in \Z$; this is a subset
of the special RBG links of \cite{ManolescuPiccirillo2023}.  The key
facts for us are:

\begin{theorem}[\cite{ManolescuPiccirillo2023}]
  \label{thm: super-special}
  Suppose $L$ is a super-special RBG link.  If $r$
  is even, then $K_B$ and $K_G$ have homeomorphic traces; if
  $r = 0$, their traces are diffeomorphic.  If instead $r$ is odd and
  their common 0-surgery has trivial mapping class group, then their
  traces are not homeomorphic.
\end{theorem}
    
\begin{proof}
This is in Sections 3 and 4 of \cite{ManolescuPiccirillo2023};
specifically, combine Theorems~3.7 and~3.13 with Lemmas~4.2 and~4.3,
noting Remark~3.8.
\end{proof}

\begin{theorem}[\cite{Nakamura2022, Ren2024}]
  \label{thm: naka ren}
  Suppose $L$ is a super-special RBG link.  Suppose
  $s_\F(K_G) \neq 0$, where $s_\F$ is the Rasmussen $s$-invariant over
  a field $\F$. Then both $K_B$ and $K_G$ are not smoothly slice.
\end{theorem}

\begin{proof}
As $s_\F(K_G) \neq 0$, of course $K_G$ is not smoothly slice. For $K_B$, we use
Theorem~3.13 of \cite{Nakamura2022}; this result is conditional on
Conjecture~2.15 in that paper, but this was established as
Corollary~1.5 of \cite{Ren2024}. If $K_B$ was smoothly slice, since $R$ is the
unknot, Theorem~3.13 of \cite{Nakamura2022} applies to say
$s_\F(K_G) = 0$, a contradiction.
\end{proof}

The main result of this section is:
\begin{theorem}
\label{thm: slice via friends}
The 25 knots listed in Table~\ref{tab: friends} are not smoothly slice.
\end{theorem}
The first two knots in Table~\ref{tab: friends} were already known not
to be smoothly slice by \cite{Piccirillo2020} in the case of the Conway knot
$K11n34$ and by unpublished work of Ciprian Manolescu, Lisa
Piccirillo, and Julius Zhang in the case of $K13n866$ using a similar
method to \cite{Piccirillo2020}.

\begin{table}
  \input tables/friends_1.tex
  \caption{Some 0-friends that admit super-special RBG
    links. In each case, we can conclude that the original knot in
    \PS\ is not smoothly slice.}
  \label{tab: friends}
\end{table}

\begin{proof}[Proof of Theorem~\ref{thm: slice via friends}]
For each of the 25 knots in the statment, the method of
Section~\ref{sec: drill RBG} below was used to find a super-special RBG
link that encoded the relationship between the original $K$ in \PS\
and its 0-friend $K'$; this is recorded in Table~\ref{tab: friends}.

For $K = 18nh_{00010270}$, as $r = 0$ in that RBG link we learn $K$
and $K'$ have diffeomorphic traces by Theorem~\ref{thm:
  super-special}.  By the Trace Embedding Lemma (see
e.g.~\cite[Lemma~3.5]{ManolescuPiccirillo2023}), it follows that $K$
and $K'$ have the same smooth slice status.  Now $K'$ is not smoothly
slice, as $\svec^{\Sq^1_o(K)}(K') \neq 0$, and so $K$ cannot be
either.
The remaining 24 knots in Table~\ref{tab: friends} are covered by
Theorem~\ref{thm: naka ren}.
\end{proof}

\subsection{Finding RBG links}
\label{sec: drill RBG}

Given 0-friends $K$ and $K'$ whose exteriors $E_K$ and $E_{K'}$ are
hyperbolic, we used the following method to search for a
super-special RBG link $L$ that expressed this fact.  As in
Section~\ref{sec: friend finder}, we will use
\cite{DunfieldObeidinRudd2024} to recover a diagram for $L$ from its
exterior as the very last step.  In particular, we searched for a
\3-cusped hyperbolic manifold $Y$ with cusps labelled red, blue,
and green with peripheral framings $(\mu_R, \nu_R)$, $(\mu_B,
\nu_B)$, and $(\mu_G, \nu_G)$ respectively so that:
\begin{enumerate}[label=(\alph*), ref=\alph*]
\item The Dehn filling $Y(\mu_R, \mu_B, \mu_G)$ is homeomorphic to $S^3$.

\item We have $Y(\nu_R, \cdotspaced, \nu_G) \cong E_K$ by a
  homeomorphism that takes $(\mu_B, \nu_B)$ to
  the standard framing $(\mu_K, \lambda_K)$
  for the knot exterior $E_K$.  Likewise, 
  $Y(\nu_R, \nu_B, \cdotspaced) \cong E_{K'}$ where $(\mu_G, \nu_G) 
  \mapsto (\mu_{K'}, \lambda_{K'})$.  (Note this implies  $Y(\nu_R,
  \nu_B, \nu_G)$ is the common $0$-surgery of $K$ and $K'$.)

\item
  \label{item: Hopf}
  Both the Dehn fillings $Y(\cdotspaced, \mu_B, \cdotspaced)$, and
  $Y(\cdotspaced, \cdotspaced, \mu_G)$ have fundamental group $\Z^2$.
  (This gives the Hopf link condition.)
  
\item
  \label{item: frame}
  The curve $\nu_B$ is $0$ in $H_1(Y(\mu_R,
  \cdotspaced, \mu_G); \Z)$ and $\nu_G$ is $0$ in $H_1(Y(\mu_R, \mu_B,
  \cdotspaced); \Z)$. (This gives that $b = g = 0$; note that
  $r \in \Z$ is automatic since $(\mu_R, \nu_R)$ is a framing for
  the red cusp.)
\end{enumerate}
You can check that these four conditions are equivalent to the cores
of the fillings in $Y(\mu_R, \mu_B, \mu_G)$ being the needed
super-special RBG link.  Our method for searching for such $Y$ was:

\begin{enumerate}
\item
  \label{item: dual}
  We looked at simple curves in the dual 1-skeleton of the
  triangulation of $E_K$ that were not null-homotopic.  For each
  such curve, we drilled it out to get a 2-cusped manifold $X$.  The
  cusp of $X$ coming from $E_K$ will be called blue and the new one
  green. The blue cusp inherits a framing $(\mu_B, \nu_B)$ from the
  framing $(\mu_K, \lambda_K)$ for $E_K$.
  
\item
  \label{item: blue-green}
  Check if the filling $X(\nu_B, \cdotspaced)$ is homeomorphic
  to $E_{K'}$.  If so, give the green cusp of $X$ the framing
  $(\mu_G, \nu_G)$ inherited from the framing $(\mu_{K'},
  \lambda_{K'})$ of $E_{K'}$ and proceed to the next step.
  Otherwise, go back to Step~\ref{item: dual}.  
  If we succeed, $X$ will be $Y(\mu_R, \cdotspaced, \cdotspaced)$.
  
\item
  \label{item: red meridan?}
  
  For  $X$ passing Step~\ref{item: blue-green}, we look
  through  simple curves in the dual 1-skeleton to its
  triangulation.  For each, we drilled it out to get a 3-cusped
  manifold $Y$.  On the new ``red'' cusp, we let $\nu_R$ (sic) be
  the curve so that filling along it gives $X$.
  
  Provided that requirement (\ref{item: Hopf}) holds, we consider
  $Y(\cdotspaced, \mu_B, \mu_G)$ and test the finitely-many
  hyperbolically plausible slopes $\mu_R$ to see if any
  $Y(\mu_R, \mu_B, \mu_G)$ is $S^3$.  If so, we check whether
  (\ref{item: frame}) is satisfied.
  
\item If we succeeded in Step~\ref{item: red meridan?}, apply
  \cite{DunfieldObeidinRudd2024} to get a link $L$ in $S^3$ whose
  complement in $Y$ where $(\mu_R, \mu_B, \mu_G)$ are the usual
  meridians for $L$.  This is the needed super-special RBG link.
\end{enumerate}
Note that this method is not symmetric in $K$ and $K'$, and
sometimes it was necessary to switch their roles to find $L$.

\begin{remark}
  As described in Section~\ref{sec: friend finder}, our 0-friend pairs
  each come from a pair of simple closed geodesics in the hyperbolic
  structure of their common 0-filling $Z_K$.  At least assuming these
  geodesics are disjoint, one could search for the red knot among
  other closed geodesics in $Z_K$.  Our experience was that this
  approach is not very effective, and it was often necessary to
  consider the above larger pool of candidates for the red knot.
  Moreover, unlike Section~\ref{sec: hyp rules}, we know of no
  heuristic reason why the method of Section~\ref{sec: drill RBG}
  should typically suffice to find such an $L$, even though it did in
  all 36 of the examples at hand.
\end{remark}

\subsection{Remaining mystery 0-friends}
\label{sec: mystery}

We now turn to an expanded version of Theorem~\ref{thm: mystery} from
the intro.
%
%
\begin{theorem}
  \label{thm: mystery 2}
  If $18nh_{00000601}$ is not smoothly slice, or if any of
  $\{16n68278$, $17nh_{0010647}$, $18nh_{00098198}\}$ are smoothly
  slice, then there is an exotic smooth 4-sphere.  All four knots are
  topologically slice. The knot $18nh_{00000601}$ is moreover smoothly
  slice in a homotopy 4-ball.
\end{theorem}

\begin{proof}
If the hypothesis holds, then from Table~\ref{tab: friends3} we have a
pair of 0-friends where one is smoothly slice and the other is not. This
gives an exotic smooth 4-sphere as discussed in e.g.~\cite[\S
1]{ManolescuPiccirillo2023}. To sketch the idea, suppose $K$ and $K'$
are $0$-friends where $K$ bounds a smooth 2-disk $D$ in $D^4$ and $K'$
is not smoothly slice.  Let $A = D^4 \setminus \nu(D)$ and $B$ be
$S^3 \times I$ with a 0-framed 4-dimensional 2-handle attached to
$K' \times \{1\}$.  Then $\partial A$ and one component of
$\partial B$ are both $Z_K \cong Z_{K'}$; gluing these together gives
a manifold $W$ with $\partial W = S^3$.  Now $W$ is a homotopy 4-ball
and $K' \subset \partial W$ bounds a disk in $W$, basically the core
of the 2-handle used to build $B$. As $K'$ is not smoothly slice, $W$ is
an exotic smooth 4-ball, as needed.

The three knots $\{16n68278$, $17nh_{0010647}$, $18nh_{00098198}\}$
are topologically slice since they all have $\Delta_K = 1$ \cite[\S
11.7]{FreedmanQuinn1990}.  For $K' = 18nh_{00000601}$, we know it is a
0-friend of a smoothly slice knot $K$.  By the preceding paragraph,
$K'$ is smoothly slice in a homotopy 4-ball $W$.  Since every homotopy
4-ball is \emph{topologically} standard \cite{FreedmanQuinn1990}, it
follows that $K'$ topologically slice in $D^4$, as needed.
\end{proof}

\begin{table}
  \input tables/friends_3.tex
  \caption{Five more 0-friends with super-special RBG links, where the
    original knot in \PS\ remains a mystery.  The two $K'$ for
    $17nh_{0010647}$ are distinct.}
  \label{tab: friends3}
\end{table}

%% file: plots/friends_cross_hyp_medium.tex
\begin{tikzoverlay*}[width=\nmdsubfigurewidth]{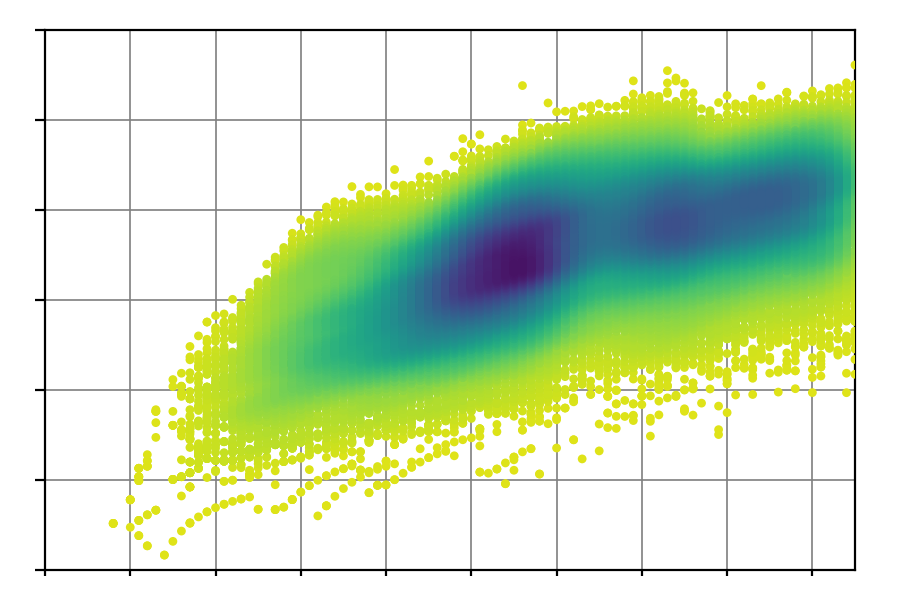}
  \draw (50.000000, -3) node[below, axis label] {Crossings};
  \draw (5.000000, 1.635802) node[below] {$0$};

  \draw (14.473684, 1.635802) node[below] {$10$};

  \draw (23.947368, 1.635802) node[below] {$20$};

  \draw (33.421053, 1.635802) node[below] {$30$};

  \draw (42.894737, 1.635802) node[below] {$40$};

  \draw (52.368421, 1.635802) node[below] {$50$};

  \draw (61.842105, 1.635802) node[below] {$60$};

  \draw (71.315789, 1.635802) node[below] {$70$};

  \draw (80.789474, 1.635802) node[below] {$80$};

  \draw (90.263158, 1.635802) node[below] {$90$};

  \draw (-7, 33.333333) node[rotate=90.0, axis label] {Hyperbolic volume};
  \draw (2.839506, 3.333333) node[left] {$5$};

  \draw (2.839506, 13.333333) node[left] {$10$};

  \draw (2.839506, 23.333333) node[left] {$15$};

  \draw (2.839506, 33.333333) node[left] {$20$};

  \draw (2.839506, 43.333333) node[left] {$25$};

  \draw (2.839506, 53.333333) node[left] {$30$};

  \draw (2.839506, 63.333333) node[left] {$35$};

  \begin{scope}[shift={(5.00000000, -6.66666667)},
                xscale=0.94736842, yscale=2.00000000]
  \end{scope}
\end{tikzoverlay*}

%% file: plots/hyp_dehn_0_filling.tex
\begin{tikzoverlay*}[width=\nmdsubfigurewidth]{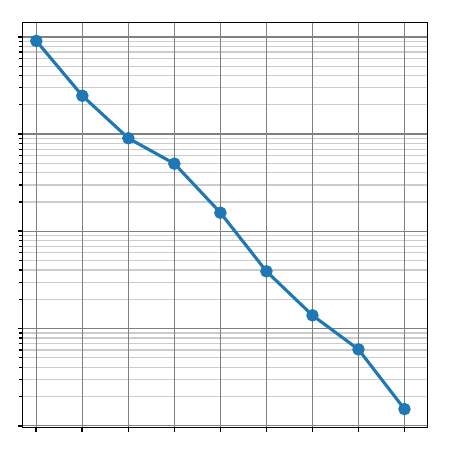}
  \draw (50.000000, -4.731481) node[below, axis label] {Crossings};
  \draw (8.068182, 2.453704) node[below] {$11$};

  \draw (18.295455, 2.453704) node[below] {$12$};

  \draw (28.522727, 2.453704) node[below] {$13$};

  \draw (38.750000, 2.453704) node[below] {$14$};

  \draw (48.977273, 2.453704) node[below] {$15$};

  \draw (59.204545, 2.453704) node[below] {$16$};

  \draw (69.431818, 2.453704) node[below] {$17$};

  \draw (79.659091, 2.453704) node[below] {$18$};

  \draw (89.886364, 2.453704) node[below] {$19$};

  \draw (-15.5, 50.000000) node[rotate=90.0, axis label] {$Z_K$ is hyperbolic};
  \begin{scope}[shift={(3.5, 0)}]
  \draw (0, 5.393112) node[left] {${0.001\%}$};

  \draw (0, 26.995653) node[left] {${0.01\%}$};

  \draw (0, 48.598195) node[left] {${0.1\%}$};

  \draw (0, 70.200737) node[left] {${1\%}$};

  \draw (0, 91.803278) node[left] {${10\%}$};
  \end{scope}

\end{tikzoverlay*}

%% file: plots/core_len_violin.tex
\begin{tikzoverlay*}[width=\nmdsubfigurewidth]{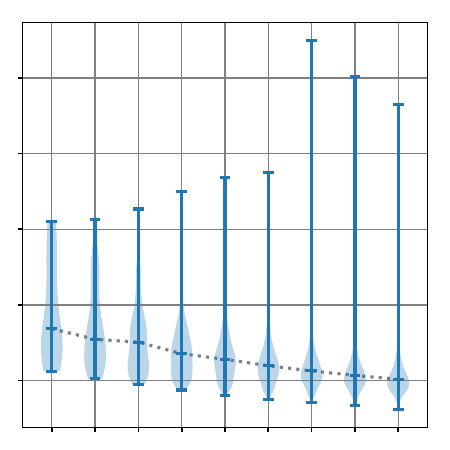}
  \draw (-10, 50.000000) node[rotate=90.0, axis label] {Length of geodesic in $Z_K$};
  \draw (50.000000, -4.70000) node[below, axis label] {Crossings};
  \begin{scope}[shift={(0, 2.5)}]
    \draw (11.497326, 0) node[below] {$11$};
    \draw (21.122995, 0) node[below] {$12$};
    \draw (30.748663, 0) node[below] {$13$};
    \draw (40.374332, 0) node[below] {$14$};
    \draw (50.000000, 0) node[below] {$15$};
    \draw (59.625668, 0) node[below] {$16$};
    \draw (69.251337, 0) node[below] {$17$};
    \draw (78.877005, 0) node[below] {$18$};
    \draw (88.502674, 0) node[below] {$19$};
  \end{scope}
  \begin{scope}[shift={(3.8, 0)}]
    \draw (0, 15.436669) node[left] {$0.5$};
    \draw (0, 32.245887) node[left] {$1.0$};
    \draw (0, 49.055104) node[left] {$1.5$};
    \draw (0, 65.864322) node[left] {$2.0$};
    \draw (0, 82.673540) node[left] {$2.5$};
  \end{scope}

  \begin{scope}[shift={(-94.38502674, -1.37254850)},
                xscale=9.62566845, yscale=33.61843522]
  \end{scope}
\end{tikzoverlay*}

%% file: plots/orig_core_len.tex
\begin{tikzoverlay*}[width=\nmdsubfigurewidth]{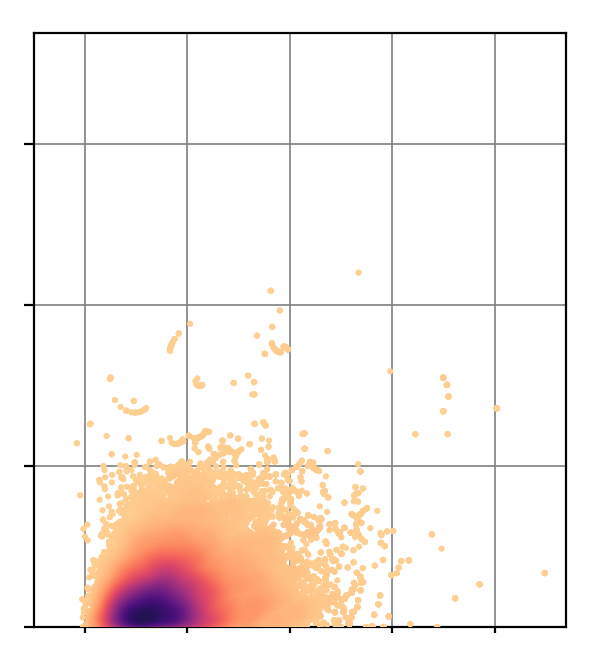}
  \draw (14.155172, 2.953704) node[below] {$0.5$};
  \draw (31.224138, 2.953704) node[below] {$1.0$};
  \draw (48.293103, 2.953704) node[below] {$1.5$};
  \draw (65.362069, 2.953704) node[below] {$2.0$};
  \draw (82.431034, 2.953704) node[below] {$2.5$};
  \draw (4, 5.500000) node[left] {$0$};
  \draw (4, 32.311868) node[left] {$\frac{\pi}{4}$};
  \draw (4, 59.123737) node[left] {$\frac{\pi}{2}$};
  \draw (4, 85.935605) node[left] {$\frac{3 \pi}{4}$};
    \draw (50, -5) node[below, axis label]
        {Geodesic in $Z_K$ giving $K$ in \PS};
  \begin{scope}[shift={(-2.91379310, 5.50000000)},
    xscale=34.13793103, yscale=34.13793103]
      \draw[black, line width=0.6pt] (0.25, 0.0035) -- +(2.6, 0);
  \end{scope}
\end{tikzoverlay*}

%% file: plots/friend_core_len.tex
\begin{tikzoverlay*}[width=\nmdsubfigurewidth]{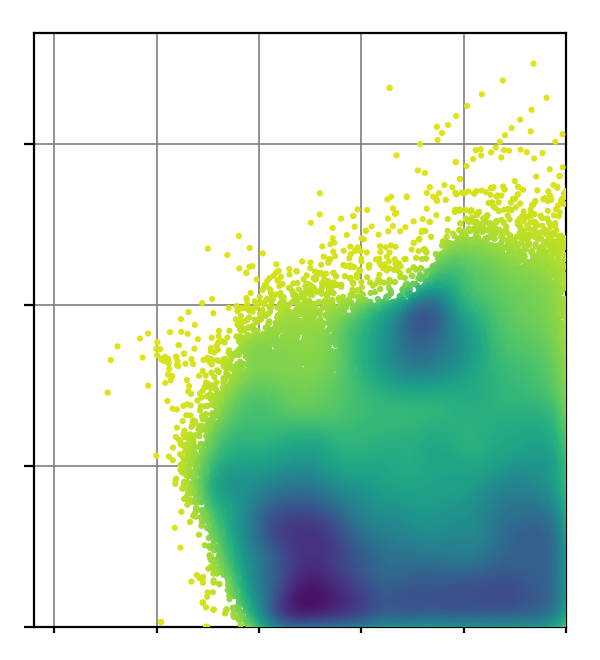}
  \draw (9.034483, 2.953704) node[below] {$0.5$};
  \draw (26.103448, 2.953704) node[below] {$1.0$};
  \draw (43.172414, 2.953704) node[below] {$1.5$};
  \draw (60.241379, 2.953704) node[below] {$2.0$};
  \draw (77.310345, 2.953704) node[below] {$2.5$};
  \draw (94.379310, 2.953704) node[below] {$3.0$};
  \draw (4, 5.500000) node[left] {$0$};
  \draw (4, 32.311868) node[left] {$\frac{\pi}{4}$};
  \draw (4, 59.123737) node[left] {$\frac{\pi}{2}$};
  \draw (4, 85.935605) node[left] {$\frac{3 \pi}{4}$};
  \draw (50, -5) node[below, axis label]
     {Geodesic in $Z_K$ giving $K'$ in \ZF};
  \begin{scope}[shift={(-8.03448276, 5.50000000)},
                xscale=34.13793103, yscale=34.13793103]
       \draw[black, line width=0.6pt] (0.4, 0.0035) -- ++(2.597, 0) --
       ++(0, 2.85) ;
  \end{scope}
\end{tikzoverlay*}

%% file: plots/geod_length_vs_crossings.tex
\begin{tikzoverlay*}[width=\nmdsubfigurewidth]{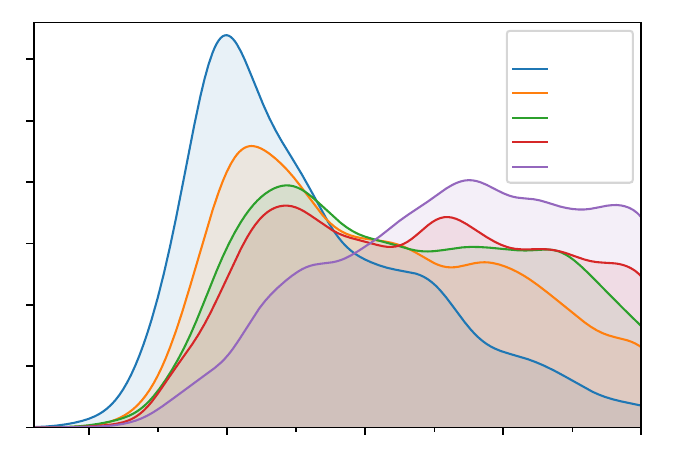}
  \draw (50.000000, -4) node[below, axis label] {Length of geodesic in $Z_K$};
  \begin{scope}[shift={(0, 0.7)}]
    \draw (13.181818, 1.172840) node[below] {$1.0$};
    \draw (33.636364, 1.172840) node[below] {$1.5$};
    \draw (54.090909, 1.172840) node[below] {$2.0$};
    \draw (74.545455, 1.172840) node[below] {$2.5$};
    \draw (95.000000, 1.172840) node[below] {$3.0$};
  \end{scope}
  \begin{scope}[shift={(0.7, 0)}]
    \draw (2.839506, 3.333333) node[left] {$0.0$};
    \draw (2.839506, 12.424242) node[left] {$0.2$};
    \draw (2.839506, 21.515152) node[left] {$0.4$};
    \draw (2.839506, 30.606061) node[left] {$0.6$};
    \draw (2.839506, 39.696970) node[left] {$0.8$};
    \draw (2.839506, 48.787879) node[left] {$1.0$};
    \draw (2.839506, 57.878788) node[left] {$1.2$};
  \end{scope}

  \begin{scope}[shift={(21, 16.5)}]
    \draw (56.503819, 43.25) node[right] {Crossings};
    \begin{scope}[shift={(0.5, -1.09)}]
      \draw (59.751389, 40.97) node[right] {$[0, 30)$};
      \draw (59.751389, 37.34) node[right] {$[30, 40)$};
      \draw (59.751389, 33.73) node[right] {$[40, 50)$};
      \draw (59.751389, 30.1) node[right] {$[50, 60)$};
      \draw (59.751389, 26.48) node[right] {$[60, 70)$};
    \end{scope}
  \end{scope}
  \begin{scope}[shift={(-27.72727273, 3.33333333)},
                xscale=40.90909091, yscale=45.45454545]
  \end{scope}
\end{tikzoverlay*}

%% file: plots/volume_change_vir.tex
\begin{tikzoverlay*}[width=\nmdsubfigurewidth]{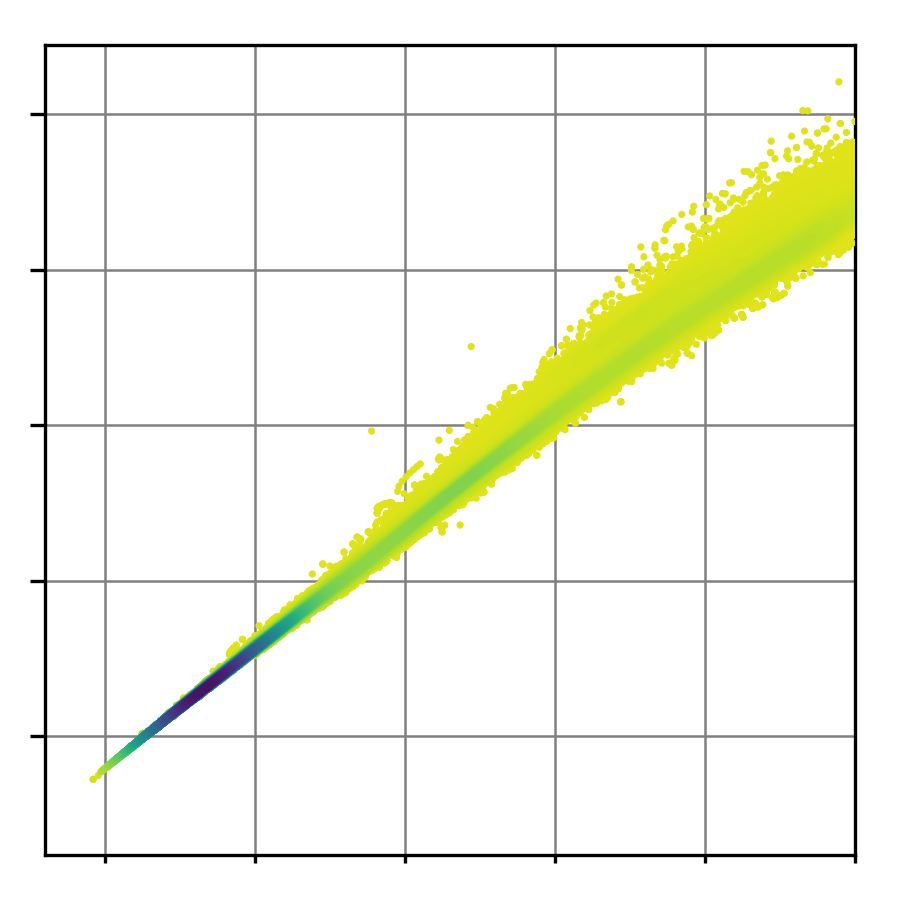}
  \draw (50.000000, -4.462963) node[below, axis label] {Length $L$ of geodesic in $Z_K = Z_{K'}$};
  \draw (11.666667, 2.453704) node[below] {$0.5$};

  \draw (28.333333, 2.453704) node[below] {$1.0$};

  \draw (45.000000, 2.453704) node[below] {$1.5$};

  \draw (61.666667, 2.453704) node[below] {$2.0$};

  \draw (78.333333, 2.453704) node[below] {$2.5$};

  \draw (95.000000, 2.453704) node[below] {$3.0$};

  \draw (-7, 50.000000) node[rotate=90.0, axis label] {$\Delta \Vol$};
  \draw (1.759259, 18.221028) node[left] {$1$};

  \draw (1.759259, 35.488031) node[left] {$2$};

  \draw (1.759259, 52.755033) node[left] {$3$};

  \draw (1.759259, 70.022035) node[left] {$4$};

  \draw (1.759259, 87.289038) node[left] {$5$};

  \begin{scope}[shift={(-5.00000000, 0.95402594)},
                xscale=33.33333333, yscale=17.26700232]
      \draw[line width=0.5pt, line cap=round, dashed, color=black!50]
            (0.305000, 0.47909287967244346) -- (3.000000, 4.71238898038469);
  \end{scope}
\end{tikzoverlay*}

%% file: plots/volume_change_ratio.tex
\begin{tikzoverlay*}[width=\nmdsubfigurewidth]{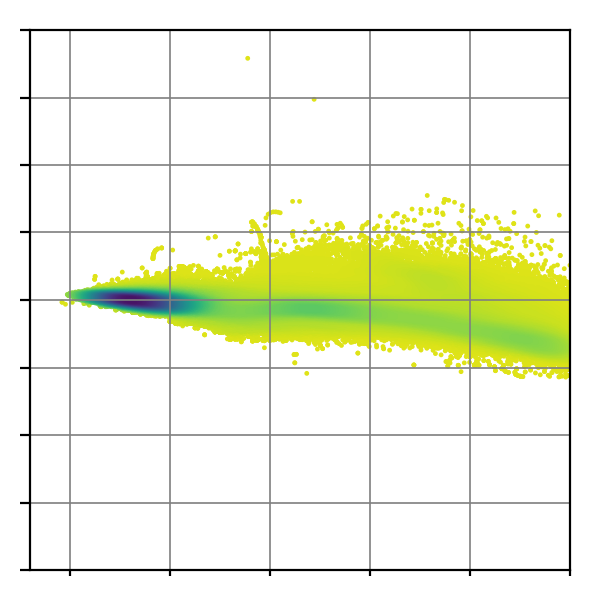}
  \draw (50.000000, -5.462963) node[below, axis label] {Length $L$ of geodesic in $Z_K = Z_{K'}$};
  \draw (11.666667, 2.453704) node[below] {$0.5$};

  \draw (28.333333, 2.453704) node[below] {$1.0$};

  \draw (45.000000, 2.453704) node[below] {$1.5$};

  \draw (61.666667, 2.453704) node[below] {$2.0$};

  \draw (78.333333, 2.453704) node[below] {$2.5$};

  \draw (95.000000, 2.453704) node[below] {$3.0$};

  \draw (-12.768519, 50.000000) node[rotate=90.0, axis label] {$\Delta \Vol/\pi L$};
  \draw (1.759259, 5.000000) node[left] {$0.30$};

  \draw (1.759259, 16.250000) node[left] {$0.35$};

  \draw (1.759259, 27.500000) node[left] {$0.40$};

  \draw (1.759259, 38.750000) node[left] {$0.45$};

  \draw (1.759259, 50.000000) node[left] {$0.50$};

  \draw (1.759259, 61.250000) node[left] {$0.55$};

  \draw (1.759259, 72.500000) node[left] {$0.60$};

  \draw (1.759259, 83.750000) node[left] {$0.65$};

  \draw (1.759259, 95.000000) node[left] {$0.70$};

  \begin{scope}[shift={(-5.00000000, -62.50000000)},
                xscale=33.33333333, yscale=225.00000000]
  \end{scope}
\end{tikzoverlay*}

%% file: plots/core_vs_volume_two_stories.tex
\begin{tikzoverlay*}[width=\nmdsubfigurewidth]{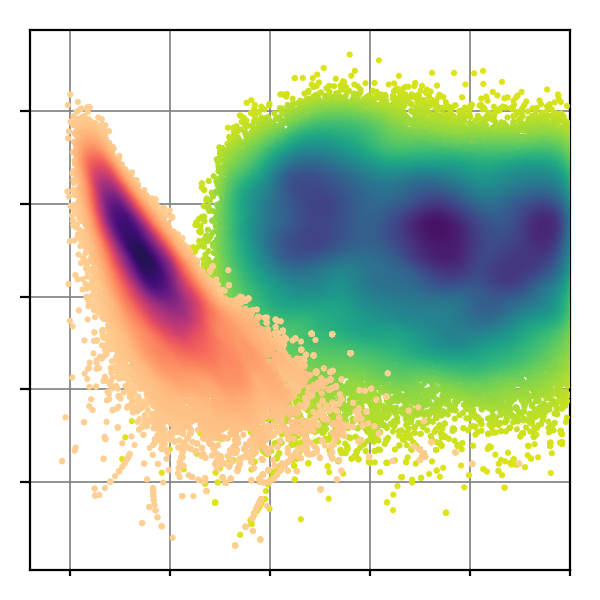}
  \draw (50.000000, -3.111111) node[below, axis label] {Core geodesic length $L$};
  \draw (11.666667, 2.453704) node[below] {$0.5$};

  \draw (28.333333, 2.453704) node[below] {$1.0$};

  \draw (45.000000, 2.453704) node[below] {$1.5$};

  \draw (61.666667, 2.453704) node[below] {$2.0$};

  \draw (78.333333, 2.453704) node[below] {$2.5$};

  \draw (95.000000, 2.453704) node[below] {$3.0$};

  \draw (-9, 50.000000) node[rotate=90.0, axis label] {Volume};
  \draw (1.759259, 19.630485) node[left] {$10$};

  \draw (1.759259, 35.089942) node[left] {$15$};

  \draw (1.759259, 50.549399) node[left] {$20$};

  \draw (1.759259, 66.008856) node[left] {$25$};

  \draw (1.759259, 81.468313) node[left] {$30$};

  \begin{scope}[shift={(-5.00000000, -11.28842928)},
                xscale=33.33333333, yscale=3.09189141]
  \end{scope}
\end{tikzoverlay*}

%% file: plots/one_on_core_vs_volume.tex
\begin{tikzoverlay*}[width=\nmdsubfigurewidth]{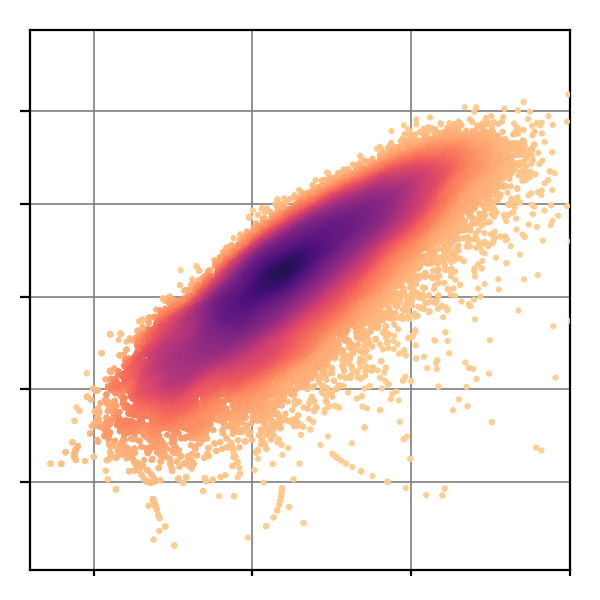}
  \draw (50.000000, -3.11111) node[below, axis label] {$1/L$};
  \draw (15.588235, 2.453704) node[below] {$0.5$};

  \draw (42.058824, 2.453704) node[below] {$1.0$};

  \draw (68.529412, 2.453704) node[below] {$1.5$};

  \draw (95.000000, 2.453704) node[below] {$2.0$};

  \draw (-9, 50.000000) node[rotate=90.0, axis label] {Volume};
  \draw (1.759259, 19.630485) node[left] {$10$};

  \draw (1.759259, 35.089942) node[left] {$15$};

  \draw (1.759259, 50.549399) node[left] {$20$};

  \draw (1.759259, 66.008856) node[left] {$25$};

  \draw (1.759259, 81.468313) node[left] {$30$};

  \begin{scope}[shift={(-10.88235294, -11.28842928)},
                xscale=52.94117647, yscale=3.09189141]
  \end{scope}
\end{tikzoverlay*}

%% file: plots/hyperbolic_aside.tex
\begin{tikzoverlay*}[width=0.90\textwidth]{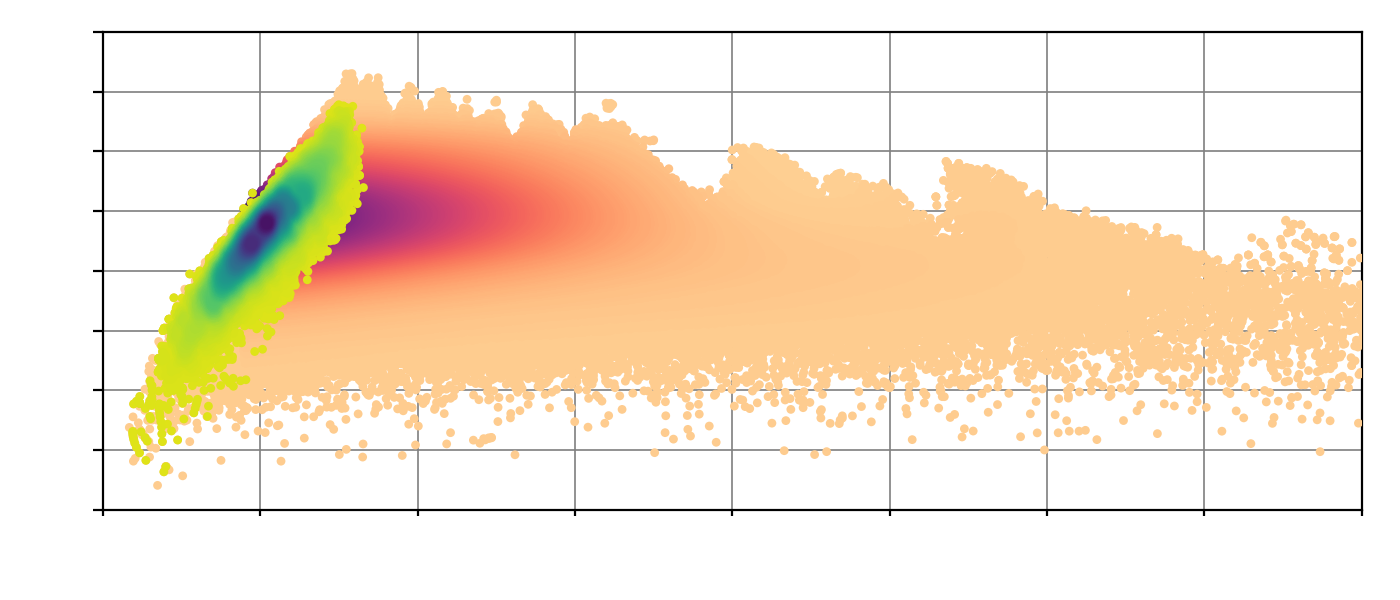}
  \draw (52.314484, 1.857143) node[below, axis label] {$\abs{\tau_K}^2/\Im(\tau_K)$};
  \draw (7.361111, 5.365079) node[below] {$0$};

  \draw (18.599454, 5.365079) node[below] {$10$};

  \draw (29.837798, 5.365079) node[below] {$20$};

  \draw (41.076141, 5.365079) node[below] {$30$};

  \draw (52.314484, 5.365079) node[below] {$40$};

  \draw (63.552827, 5.365079) node[below] {$50$};

  \draw (74.791171, 5.365079) node[below] {$60$};

  \draw (86.029514, 5.365079) node[below] {$70$};

  \draw (97.267857, 5.365079) node[below] {$80$};

  \draw (0.2, 25.513889) node[rotate=90.0, axis label] {Volume};
  \draw (5.972222, 6.456349) node[left] {$0$};

  \draw (5.972222, 10.720734) node[left] {$5$};

  \draw (5.972222, 14.985119) node[left] {$10$};

  \draw (5.972222, 19.249504) node[left] {$15$};

  \draw (5.972222, 23.513889) node[left] {$20$};

  \draw (5.972222, 27.778274) node[left] {$25$};

  \draw (5.972222, 32.042659) node[left] {$30$};

  \draw (5.972222, 36.307044) node[left] {$35$};

  \draw (5.972222, 40.571429) node[left] {$40$};

  \begin{scope}[shift={(7.36111111, 6.45634921)},
                xscale=1.12383433, yscale=0.85287698]
  \end{scope}
\end{tikzoverlay*}

%% file: plots/RBG_smooth_gaps_20_0.3_width_800.tikz
  \begin{scope}[color=linkcolor0]
    \draw (6.62, 2.20) .. controls (6.60, 1.79) and (7.09, 1.62) .. (7.58, 1.67);
    \draw (7.58, 1.67) .. controls (7.84, 1.70) and (8.11, 1.73) .. (8.37, 1.76);
    \draw (8.37, 1.76) .. controls (9.72, 1.91) and (9.81, 3.64) .. 
          (9.83, 5.16) .. controls (9.85, 6.66) and (9.87, 8.40) .. (8.63, 8.40);
    \draw (8.10, 8.40) .. controls (7.01, 8.40) and (5.92, 8.40) .. 
          (4.82, 8.40) .. controls (4.01, 8.41) and (3.09, 8.36) .. (3.09, 7.66);
    \draw (3.09, 7.66) .. controls (3.09, 7.15) and (3.09, 6.65) .. (3.09, 6.15);
    \draw (3.09, 6.15) .. controls (3.09, 5.63) and (3.09, 5.12) .. (3.09, 4.60);
    \draw (3.09, 4.60) .. controls (3.09, 4.25) and (2.95, 3.90) .. (2.64, 3.90);
    \draw (2.11, 3.90) .. controls (1.78, 3.90) and (1.44, 3.90) .. (1.11, 3.90);
    \draw (0.58, 3.89) .. controls (0.08, 3.89) and (0.08, 2.85) .. 
          (0.08, 2.01) .. controls (0.08, 1.01) and (0.75, 0.11) .. 
          (1.71, 0.10) .. controls (2.49, 0.09) and (3.35, 0.19) .. (3.44, 0.88);
    \draw (3.44, 0.88) .. controls (3.55, 1.63) and (3.65, 2.39) .. (3.75, 3.14);
    \draw (3.75, 3.14) .. controls (3.80, 3.50) and (3.90, 3.89) .. 
          (4.23, 3.89) .. controls (4.48, 3.89) and (4.64, 3.66) .. (4.68, 3.40);
    \draw (4.76, 2.88) .. controls (4.85, 2.30) and (4.94, 1.72) .. (5.02, 1.14);
    \draw (5.10, 0.62) .. controls (5.17, 0.17) and (5.73, 0.11) .. 
          (6.24, 0.11) .. controls (6.79, 0.11) and (7.34, 0.36) .. (7.43, 0.88);
    \draw (7.43, 0.88) .. controls (7.47, 1.06) and (7.50, 1.25) .. (7.54, 1.44);
    \draw (7.63, 1.93) .. controls (7.71, 2.34) and (7.16, 2.49) .. (6.63, 2.46);
    \draw (6.63, 2.46) .. controls (6.01, 2.43) and (5.36, 2.61) .. (5.36, 3.14);
    \draw (5.36, 3.14) .. controls (5.36, 3.55) and (5.36, 3.96) .. (5.36, 4.38);
    \draw (5.36, 4.91) .. controls (5.36, 5.32) and (5.36, 5.73) .. (5.36, 6.15);
    \draw (5.36, 6.15) .. controls (5.36, 6.55) and (5.65, 6.89) .. 
          (6.04, 6.88) .. controls (6.54, 6.87) and (6.72, 6.26) .. 
          (6.72, 5.68) .. controls (6.72, 4.83) and (6.70, 3.99) .. (6.66, 3.14);
    \draw (6.66, 3.14) .. controls (6.65, 2.98) and (6.65, 2.82) .. (6.64, 2.66);
  \end{scope}
  \begin{scope}[color=linkcolor1]
    \draw (2.39, 4.59) .. controls (2.39, 5.05) and (2.24, 5.52) .. (1.86, 5.49);
    \draw (1.33, 5.44) .. controls (0.84, 5.39) and (0.84, 4.57) .. (0.84, 3.89);
    \draw (0.84, 3.89) .. controls (0.84, 3.27) and (1.64, 3.14) .. (2.37, 3.14);
    \draw (2.37, 3.14) .. controls (2.74, 3.14) and (3.11, 3.14) .. (3.49, 3.14);
    \draw (4.02, 3.14) .. controls (4.25, 3.14) and (4.49, 3.14) .. (4.72, 3.14);
    \draw (4.72, 3.14) .. controls (4.87, 3.14) and (5.02, 3.14) .. (5.17, 3.14);
    \draw (5.62, 3.14) .. controls (5.88, 3.14) and (6.14, 3.14) .. (6.40, 3.14);
    \draw (6.93, 3.14) .. controls (7.32, 3.14) and (7.71, 3.14) .. (8.10, 3.14);
    \draw (8.63, 3.14) .. controls (9.12, 3.14) and (9.12, 4.41) .. 
          (9.12, 5.40) .. controls (9.12, 6.43) and (9.12, 7.66) .. (8.37, 7.66);
    \draw (8.37, 7.66) .. controls (6.70, 7.66) and (5.03, 7.66) .. (3.36, 7.66);
    \draw (2.87, 7.66) .. controls (2.69, 7.66) and (2.52, 7.66) .. (2.34, 7.66);
    \draw (2.34, 7.66) .. controls (2.17, 7.66) and (1.99, 7.66) .. (1.82, 7.66);
    \draw (1.33, 7.66) .. controls (1.00, 7.66) and (0.84, 8.04) .. 
          (0.84, 8.41) .. controls (0.84, 9.16) and (2.97, 9.16) .. 
          (4.61, 9.16) .. controls (6.24, 9.16) and (8.37, 9.16) .. (8.37, 8.40);
    \draw (8.37, 8.40) .. controls (8.37, 8.23) and (8.37, 8.05) .. (8.37, 7.88);
    \draw (8.37, 7.39) .. controls (8.37, 5.98) and (8.37, 4.56) .. (8.37, 3.14);
    \draw (8.37, 3.14) .. controls (8.37, 2.77) and (8.37, 2.40) .. (8.37, 2.03);
    \draw (8.37, 1.50) .. controls (8.37, 1.15) and (8.06, 0.88) .. (7.70, 0.88);
    \draw (7.17, 0.88) .. controls (6.47, 0.88) and (5.77, 0.88) .. (5.06, 0.88);
    \draw (5.06, 0.88) .. controls (4.61, 0.88) and (4.16, 0.88) .. (3.71, 0.88);
    \draw (3.18, 0.88) .. controls (2.37, 0.88) and (2.36, 1.94) .. (2.37, 2.88);
    \draw (2.37, 3.37) .. controls (2.37, 3.55) and (2.38, 3.72) .. (2.38, 3.90);
    \draw (2.38, 3.90) .. controls (2.38, 4.13) and (2.38, 4.36) .. (2.39, 4.59);
  \end{scope}
  \begin{scope}[color=linkcolor2]
    \draw (5.62, 6.15) .. controls (5.95, 6.15) and (6.11, 5.77) .. 
          (6.11, 5.40) .. controls (6.11, 4.98) and (5.77, 4.65) .. (5.36, 4.64);
    \draw (5.36, 4.64) .. controls (4.69, 4.63) and (4.02, 4.62) .. (3.36, 4.61);
    \draw (2.88, 4.60) .. controls (2.79, 4.60) and (2.69, 4.60) .. (2.60, 4.59);
    \draw (2.12, 4.59) .. controls (1.75, 4.58) and (1.59, 5.03) .. (1.59, 5.46);
    \draw (1.59, 5.46) .. controls (1.59, 6.19) and (1.59, 6.93) .. (1.59, 7.66);
    \draw (1.59, 7.66) .. controls (1.59, 8.02) and (1.66, 8.41) .. 
          (1.97, 8.41) .. controls (2.20, 8.41) and (2.34, 8.17) .. (2.34, 7.92);
    \draw (2.34, 7.39) .. controls (2.34, 6.81) and (2.34, 6.15) .. (2.83, 6.15);
    \draw (3.36, 6.15) .. controls (3.94, 6.15) and (4.51, 6.15) .. (5.09, 6.15);
  \end{scope}

%% file: tables/friends_1.tex
\begin{center}
\small
\begin{tabular}{llrl}
  \toprule
  knot & DT code of RBG link & $r$ &  $K'$ notes\\
  \midrule
  $K11n34$ & vcdemhUSFVDCoKIMeJaRgtNpBql & $0$ & $s_{\F_3} = 2$\\
  $K13n866$ & tcbiisTAbrNEomgQKhFiLdJpc & $-1$ & $s_{\F_3} = 2$\\
  $K15n25044$ & xcdhlUkRJgNxsVpQMLFWhdBCoAITe & $-4$ & $s_{\F_3} = 2$\\
  $16n180537$ & rcbdlGCARMoFBkQhiEPlDNJ & $1$ & $s_{\F_3} = 2$\\
  $16n74539$ & tcbiiEnTOBDlRpqMfKgaCsjHi & $-1$ & $s_{\F_3} = -2$\\
  $17nh_{0001844}$ & scbfkjFinLRAQObhDpeSKmGc & $1$ & $s_{\F_3} = 2$\\
  $17nh_{0002715}$ & xcdoelHtPXVRDuSEaqFwCGmOibKnJ & $0$ & $s_{\F_3} = -2$\\
  $17nh_{0212094}$ & sccdlrcesbjkOGlmifpqHnda & $-1$ & $s_{\F_3} = 2$\\
  $17nh_{0212095}$ & qccdjenaQcmPOGldfjbkHI & $-1$ & $s_{\F_3} = 2$\\
  $18nh_{00010270}$ & zcdkkgNKtZYuaQpCsrVDjXmwbOhlIFE & $0$ & $\svec^{\Sq^1_o(K)} \neq 0$ \\
  $18nh_{00166702}$ & rcehemldoJRhqpEIFnabcgK & $0$ & $s_{\F_3} = -2$\\
  $18nh_{00610378}$ & tccelDPAECTJSQGMNrKfHBlIO & $1$ & $s_{\F_3} = 2$\\
  $18nh_{00610381}$ & ucecmKDGBAlRCOsEhTpqInFuMj & $-1$ & $s_{\F_3} = 2$\\
  $19nh_{000002588}$ & ycblkuLFkRYOmWeSpAvNThIxCbGqDJ & $-1$ & $s_{\F_3} = -2$\\
  $19nh_{000003154}$ & rcgfeClHoJPdBNqFaRMGkEi & $3$ & $s_{\F_3} = -2$\\
  $19nh_{000003570}$ & uciffKlPtSBhoQUNrEdgCfIamJ & $-5$ & $s_{\F_3} = 2$\\
  $19nh_{000032808}$ & xchigIfEOtUkwXmAsPjCVhQGdbNLR & $0$ & $s_{\F_3} = -2$\\
  $19nh_{000076489}$ & wcbjkDftqRSMWovpGLKujcbEAhiN & $-1$ & $s_{\F_3} = -2$\\
  $19nh_{000018991}$ & wcdijMJrFTLWiqgNUDaSvBhPkEoC & $0$ & $s_{\F_3} = -2$\\
  $19nh_{000066839}$ & xcbnhMqXegIdTFOPcBrSWlavUHnJK & $-1$ & $s_{\F_3} = -2$\\
  $19nh_{000066841}$ & ycbohEgYwUtsRxpMoJLVkiCAfbNdqh & $1$ & $s_{\F_3} = -2$\\
  $19nh_{000177115}$ & zcdnhRqPoGwUVasCnTxyJzibDMEHflk & $0$ & $s_{\F_3} = 2$\\
  $19nh_{000177116}$ & zcdscVuopGhWmqLZrFedcsiknXyTJaB & $0$ & $s_{\F_3} = -2$\\
  $19nh_{001336127}$ & ycejjDCNAYSLrXtqWvPBOhkMuifGJE & $-1$ & $s_{\F_3} = -2$\\
  $19nh_{002457201}$ & ycdkjOPHexuMLQyrcGVWASkIBNfTDj & $0$ & $s_{\F_3} = -2$\\
  \bottomrule
\end{tabular}
\end{center}

%% file: tables/friends_3.tex
\begin{center}
\small
\begin{tabular}{llrl}
  \toprule
  knot & DT code of RBG link & $r$ &  $K'$ notes\\
  \midrule
  $18nh_{00000601}$ & ycjkdnhQyUtMsaVweFIRCXOBgLJDkP & $1$ & ribbon \\
  \\
  $16n68278$ & tcbjhTDJBgQmsOLhPcekIRFnA & $1$ & $\svectil_c \neq 0$\\
  $17nh_{0010647}$ & scdifSmJkOBcNQPFrdAHEIgl & $-1$ & $\svectil_c \neq 0$\\
  $17nh_{0010647}$ & xcdhlXqLTCvumPWsoAikRHbnEDfJg & $-1$ & $\svectil_c \neq 0$\\
  $18nh_{00098198}$ & rcbgiHMRfPlnJAIOdBcqEKg & $-1$ & $\svectil_c \neq 0$\\
  \bottomrule
\end{tabular}
\end{center}

%% file: nonslice_scraps.tex
\section{A few more non-slice knots}
\label{sec: nonslice scraps}

\begin{figure}
  \centering
  \begin{tikzpicture}[
  line cap=round,
  line join=round,
  line width=1.01,
  font=\small,
  mid arrow/.style={postaction={
        decorate,
        decoration={markings,
          mark=at position #1 with {
            \arrow[xshift=2.25pt, line width=0.85]{
                      Computer Modern Rightarrow[length=1pt 4.5, width'=0pt 1]}
        }}}}]
  \definecolor{linkcolor1}{rgb}{0.15, 0.15, 0.85}
  \begin{scope}[shift={(8.5, 4.3)}, rotate=0, scale=0.425]
    \input plots/K0_smooth_gaps_20_0.3_width_800.tikz
    \node[below=2pt] at (5, 0) {$K_0 = 19nh_{000143796}$};
  \end{scope}

  \begin{scope}[shift={(8.75, 0.0)}, rotate=0, scale=0.385]
    \input plots/K1_base_smooth_gaps_25_0.3_width_800.tikz
    \node[below=2pt] at (5, 0) {
      $K_1 = K\left(-\frac{3}{2}, \frac{1}{2}, \frac{1}{15}\right)$};
    \node at (8.75, 3.8) {\footnotesize $-15$};
  \end{scope}

 \begin{scope}[ shift={(0, 7.25)}, rotate=-90, scale=0.7]
   \input plots/cert_smooth_gaps_15_0.3_width_800.tikz
 \end{scope}
 \node[below] at (3.4, 0.3) {$L$}; 
\end{tikzpicture}
\caption{The links from Theorem~\ref{thm: one knot}.  The standard
  projection of $K_0$ is shown top right. The knot $K_1$ is the
  Montesinos $K\left(-3/2, 1/2, 1/15\right)$, so the labeled box is
  filled in with 15 negative half-twists.  }
\label{fig: link cert}
\end{figure}
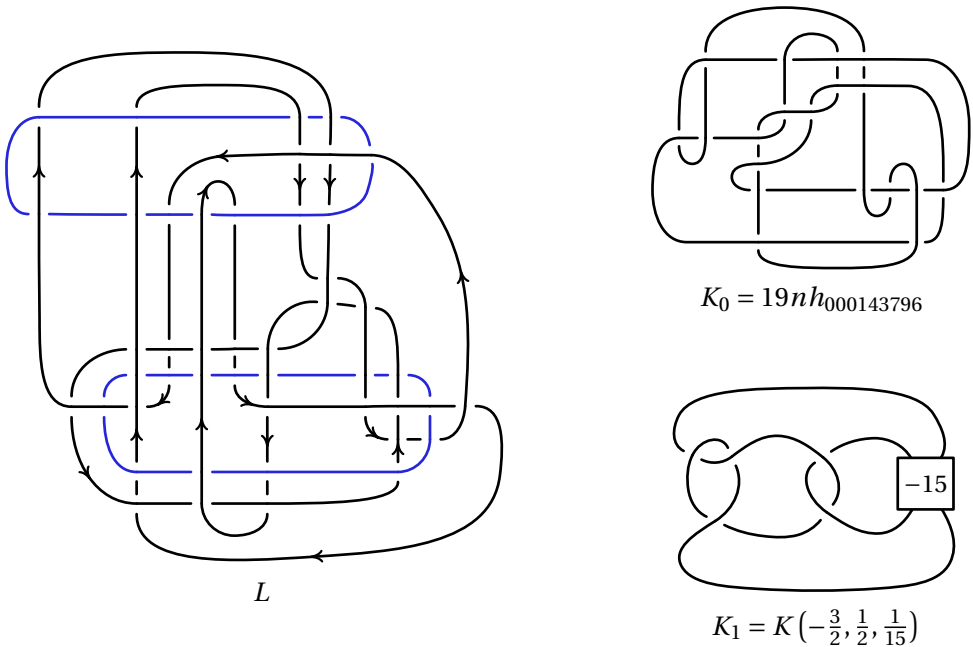

In this section, we deal with a few knots where the techniques in
previous sections seem not to apply.

\subsection{Smooth category}

Recently, the 40-year-old question of whether the knot $17ns_{29}$ is
smoothly slice was resolved:
\begin{theorem}[\cite{DaiEtAl2024}]
  \label{thm: cable 8}
  The (2,1)-cable on the
  figure-8 knot, also known as $17ns_{29}$, is not smoothly slice.
\end{theorem}
An alternate proof of this was given in \cite{AcetoEtAl2023} using
genus bounds in definite 4-manifolds.  In this section, we use the
latter approach to show:

\begin{theorem}
  \label{thm: one knot}
  The knot $K_0 = 19nh_{000143796}$ is topologically slice but not
  smoothly slice.
\end{theorem}

\begin{proof}[Proof of Theorem~\ref{thm: one knot}] First, the knot
$K_0$ is topologically slice because $\Delta_{K_0} = 1$ \cite[\S
11.7]{FreedmanQuinn1990}. Let $L$ and $K_1$ be the links shown in
Figure~\ref{fig: link cert}. We denote the components of $L$ as
$L_0 \cup L_b \cup L_t$ where $L_b$ and $L_t$ are the bottom and top
unknotted components, respectively.  SnapPy \cite{SnapPy} easily
checks the following or you can draw pictures: $L_0$ is isotopic to
$K_0$, and doing $+1$ Dehn surgery on both $L_b$ and $L_t$ turns $L_0$ into
$K_1$.  Further, from Figure~\ref{fig: link cert}, we see that
$\lk(L_0, L_b) = 2$ and $\lk(L_0, L_t) = 6$. Finally, the Seifert
genus of $K_1$ is $g_3(K_1) = 9$ by computing knot Floer homology with
\cite{HFKCalculator}.  The proof is now identical to that of Theorem 2
of \cite{AcetoEtAl2023}, with $K_0$ playing the role of the
(2,1)-cable on the figure-8 knot and $K_1$ playing the role of the
torus knot $T_{2, -19}$, but we give a sketch regardless.

Suppose $L_0 = K_0$ bounds a smooth disk $D_0$ in a 4-ball
$B_0$. Create a closed 4\hyp manifold $W$ by attaching $+1$-framed
2-handles to $B_0$ along $L_b \cup L_t$ and capping the result off
with another 4-ball $B_1$.  Since $L_b \cup L_t$ is the unlink, $W$ is
just $\CP^2 \# \CP^2$.  Moreover, $L_0$ viewed from the other ball
$B_1$ can be identified with $K_1$ after an isotopy.  Let $F_1$ be a
genus $9$ Seifert surface for $K_1$ whose interior has been pushed
into $B_1$.  Gluing $F_0$ and $F_1$ together along their common
boundary gives a smooth closed surface $F \subset W$.  Now
$H_2(W) = \Z^2$, where the first generator is the core of the 2-handle
attached to $L_b$ together with a disk bounded by $L_b$ in $S^3$, and
the second generator is the analogue with $L_t$.  From the linking
information about $L$, we see $[F] = (2, 6)$.  However, by Corollary
1.7 of \cite{Bryan1998}, any smooth connected surface representing the
class $(2, 6)$ has genus at least 10, a contradiction.  So $K_0$ is
not smoothly slice.
\end{proof}

\begin{remark}
  We found the knot in Theorem~\ref{thm: one knot} by a large-scale
  search for proofs similar to \cite{AcetoEtAl2023} in the spirit of
  our Section~\ref{sec: drill RBG}.  That is, starting with a knot
  $K$, we looked at closed geodesics in the hyperbolic structure of
  length at most $4$, and recorded those giving unknots whose linking
  number with $K$ was either $2$ or $6$.  This gives candidates for
  $L_b$ and $L_t$, and one can recover the diagram for the knot $K'$
  resulting from $+1$ surgery on $L_b$ and $L_t$ using
  \cite{DunfieldObeidinRudd2024}.  The $\tau$-invariant in knot Floer
  homology, computed via \cite{HFKCalculator}, was used to give a
  lower bound on the 4-ball genus of $K'$.  Of the roughly \num{12000}
  mystery knots where we tried this technique,
  $K_0 = 19nh_{000143796}$ was the only one where we succeeded.
\end{remark}

\subsection{Topological category}

We will also need:

\begin{theorem}
  \label{thm: alt scraps}
  The nine alternating knots $K14a12741$, $16a350194$,
  $17ah_{0000055}$, $18ah_{0000122}$, $18ah_{3327857}$,
  $18ah_{4025786}$, $18ah_{4099296}$, $18ah_{4099297}$, and
  $19ah_{00000457}$ are not topologically slice.
\end{theorem}

\begin{proof}
The five knots $K14a12741$, $16a350194$, $16a350194$, $17ah_{0000055}$,
$18ah_{0000122}$, and $19ah_{00000457}$ are in fact the 2-bridge knots
$K_{m^2, q}$ for $(m^2, q)$ being $(25, 2)$, $(49, 4)$, $(121, 48)$,
$(169, 48)$, and $(289, 190)$ respectively.  We checked that the
Casson-Gordon signatures obstruct each of these from being
topologically slice, by directly calculating using the
formulas from the original paper \cite{CassonGordonOrsay} and
Theorem 3 therein.

For each of the four knots $18ah_{3327857}$, $18ah_{4025786}$,
$18ah_{4099296}$, and $18ah_{4099297}$, we used the method of
Section~\ref{sec: bands} to find a 1-band ribbon concordance from the
initial knot $K$ to $K_0 \# K_0$ for $K_0 = K7a1 = 7_7$.  By
Corollary 0.2 of \cite{LivingstonNaik1999}, the knot $K_0$ has
infinite order in the topological concordance group
$\cC_1^\mathit{top}$.  Hence, in $\cC_1^\mathit{top}$ we have
$[K] = 2 [K_0] \neq 0$, and so $K$ is not topologically slice.
\end{proof}

%% file: plots/K0_smooth_gaps_20_0.3_width_800.tikz
\begin{scope}
    \draw (8.50, 0.89) .. controls (9.05, 0.89) and (9.05, 1.62) .. (9.06, 2.26);
    \draw (9.06, 2.79) .. controls (9.07, 4.16) and (9.07, 5.76) .. 
          (7.97, 5.78) .. controls (7.52, 5.78) and (6.87, 5.79) .. (6.61, 5.79);
    \draw (6.61, 5.79) .. controls (6.33, 5.79) and (6.06, 5.79) .. (5.79, 5.79);
    \draw (5.79, 5.79) .. controls (5.38, 5.79) and (4.98, 5.60) .. (4.98, 5.23);
    \draw (4.98, 4.70) .. controls (4.98, 3.90) and (4.20, 3.34) .. (3.35, 3.34);
    \draw (3.35, 3.34) .. controls (2.96, 3.34) and (2.53, 3.27) .. 
          (2.53, 2.93) .. controls (2.53, 2.67) and (2.80, 2.52) .. (3.08, 2.52);
    \draw (3.61, 2.52) .. controls (4.52, 2.52) and (5.43, 2.52) .. (6.34, 2.52);
    \draw (6.85, 2.52) .. controls (7.04, 2.52) and (7.23, 2.52) .. (7.42, 2.52);
    \draw (7.42, 2.52) .. controls (7.61, 2.52) and (7.80, 2.52) .. (7.99, 2.52);
    \draw (8.48, 2.52) .. controls (8.67, 2.52) and (8.87, 2.52) .. (9.06, 2.52);
    \draw (9.06, 2.52) .. controls (9.87, 2.52) and (9.86, 3.61) .. 
          (9.86, 4.56) .. controls (9.86, 5.58) and (9.40, 6.60) .. 
          (8.50, 6.60) .. controls (7.87, 6.61) and (7.14, 6.61) .. (6.61, 6.61);
    \draw (6.61, 6.61) .. controls (6.33, 6.61) and (6.06, 6.61) .. (5.79, 6.61);
    \draw (5.79, 6.61) .. controls (5.34, 6.61) and (4.88, 6.60) .. (4.43, 6.60);
    \draw (3.90, 6.60) .. controls (3.17, 6.60) and (2.44, 6.60) .. (1.72, 6.60);
    \draw (1.72, 6.60) .. controls (0.90, 6.60) and (0.90, 5.42) .. (0.90, 4.42);
    \draw (0.90, 3.89) .. controls (0.90, 3.61) and (1.05, 3.34) .. 
          (1.31, 3.34) .. controls (1.64, 3.34) and (1.72, 3.76) .. (1.72, 4.15);
    \draw (1.72, 4.15) .. controls (1.72, 4.88) and (1.72, 5.61) .. (1.72, 6.33);
    \draw (1.72, 6.86) .. controls (1.72, 7.89) and (2.97, 8.23) .. 
          (4.16, 8.23) .. controls (5.35, 8.23) and (6.61, 7.90) .. (6.61, 6.87);
    \draw (6.61, 6.36) .. controls (6.61, 6.25) and (6.61, 6.15) .. (6.61, 6.04);
    \draw (6.61, 5.53) .. controls (6.61, 4.53) and (6.61, 3.52) .. (6.61, 2.52);
    \draw (6.61, 2.52) .. controls (6.61, 2.13) and (6.68, 1.71) .. 
          (7.01, 1.71) .. controls (7.27, 1.71) and (7.42, 1.98) .. (7.42, 2.26);
    \draw (7.42, 2.79) .. controls (7.42, 3.07) and (7.57, 3.34) .. 
          (7.83, 3.34) .. controls (8.16, 3.34) and (8.24, 2.91) .. (8.24, 2.52);
    \draw (8.24, 2.52) .. controls (8.24, 1.98) and (8.24, 1.44) .. (8.24, 0.89);
    \draw (8.24, 0.89) .. controls (8.24, 0.08) and (6.90, 0.08) .. 
          (5.79, 0.08) .. controls (4.72, 0.08) and (3.35, 0.08) .. (3.35, 0.63);
    \draw (3.35, 1.16) .. controls (3.35, 1.61) and (3.35, 2.07) .. (3.35, 2.52);
    \draw (3.35, 2.52) .. controls (3.35, 2.71) and (3.35, 2.90) .. (3.35, 3.09);
    \draw (3.35, 3.58) .. controls (3.35, 3.69) and (3.35, 3.80) .. (3.35, 3.91);
    \draw (3.35, 4.42) .. controls (3.35, 4.78) and (3.75, 4.97) .. (4.16, 4.97);
    \draw (4.16, 4.97) .. controls (4.43, 4.97) and (4.70, 4.97) .. (4.98, 4.97);
    \draw (4.98, 4.97) .. controls (5.39, 4.97) and (5.79, 5.16) .. (5.79, 5.52);
    \draw (5.79, 6.03) .. controls (5.79, 6.14) and (5.79, 6.25) .. (5.79, 6.36);
    \draw (5.79, 6.87) .. controls (5.79, 7.23) and (5.38, 7.41) .. 
          (4.98, 7.41) .. controls (4.53, 7.41) and (4.16, 7.05) .. (4.16, 6.60);
    \draw (4.16, 6.60) .. controls (4.16, 6.15) and (4.16, 5.69) .. (4.16, 5.23);
    \draw (4.16, 4.70) .. controls (4.16, 4.34) and (3.76, 4.15) .. (3.35, 4.15);
    \draw (3.35, 4.15) .. controls (2.89, 4.15) and (2.44, 4.15) .. (1.98, 4.15);
    \draw (1.47, 4.15) .. controls (1.28, 4.15) and (1.09, 4.15) .. (0.90, 4.15);
    \draw (0.90, 4.15) .. controls (0.23, 4.15) and (0.08, 3.30) .. 
          (0.08, 2.52) .. controls (0.08, 1.71) and (0.42, 0.89) .. 
          (1.14, 0.89) .. controls (1.88, 0.89) and (2.61, 0.89) .. (3.35, 0.89);
    \draw (3.35, 0.89) .. controls (4.89, 0.89) and (6.43, 0.89) .. (7.97, 0.89);
  \end{scope}

%% file: plots/K1_base_smooth_gaps_25_0.3_width_800.tikz
  \begin{scope}
    \draw (0.50, 4.94) .. controls (0.09, 5.15) and (0.10, 5.71) .. 
          (0.34, 6.17) .. controls (0.82, 7.09) and (3.00, 7.10) .. 
          (4.63, 7.11) .. controls (6.27, 7.13) and (8.45, 7.14) .. 
          (9.02, 6.21) .. controls (9.31, 5.74) and (9.55, 5.20) .. (9.29, 4.73);
    \draw (9.29, 4.73) .. controls (9.12, 4.42) and (8.95, 4.11) .. (8.78, 3.81);
    \draw (8.78, 3.81) .. controls (8.61, 3.50) and (8.44, 3.19) .. (8.27, 2.88);
    \draw (8.27, 2.88) .. controls (8.03, 2.46) and (7.65, 2.11) .. 
          (7.17, 2.09) .. controls (6.50, 2.05) and (5.92, 2.44) .. (5.38, 2.82);
    \draw (5.38, 2.82) .. controls (4.69, 3.30) and (4.43, 4.15) .. (4.97, 4.53);
    \draw (5.53, 4.92) .. controls (5.96, 5.22) and (6.48, 5.37) .. 
          (7.01, 5.38) .. controls (7.52, 5.38) and (8.00, 5.13) .. (8.25, 4.70);
    \draw (8.25, 4.70) .. controls (8.38, 4.49) and (8.50, 4.28) .. (8.62, 4.07);
    \draw (8.94, 3.54) .. controls (9.06, 3.33) and (9.19, 3.12) .. (9.31, 2.91);
    \draw (9.31, 2.91) .. controls (9.65, 2.34) and (9.87, 1.67) .. 
          (9.58, 1.08) .. controls (9.21, 0.34) and (8.10, 0.26) .. 
          (7.16, 0.19) .. controls (5.76, 0.09) and (4.35, 0.09) .. 
          (2.95, 0.18) .. controls (2.00, 0.24) and (0.79, 0.31) .. 
          (0.44, 0.94) .. controls (0.08, 1.57) and (0.89, 2.08) .. (1.58, 2.52);
    \draw (1.58, 2.52) .. controls (2.24, 2.94) and (2.58, 3.73) .. (2.28, 4.43);
    \draw (2.01, 5.05) .. controls (1.81, 5.50) and (1.10, 5.40) .. (0.80, 4.79);
    \draw (0.80, 4.79) .. controls (0.62, 4.42) and (0.60, 4.01) .. 
          (0.66, 3.60) .. controls (0.72, 3.21) and (0.92, 2.84) .. (1.28, 2.67);
    \draw (1.89, 2.38) .. controls (2.31, 2.17) and (2.76, 2.05) .. 
          (3.22, 2.00) .. controls (3.61, 1.95) and (4.02, 1.93) .. 
          (4.41, 2.04) .. controls (4.71, 2.13) and (4.98, 2.31) .. (5.17, 2.55);
    \draw (5.59, 3.09) .. controls (5.98, 3.59) and (5.77, 4.30) .. (5.25, 4.72);
    \draw (5.25, 4.72) .. controls (4.80, 5.09) and (4.31, 5.44) .. 
          (3.73, 5.46) .. controls (3.13, 5.49) and (2.62, 5.11) .. (2.14, 4.74);
    \draw (2.14, 4.74) .. controls (1.85, 4.51) and (1.44, 4.47) .. (1.11, 4.64);
  \end{scope}

  \draw[fill=white] (7.8, 2.92) rectangle (9.7, 4.7);
  

%% file: plots/cert_smooth_gaps_15_0.3_width_800.tikz
  \begin{scope}
    \draw (6.20, 7.42) .. controls (5.64, 7.42) and (4.97, 7.42) .. (4.93, 7.01);
    \draw (4.89, 6.61) .. controls (4.87, 6.50) and (4.86, 6.40) .. (4.85, 6.29);
    \draw (4.81, 5.90) .. controls (4.76, 5.40) and (5.18, 4.98) .. (5.70, 4.98);
    \draw (5.70, 4.98) .. controls (5.87, 4.98) and (6.03, 4.98) .. (6.20, 4.98);
    \draw (6.20, 4.98) .. controls (6.34, 4.98) and (6.48, 4.98) .. (6.62, 4.98);
    \draw[mid arrow=0.55] (7.00, 4.97) .. controls (7.28, 4.97) and (7.56, 4.97) .. (7.84, 4.97);
    \draw (8.22, 4.97) .. controls (8.30, 4.97) and (8.38, 4.97) .. (8.46, 4.97);
    \draw (8.84, 4.97) .. controls (9.12, 4.97) and (9.25, 4.66) .. 
          (9.25, 4.35) .. controls (9.25, 4.02) and (8.98, 3.74) .. (8.64, 3.74);
    \draw (8.64, 3.74) .. controls (8.44, 3.74) and (8.24, 3.74) .. (8.04, 3.74);
    \draw[mid arrow=0.5] (8.04, 3.74) .. controls (7.42, 3.74) and (6.81, 3.74) .. (6.20, 3.74);
    \draw (6.20, 3.74) .. controls (6.03, 3.74) and (5.87, 3.74) .. (5.71, 3.74);
    \draw (5.71, 3.74) .. controls (4.86, 3.74) and (4.01, 3.74) .. (3.15, 3.74);
    \draw[mid arrow=0.33] (3.15, 3.74) .. controls (2.85, 3.74) and (2.52, 3.79) .. 
          (2.52, 4.05) .. controls (2.52, 4.25) and (2.74, 4.36) .. (2.96, 4.36);
    \draw (3.36, 4.36) .. controls (4.07, 4.36) and (4.79, 4.36) .. (5.51, 4.36);
    \draw (5.85, 4.36) .. controls (5.92, 4.36) and (5.98, 4.36) .. (6.05, 4.36);
    \draw[mid arrow=0.55] (6.40, 4.36) .. controls (6.67, 4.36) and (6.80, 4.67) .. (6.80, 4.98);
    \draw (6.80, 4.98) .. controls (6.80, 5.59) and (6.80, 6.20) .. (6.80, 6.81);
    \draw (6.80, 6.81) .. controls (6.79, 7.01) and (6.79, 7.21) .. (6.79, 7.42);
    \draw (6.79, 7.42) .. controls (6.79, 7.62) and (6.79, 7.82) .. (6.79, 8.02);
    \draw (6.79, 8.02) .. controls (6.79, 8.18) and (6.79, 8.33) .. (6.79, 8.48);
    \draw[mid arrow=0.6](6.79, 8.88) .. controls (6.78, 9.40) and (7.48, 9.41) .. 
          (8.11, 9.33) .. controls (9.44, 9.17) and (9.55, 7.38) .. 
          (9.65, 5.84) .. controls (9.75, 4.34) and (9.86, 2.52) .. (8.84, 2.52);
    \draw (8.46, 2.52) .. controls (8.38, 2.52) and (8.30, 2.52) .. (8.22, 2.52);
    \draw[mid arrow=0.5] (7.84, 2.52) .. controls (7.49, 2.52) and (7.15, 2.52) .. (6.80, 2.52);
    \draw (6.80, 2.52) .. controls (6.60, 2.52) and (6.40, 2.52) .. (6.20, 2.52);
    \draw (6.20, 2.52) .. controls (6.04, 2.52) and (5.87, 2.52) .. (5.71, 2.52);
    \draw (5.71, 2.52) .. controls (4.86, 2.52) and (4.00, 2.52) .. (3.15, 2.52);
    \draw[mid arrow=0.5] (3.15, 2.52) .. controls (2.60, 2.52) and (2.05, 2.52) .. (1.50, 2.52);
    \draw (1.10, 2.52) .. controls (0.69, 2.52) and (0.69, 3.37) .. 
          (0.69, 4.05) .. controls (0.69, 4.76) and (0.69, 5.58) .. (1.30, 5.58);
    \draw (1.30, 5.58) .. controls (1.47, 5.58) and (1.64, 5.58) .. (1.81, 5.58);
    \draw[mid arrow=0.5] (2.20, 5.58) .. controls (2.46, 5.58) and (2.71, 5.58) .. (2.96, 5.58);
    \draw (3.36, 5.58) .. controls (3.81, 5.58) and (4.36, 5.58) .. (4.36, 5.90);
    \draw (4.36, 6.30) .. controls (4.36, 6.59) and (4.61, 6.81) .. (4.91, 6.81);
    \draw (4.91, 6.81) .. controls (5.34, 6.81) and (5.77, 6.81) .. (6.20, 6.81);
    \draw (6.20, 6.81) .. controls (6.34, 6.81) and (6.48, 6.81) .. (6.62, 6.81);
    \draw[mid arrow=0.6] (7.00, 6.81) .. controls (7.23, 6.81) and (7.42, 6.99) .. (7.42, 7.22);
    \draw (7.42, 7.60) .. controls (7.42, 7.68) and (7.42, 7.76) .. (7.42, 7.84);
    \draw (7.42, 8.22) .. controls (7.42, 8.51) and (7.10, 8.66) .. (6.79, 8.68);
    \draw[mid arrow=0.45] (6.79, 8.68) .. controls (5.54, 8.77) and (4.24, 8.82) .. 
          (3.17, 8.18) .. controls (2.64, 7.86) and (2.03, 7.50) .. (2.02, 6.92);
    \draw (2.02, 6.92) .. controls (2.02, 6.66) and (2.02, 6.41) .. (2.01, 6.15);
    \draw (2.01, 6.15) .. controls (2.01, 5.96) and (2.01, 5.77) .. (2.00, 5.58);
    \draw[mid arrow=0.5] (2.00, 5.58) .. controls (2.00, 5.16) and (2.02, 4.62) .. 
          (2.04, 4.15) .. controls (2.06, 3.61) and (2.43, 3.13) .. (2.95, 3.13);
    \draw (3.35, 3.13) .. controls (4.07, 3.13) and (4.79, 3.13) .. (5.51, 3.13);
    \draw (5.86, 3.13) .. controls (5.92, 3.13) and (5.99, 3.13) .. (6.05, 3.13);
    \draw[mid arrow=0.55] (6.40, 3.13) .. controls (6.62, 3.13) and (6.81, 2.95) .. (6.80, 2.72);
    \draw (6.80, 2.34) .. controls (6.80, 2.20) and (6.80, 2.06) .. (6.80, 1.91);
    \draw (6.80, 1.91) .. controls (6.80, 1.71) and (6.80, 1.51) .. (6.80, 1.30);
    \draw (6.80, 1.30) .. controls (6.80, 0.71) and (5.90, 0.70) .. 
          (5.15, 0.70) .. controls (4.48, 0.69) and (3.50, 0.69) .. (3.14, 0.69);
    \draw[mid arrow=0.5] (3.14, 0.69) .. controls (2.60, 0.69) and (2.05, 0.69) .. (1.50, 0.69);
    \draw (1.10, 0.69) .. controls (0.07, 0.69) and (0.07, 2.17) .. 
          (0.07, 3.44) .. controls (0.07, 4.73) and (0.19, 6.18) .. (1.30, 6.16);
    \draw (1.30, 6.16) .. controls (1.47, 6.16) and (1.64, 6.16) .. (1.81, 6.15);
    \draw[mid arrow=0.5] (2.21, 6.15) .. controls (2.46, 6.14) and (2.70, 6.14) .. (2.95, 6.13);
    \draw (3.35, 6.12) .. controls (3.69, 6.12) and (4.02, 6.11) .. (4.36, 6.10);
    \draw (4.36, 6.10) .. controls (4.52, 6.10) and (4.67, 6.10) .. (4.83, 6.10);
    \draw (4.83, 6.10) .. controls (5.32, 6.09) and (5.70, 5.67) .. (5.70, 5.18);
    \draw (5.70, 4.79) .. controls (5.70, 4.65) and (5.70, 4.50) .. (5.70, 4.36);
    \draw (5.70, 4.36) .. controls (5.71, 4.22) and (5.71, 4.07) .. (5.71, 3.93);
    \draw (5.71, 3.56) .. controls (5.71, 3.42) and (5.71, 3.27) .. (5.71, 3.13);
    \draw (5.71, 3.13) .. controls (5.71, 2.99) and (5.71, 2.85) .. (5.71, 2.71);
    \draw (5.71, 2.32) .. controls (5.72, 1.78) and (6.08, 1.30) .. (6.60, 1.30);
    \draw[mid arrow=0.5] (7.00, 1.30) .. controls (7.84, 1.30) and (8.64, 1.76) .. (8.64, 2.52);
    \draw (8.64, 2.52) .. controls (8.64, 2.86) and (8.64, 3.20) .. (8.64, 3.54);
    \draw (8.64, 3.94) .. controls (8.64, 4.28) and (8.64, 4.63) .. (8.64, 4.97);
    \draw (8.64, 4.97) .. controls (8.64, 6.02) and (8.64, 7.42) .. (8.23, 7.42);
    \draw[mid arrow=0.6] (7.85, 7.42) .. controls (7.70, 7.42) and (7.56, 7.42) .. (7.42, 7.42);
    \draw (7.42, 7.42) .. controls (7.27, 7.42) and (7.12, 7.42) .. (6.98, 7.42);
    \draw (6.61, 7.42) .. controls (6.48, 7.42) and (6.34, 7.42) .. (6.20, 7.42);
  \end{scope}
  \begin{scope}[color=linkcolor1]
    \draw (1.30, 0.69) .. controls (1.30, 0.28) and (1.76, 0.08) .. 
          (2.22, 0.08) .. controls (2.66, 0.08) and (3.14, 0.11) .. (3.14, 0.49);
    \draw (3.14, 0.88) .. controls (3.15, 1.36) and (3.15, 1.84) .. (3.15, 2.32);
    \draw (3.15, 2.71) .. controls (3.15, 2.85) and (3.15, 2.99) .. (3.15, 3.13);
    \draw (3.15, 3.13) .. controls (3.15, 3.27) and (3.15, 3.42) .. (3.15, 3.56);
    \draw (3.15, 3.93) .. controls (3.16, 4.07) and (3.16, 4.22) .. (3.16, 4.36);
    \draw (3.16, 4.36) .. controls (3.16, 4.77) and (3.16, 5.17) .. (3.16, 5.58);
    \draw (3.16, 5.58) .. controls (3.16, 5.76) and (3.16, 6.02) .. (3.15, 6.13);
    \draw (3.15, 6.13) .. controls (3.11, 6.44) and (2.99, 6.77) .. 
          (2.69, 6.85) .. controls (2.52, 6.90) and (2.30, 6.96) .. (2.19, 6.94);
    \draw (1.83, 6.89) .. controls (1.54, 6.84) and (1.30, 6.64) .. (1.30, 6.36);
    \draw (1.30, 5.99) .. controls (1.30, 5.91) and (1.30, 5.83) .. (1.30, 5.75);
    \draw (1.30, 5.38) .. controls (1.30, 4.43) and (1.30, 3.47) .. (1.30, 2.52);
    \draw (1.30, 2.52) .. controls (1.30, 1.91) and (1.30, 1.30) .. (1.30, 0.69);
  \end{scope}
  \begin{scope}[color=linkcolor1]
    \draw (6.20, 7.62) .. controls (6.20, 7.84) and (6.37, 8.02) .. (6.59, 8.02);
    \draw (6.98, 8.02) .. controls (7.12, 8.02) and (7.27, 8.02) .. (7.42, 8.02);
    \draw (7.42, 8.02) .. controls (7.76, 8.02) and (8.04, 7.75) .. (8.04, 7.42);
    \draw (8.04, 7.42) .. controls (8.04, 6.60) and (8.04, 5.79) .. (8.04, 4.97);
    \draw (8.04, 4.97) .. controls (8.04, 4.63) and (8.04, 4.28) .. (8.04, 3.94);
    \draw (8.04, 3.54) .. controls (8.04, 3.20) and (8.04, 2.86) .. (8.04, 2.52);
    \draw (8.04, 2.52) .. controls (8.04, 2.08) and (7.51, 1.91) .. (7.00, 1.91);
    \draw (6.60, 1.91) .. controls (6.38, 1.91) and (6.20, 2.10) .. (6.20, 2.32);
    \draw (6.20, 2.71) .. controls (6.20, 2.85) and (6.20, 2.99) .. (6.20, 3.13);
    \draw (6.20, 3.13) .. controls (6.20, 3.27) and (6.20, 3.42) .. (6.20, 3.56);
    \draw (6.20, 3.93) .. controls (6.20, 4.07) and (6.20, 4.22) .. (6.20, 4.36);
    \draw (6.20, 4.36) .. controls (6.20, 4.50) and (6.20, 4.65) .. (6.20, 4.79);
    \draw (6.20, 5.18) .. controls (6.20, 5.65) and (6.20, 6.13) .. (6.20, 6.61);
    \draw (6.20, 6.99) .. controls (6.20, 7.07) and (6.20, 7.15) .. (6.20, 7.23);
  \end{scope}

%% file: concordances.tex
\section{Ribbon concordance graph}
\label{sec: ribbon graph}

Our search for ribbon disks in Section~\ref{sec: bands} also produced
many ribbon concordances between knots in \PS.  In this section, we
outline how these are used in the proof of Theorem~\ref{thm: main},
and how they relate to various open questions.  Suppose $K_0$ and
$K_1$ are knots in $S^3$.  Recall from \cite{Gordon1981} that a
\emph{ribbon concordance from $K_1$ to $K_0$} is a smooth annulus $A$
embedded in $S^3 \times I$ with boundary components $K_0 \times \{0\}$
and $K_1 \times \{1\}$ so that the $I$-coordinate is a Morse function
on $A$ with no local maxima.  In this case, we write $K_1 \geq K_0$.
In particular, $K$ is ribbon if and only if $K \geq \mathrm{unknot}$.
Agol recently proved that $\geq$ gives a partial order on isotopy
classes of knots \cite{Agol2022}. 

We will think of this partial order in terms of the \emph{ribbon
  concordance graph} \ConcordGraph: the graph whose vertices are
isotopy classes of knots with a \emph{directed} edge from $K_1$ to
$K_0$ if and only if $K_1 \geq K_0$.  Note that $K$ is ribbon if and
only if there is a directed path starting at $K$ and ending at the
unknot.  The connected components of \ConcordGraph\ as a topological
space are exactly the concordance classes of (unoriented) knots
\cite[\S 6]{Gordon1981}.  In particular, the connected component of
\ConcordGraph\ containing the unknot is the set of slice knots.

Considering the directed structure, define a \emph{sink} for a
component $C$ of \ConcordGraph\ to be a vertex $K$ so that every
vertex $J \in C$ has a directed path from $J$ to $K$.  Because ribbon
concordance is a partial order, a component can have at most one
sink.  The ribbon-slice conjecture is that the unknot is the unique
sink of its component, and Gordon asked in \cite{Gordon1981} whether
every component of \ConcordGraph\ has a sink.

Our data allows us to partially reconstruct the subgraph of
\ConcordGraph\ whose vertices are $\PSM \cup \{\mathrm{unknot}\}$.
Let $\Gamma_0$ be the directed graph whose vertices are
$\PSM \cup \{\mathrm{unknot}\}$ with an edge from $K_1$ to $K_0$
if we found a ribbon concordance $K_1 \geq K_0$.  We define
\PartialConcordGraph\ to be the subgraph of $\Gamma_0$ where we remove
all the isolated vertices, that is, all the knots that are not
involved in any ribbon concordance we found.  We call
\PartialConcordGraph\ the \emph{partial ribbon concordance graph} of \PS.
(Technical aside: elements of \PS\ are not isotopy classes of knots but
rather defined up to a possibly orientation-reversing homeomorphism of
$S^3$.  So \PartialConcordGraph\ is a subgraph not of \ConcordGraph\ but
of a $\Z/2\Z$-quotient of it.)

\subsection{Properties of the partial concordance graph}

The graph \PartialConcordGraph\ has \NumVertsInConcordGraph\ vertices
and \NumEdgesInConcordGraph\ edges.  Of the edges, some
\NumRibbonConcordToUnknot\ end at the unknot.  As a topological space,
it has \NumComponentsConcordGraph\ connected components.  The largest
component $U$ is the one containing the unknot.  The other knots in
$U$ are exactly the \NumRibbonKnots\ ribbon knots we identified in
Theorem~\ref{thm: ribbon}; the unknot is the sink of
$U$ as predicted by the slice-ribbon conjecture.  A very small
piece of $U$ is shown in Figures~\ref{fig: piece of U}, \ref{fig: U
  reduced}, and \ref{fig: U closeup}.

The other components of \PartialConcordGraph\ have mean size
\num{84.0} and median size \num{47.0}; the largest such component has
size \num{1673}, with sink the Conway knot $K11n34$.  Figure~\ref{fig:
  comp size} shows the sizes roughly follow a power-law
distribution. In contrast with Figure~\ref{fig: piece of U}, every one
of these other components has the following simple structure as a
directed graph: there is a sink and every edge ends there.  Note this
is consistent with the generalized slice-ribbon conjecture that every
component of \ConcordGraph\ has a sink.

It is natural to ask how various invariants behave with respect to
ribbon concordance.

\begin{enumerate}
\item
  \label{item: hyp decline}
  When the exterior $S^3 \setminus K$ is hyperbolic, i.e.~has a
  complete finite-volume Riemannian metric of constant curvature $-1$,
  we define $\vol(K)$ to be that volume. This notion naturally extends
  to all knots by considering the Gromov norm; here, unknots and torus
  knots have volume $0$.  Gordon asked in
  \cite{Gordon1981} whether $K_1 \geq K_0$ implies
  $\vol(K_1) \geq \vol(K_0)$. This is true for \PartialConcordGraph\
  in a very strong form.  For $K_1 > K_0$, the smallest value of
  $\vol(K_1) - \vol(K_0)$ was about 3.16, occurring for $K_1 = K6a3$
  and $K_0 = \mathrm{unknot}$.  Moreover, the largest ratio
  $(\vol(K_1) - \vol(K_0))/\vol(K_1)$ was $\approx 0.74$ for
  $K_1 = 19ah_{35125143}$ and $K_0 = K14a14913$; that is, the volume
  always declined by more than 25\% along each edge of
  \PartialConcordGraph.
  
  
\item Conjecture 9 of \cite{GreeneOwens2022} posits that if $K_1
  \geq K_0$ and $K_1$ is alternating then $K_0$ must be as well.
  Equivalently, there are no edges in \ConcordGraph\ from an
  alternating knot to a nonalternating knot. The graph
  \PartialConcordGraph\ has \NumConcordNonAltToAlt\ edges where one
  end is an alternating knot and the other is not.  In every case,
  these go in the expected direction.

\item Zemke showed that if $K_1 \geq K_0$ then the induced map on knot
  Floer homology $\HFK(K_0) \to \HFK(K_1)$ is injective.  In
  particular, for each $(i,j)$-bigrading one has
  $\dim \HFK_i(K_1, j) \geq \dim \HFK_i(K_0, j) $ as $\F_2$-vector
  spaces \cite{Zemke2019}.  (The analogous result holds for Khovanov
  homology \cite{LevineZemke2019}.)  Given this, a natural question is
  whether one could have $\HFK(K_0) \cong \HFK(K_1)$ when $K_1 > K_0$.
  For the total dimension of $\HFK$, the smallest absolute drop was
  $K_1 = K6a3$ and $K_0 = \mathrm{unknot}$ where it declined from 9 to
  1. The smallest decline as a proportion was the edge from $16n30703$
  to $K6a3$ where it decreased by some 78.0\%.  (In comparing this
  decline with that of hyperbolic volume, recall the connection
  shown in Figure~\ref{fig: vol vs HFK}).


\item
  \label{item: genus drop}
  A consequence of \cite{Zemke2019} is that the Seifert genus respects
  ribbon concordance: if $K_1 \geq K_0$ then $g_3(K_1) \geq g_3(K_0)$.
  For \PartialConcordGraph, moreover, every edge strictly reduced
  $g_3$.  This prompts us to ask: are there prime knots $K_1 > K_0$
  with the same $g_3$?
  
\item By \cite[Proposition~5]{Miyazaki2018}, if $K_1 \geq K_0$ and
  $K_1$ is fibered then so is $K_0$. Given this, it is natural to ask
  whether the dimension of $\HFK$ in the top Alexander grading is
  non-increasing.  This is true for 99.25\% of the edges in
  \PartialConcordGraph, but there are 3,094 cases where it increases.

\end{enumerate}

\subsection{Propagating information} Finally, we use
\PartialConcordGraph\ to complete the proof of the main theorem. The
idea is that, since each component $C$ of \PartialConcordGraph\
consists of smoothly concordant knots, if we know something about one
$K$ in $C$, we can perhaps deduce the same about every knot in $C$.
While the smooth obstructions of Section~\ref{sec: smooth} are all
concordance invariants, and so apply or not to every element in $C$
uniformly, this is not the case for the techniques in
Sections~\ref{sec: HKL}, \ref{sec: friends}, and \ref{sec: nonslice
  scraps}.

\begin{proof}[Proof of Theorem~\ref{thm: real main}]

Of necessity, we only give an outline and defer to \cite{CodeAndData}
for the detailed data.

We first do the topological category.  Recall from Theorem~\ref{thm:
  ribbon} that we identified \NumRibbonKnots\ ribbon knots in \PS.
Moreover, \PS\ contains \num{34930} knots where $\Delta_K = 1$, which
are all topologically slice by \cite[\S 11.7]{FreedmanQuinn1990}; this
contributes \num{18954} topologically slice knots beyond the known
ribbon knots.  Any component $C$ of \PartialConcordGraph\ containing a
topologically slice knot consists entirely of such knots.  For example, the
Conway knot $K11n34$ has $\Delta_K = 1$ and it is the sink of a
component $C$ of size \num{1673}; this tells us \num{1672} knots with
$\Delta_K \neq 1$ are topologically slice.  There are 17 other
components whose sink is a knot with $\Delta_K = 1$.  Collectively,
these 18 components contribute \num{3457} more topologically slice
knots.  This gives the claimed total of at least \NumTopSliceKnots\
topologically slice knots in \PS.

Theorem~\ref{thm: HKL obs summary} gives us \NumObsByAnyHKLTest\ knots
in \PS\ that are not topologically slice, and Theorem~\ref{thm: alt
  scraps} gives 9 more.  We can also propagate this information via
\PartialConcordGraph. For example, the component $C$ of
\PartialConcordGraph\ with sink $K14n5389$ has $\#C = 25$.  The method
of Section~\ref{sec: q^e} with $(m, q) = (2, 3)$ shows 13 of the knots
$K$ cannot be topologically slice and hence the other knots in $C$ are
not as well.  (The 13 knots where the method works have
$H_1(B_2; \Z) = \Z/3\Z \oplus \Z/27\Z$ whereas the other 12 have
$H_1(B_2; \Z)$ one of $(\Z/3\Z)^2$, $(\Z/9\Z)^2$, or
$(\Z/3\Z)^2 \oplus \Z/9\Z$.  Note that Lemma~\ref{lem: small meta}
gives only two possible metabolizers for $\Z/3\Z \oplus \Z/27\Z$
whereas there are 13 possibilities for $\Z/9\Z \oplus \Z/9\Z$, which
makes it harder to rule them all out.)  This kind of
deduction applies to some \num{144} knots, for a total of at least
\NumNotTopSlice\ knots in \PS\ that are not topologically slice.
This completes our discussion in the topological category.

In the smooth setting, Theorem~\ref{thm: smooth summary} obstructs
smooth slicing of \NumObsBySmoothInvsBeyondFullTopObs\ knots beyond
those known not to be topologically slice. As the obstructions in
Theorem~\ref{thm: smooth summary} are concordance invariants, we
cannot propagate them via \PartialConcordGraph.  Theorem~\ref{thm:
  slice via friends} gives 25 additional knots that are not smoothly
slice.  Two of these, $K11n34$ and $K13n866$, are sinks of components
of \PartialConcordGraph\ of sizes \num{1673} and \num{142}
respectively, which adds \num{1813} nonslice knots.  Theorem~\ref{thm:
  cable 8} and Theorem~\ref{thm: one knot} contribute one nonslice
knot apiece, for a total of at least \NumNotSmoothSlice\ knots in \PS\
that are not smoothly slice.
\end{proof}

\begin{table}
  \centering
  \begin{tabular}{lr@{\hskip 3.5em}lr@{\hskip 3.5em}lr}
    \toprule
    sink & $\# C$ & sink & $\# C$ & sink & $\# C$ \\
    \midrule
    $K13n65$ & 199 & $K13n3871$ & 88 & $K14n4621$ & 56 \\
    $K13n4582$ & 138 & $K13n3897$ & 86 & $K14n9023$ & 43 \\
    $K13n3872$ & 127 & $K14n18911$ & 59 & $K14n11063$ & 43 \\
    $K13n3936$ & 127 & $K14n18909$ & 59 & $K14n4425$ & 37 \\
    $K14n21673$ & 99 & $K14n3713$ & 58 & $K14n5486$ & 33 \\
    \bottomrule
  \end{tabular}
  \caption{The 15 largest components $C$ of \PartialConcordGraph\
    whose smooth slice status is unknown; together they contain
    \num{1251} knots, some \num{11.0}\% of the those in \PS\ whose
    smooth slice status is unknown.  With the exceptions of $K14n3713$ and
    $K14n4621$, all of the sink knots listed above have $\Delta_K = 1$
    and hence those components consist entirely of topologically slice
    knots.  The components with sinks $K14n3713$ and $K14n4621$ are
    the only two in \PS\ where we were unable to determine whether the
    knots are topologically slice.}
  \label{tab: unknown comps}
\end{table}

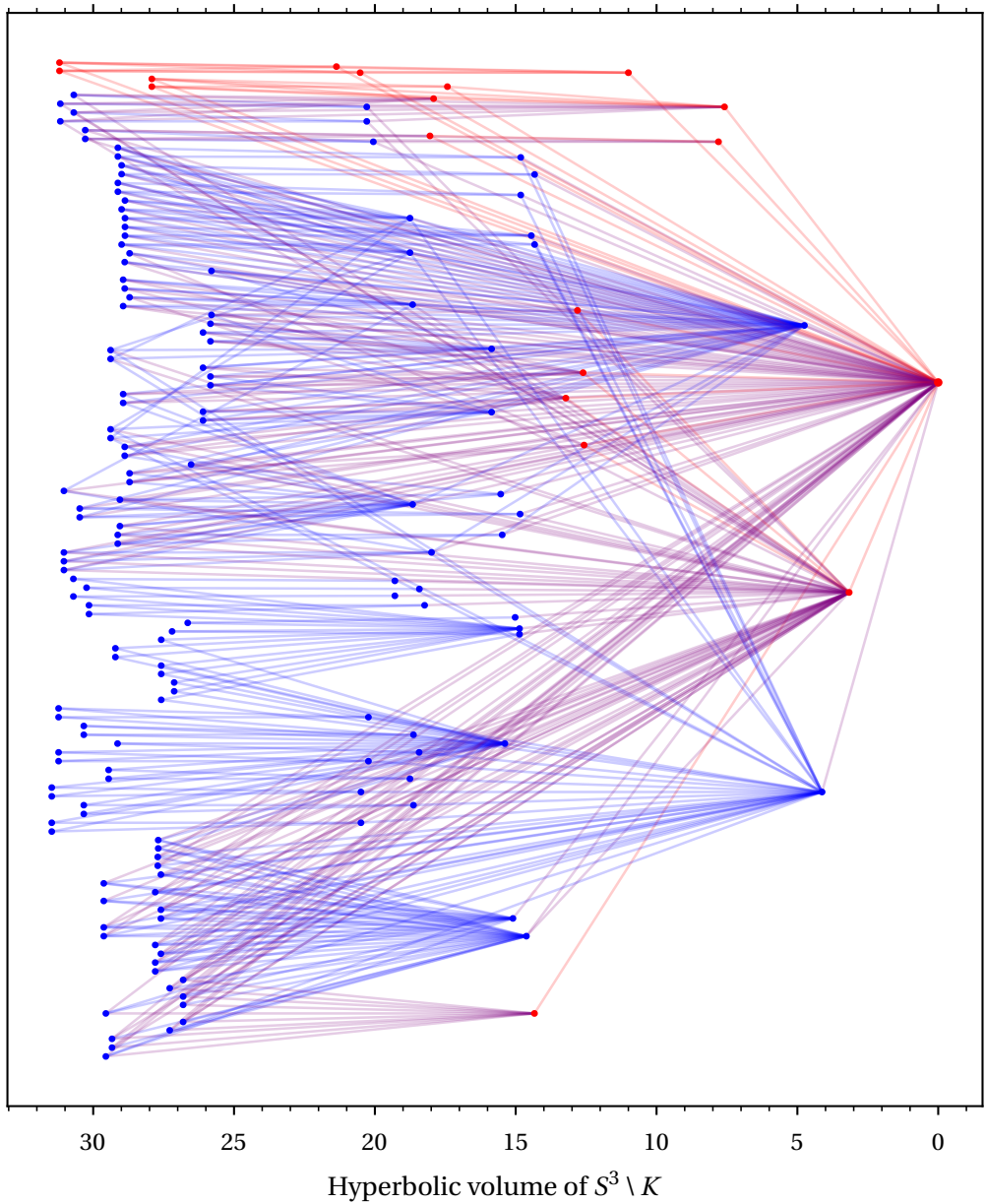
\begin{figure}
  \centering
  \begin{tikzpicture}[font=\footnotesize]
    \input plots/concord_full.tex
  \end{tikzpicture}
  \caption{A small part of the component of \PartialConcordGraph
    containing the unknot.  Each knot is indicated by a dot, where
    alternating knots are red and other knots are blue.  Only edges
    corresponding to known 1-band ribbon cobordisms are drawn; in each
    case, they are directed so that the volume decreases,
    and their color is the average of that of their endpoints.  The
    168 knots used here were carefully chosen to make the picture
    complicated and interesting. The vertical coordinates of the dots
    were algorithmically selected to make the picture more
    comprehensible; they do not have an intrinsic meaning.
    See also Figure~\ref{fig: U reduced} and
    Figure~\ref{fig: U closeup}.  }
  \label{fig: piece of U}
\end{figure}

\begin{figure}
  \centering
  \begin{tikzpicture}[font=\footnotesize]
    \input plots/concord_reduced.tex
  \end{tikzpicture}
  \caption{A subgraph of \PartialConcordGraph\ with the same vertices
    as that in Figure~\ref{fig: piece of U}, but with a minimal subset
    of edges generating the underlying partial order.}
  \label{fig: U reduced}
\end{figure}
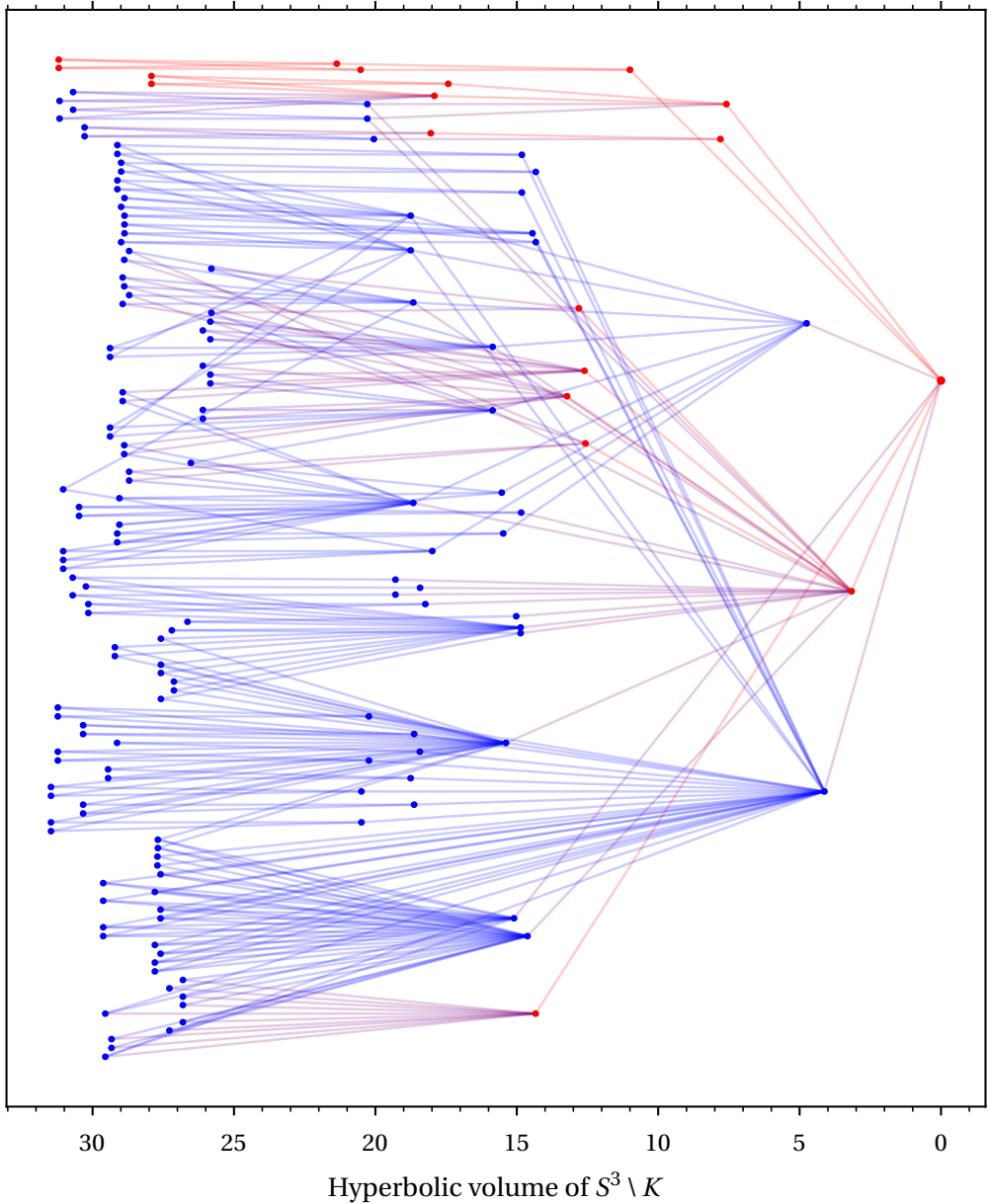

\begin{figure}
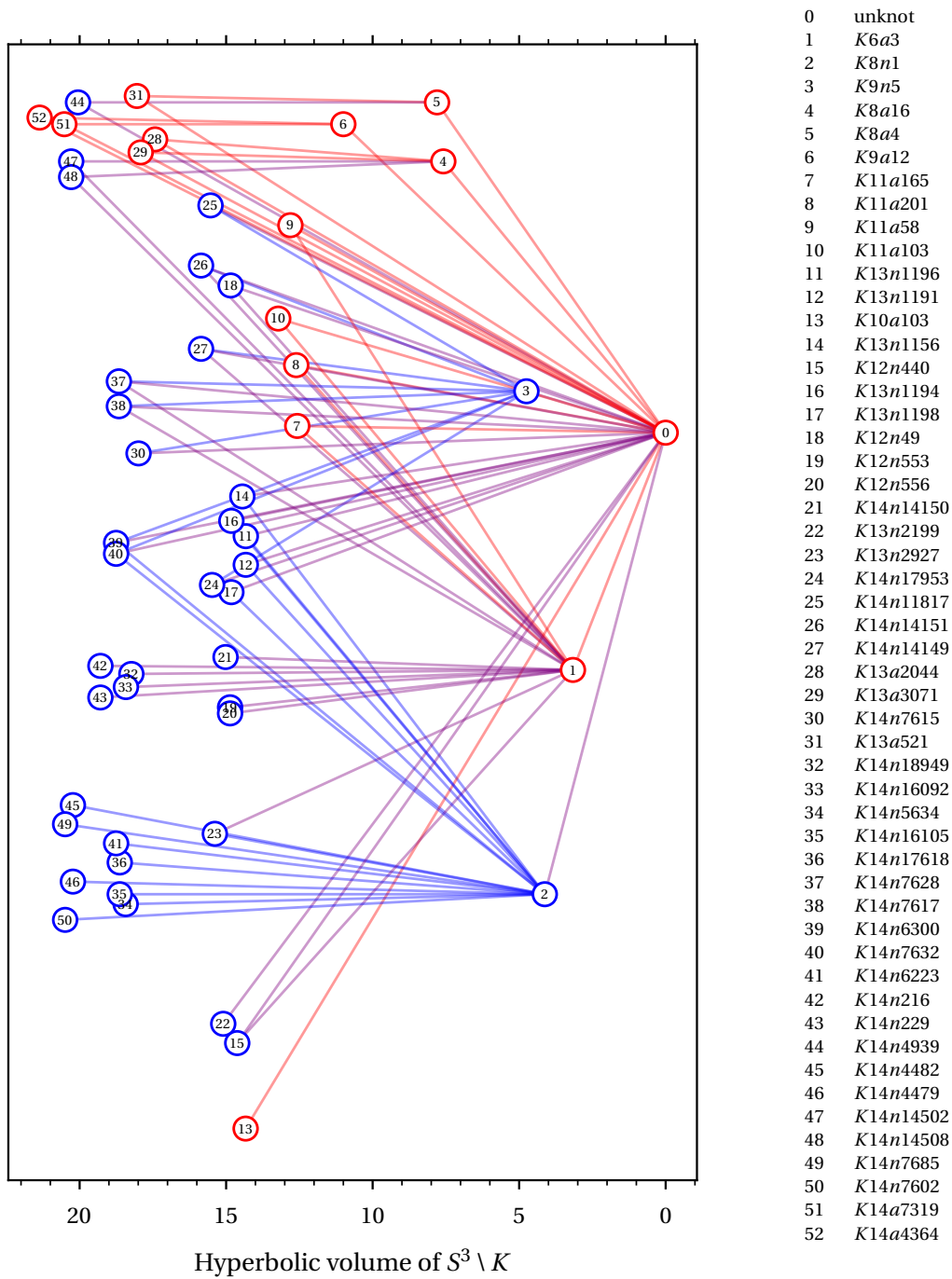

  \centering
  \input plots/concord_closeup.tex
  \caption{The subgraph of the one shown in Figure~\ref{fig: piece of
      U} corresponding to knots with at most 17 crossings,
    which in this case is equivalent to $\vol(S^3 \setminus K) < 22$.
  }
  \label{fig: U closeup}
\end{figure}

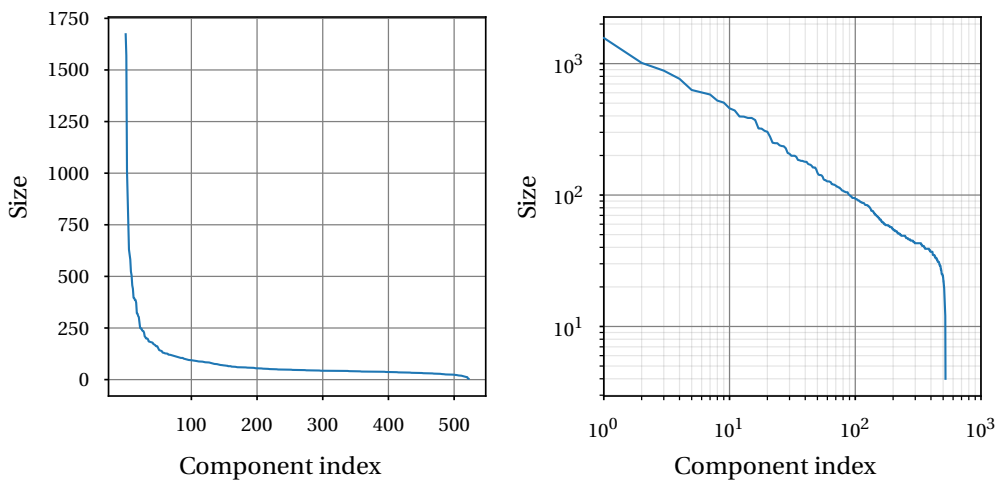
\begin{figure}
  \centering
  \begin{tikzpicture}[font=\scriptsize]
    \tikzset{axis label/.style={font=\footnotesize}}
    \newcommand{\figurewidth}{6cm}
    \input plots/comp_size.tex
    \begin{scope}[shift={(6.5, 0)}]
      \input plots/comp_size_log_log.tex
    \end{scope}
  \end{tikzpicture}
  \caption{This pair of plots shows the sizes of the components of
    \PartialConcordGraph\ that do not contain the unknot.  On the
    horizontal axis, the \num{523} components are ordered from largest
    to smallest. The two plots differ only in the use of log-scales on
    the right so that the power-law behavior is evident. }
  \label{fig: comp size}
\end{figure}

%% file: plots/concord_full.tex
\begin{tikzoverlay*}[width=0.95\textwidth]{concord_full.pdf}
  \draw (50.000000, 5.0) node[below] {\small Hyperbolic volume
    of $S^3 \setminus K$};
  \draw (93.019481, 8.670635) node[below] {$0$};

  \draw (79.349955, 8.670635) node[below] {$5$};

  \draw (65.680429, 8.670635) node[below] {$10$};

  \draw (52.010903, 8.670635) node[below] {$15$};

  \draw (38.341377, 8.670635) node[below] {$20$};

  \draw (24.671851, 8.670635) node[below] {$25$};

  \draw (11.002325, 8.670635) node[below] {$30$};

  \begin{scope}[shift={(93.01948052, 14.90599745)},
                xscale=-2.73390519, yscale=0.00511876]
  \end{scope}
\end{tikzoverlay*}

%% file: plots/concord_reduced.tex
\begin{tikzoverlay*}[width=0.95\textwidth]{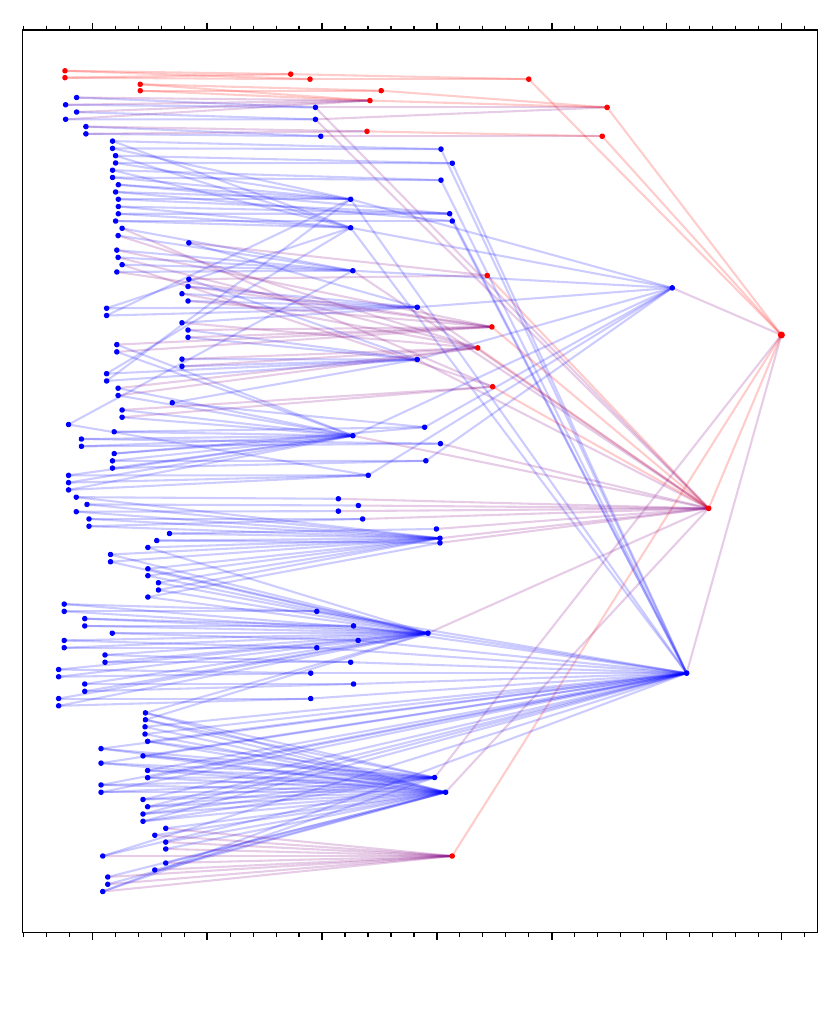}
  \draw (50.000000, 5.0) node[below] {\small Hyperbolic volume of $S^3
  \setminus K$};
  \draw (93.019481, 8.670635) node[below] {$0$};

  \draw (79.349955, 8.670635) node[below] {$5$};

  \draw (65.680429, 8.670635) node[below] {$10$};

  \draw (52.010903, 8.670635) node[below] {$15$};

  \draw (38.341377, 8.670635) node[below] {$20$};

  \draw (24.671851, 8.670635) node[below] {$25$};

  \draw (11.002325, 8.670635) node[below] {$30$};

  \begin{scope}[shift={(93.01948052, 14.90599745)},
                xscale=-2.73390519, yscale=0.00511876]
  \end{scope}
\end{tikzoverlay*}

%% file: plots/concord_closeup.tex
\begin{tikzpicture}
  \begin{scope}[font=\tiny]
    \begin{tikzoverlay*}[width=10.5cm]{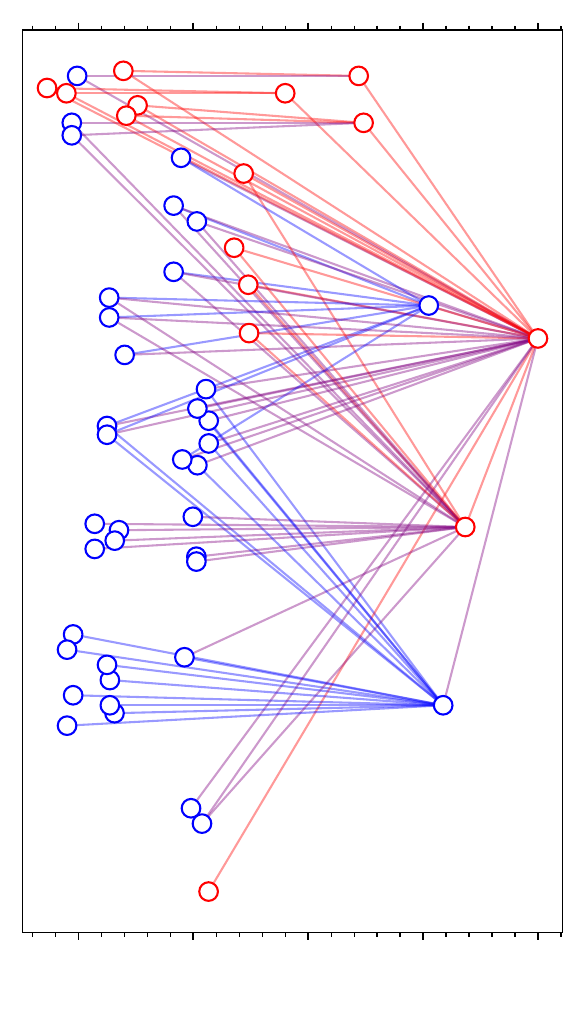}
      \draw (91.958042, 116.504821) node[] {$0$};
      \draw (79.530679, 84.270618) node[] {$1$};
      \draw (75.756314, 53.798334) node[] {$2$};
      \draw (73.294384, 122.131893) node[] {$3$};
      \draw (62.153311, 153.365215) node[] {$4$};
      \draw (61.316061, 161.383024) node[] {$5$};
      \draw (48.752495, 158.415743) node[] {$6$};
      \draw (42.574596, 117.419605) node[] {$7$};
      \draw (42.438072, 125.683405) node[] {$8$};
      \draw (41.655262, 144.693990) node[] {$9$};
      \draw (40.025523, 132.010017) node[] {$10$};
      \draw (35.677697, 102.444829) node[] {$11$};
      \draw (35.677697, 98.555078) node[] {$12$};
      \draw (35.655836, 21.957718) node[] {$13$};
      \draw (35.216106, 107.841283) node[] {$14$};
      \draw (34.529216, 33.573163) node[] {$15$};
      \draw (33.743421, 104.520389) node[] {$16$};
      \draw (33.743421, 94.834446) node[] {$17$};
      \draw (33.650759, 136.507062) node[] {$18$};
      \draw (33.569742, 79.197029) node[] {$19$};
      \draw (33.569742, 78.366805) node[] {$20$};
      \draw (32.963614, 86.023312) node[] {$21$};
      \draw (32.653209, 36.194517) node[] {$22$};
      \draw (31.529896, 62.008324) node[] {$23$};
      \draw (31.158021, 95.818414) node[] {$24$};
      \draw (30.948645, 147.384530) node[] {$25$};
      \draw (29.679950, 139.212976) node[] {$26$};
      \draw (29.679950, 127.897335) node[] {$27$};
      \draw (23.515800, 156.355558) node[] {$28$};
      \draw (21.592029, 154.587489) node[] {$29$};
      \draw (21.302698, 113.691285) node[] {$30$};
      \draw (21.083946, 162.251684) node[] {$31$};
      \draw (20.338066, 83.724823) node[] {$32$};
      \draw (19.609832, 81.941379) node[] {$33$};
      \draw (19.573239, 52.445377) node[] {$34$};
      \draw (18.782633, 53.798334) node[] {$35$};
      \draw (18.782633, 58.110885) node[] {$36$};
      \draw (18.667203, 123.484850) node[] {$37$};
      \draw (18.667203, 120.094770) node[] {$38$};
      \draw (18.295343, 101.537733) node[] {$39$};
      \draw (18.295343, 100.046405) node[] {$40$};
      \draw (18.295343, 60.709177) node[] {$41$};
      \draw (16.190330, 84.824100) node[] {$42$};
      \draw (16.190330, 80.526924) node[] {$43$};
      \draw (13.187747, 161.383024) node[] {$44$};
      \draw (12.501746, 65.905763) node[] {$45$};
      \draw (12.501746, 55.512592) node[] {$46$};
      \draw (12.272624, 153.365215) node[] {$47$};
      \draw (12.272624, 151.228158) node[] {$48$};
      \draw (11.462375, 63.307470) node[] {$49$};
      \draw (11.462375, 50.316007) node[] {$50$};
      \draw (11.348601, 158.415743) node[] {$51$};
      \draw (8.041958, 159.315152) node[] {$52$};
      \draw (50.000000, 7) node[below] {\small Hyperbolic volume of $S^3 \setminus K$};
      \begin{scope}[font=\footnotesize]
        \draw (91.958042, 12.450142) node[below] {$0$};
        \draw (72.319123, 12.450142) node[below] {$5$};
        \draw (52.680204, 12.450142) node[below] {$10$};
        \draw (33.041286, 12.450142) node[below] {$15$};
        \draw (13.402367, 12.450142) node[below] {$20$};
      \end{scope}
      \begin{scope}[shift={(91.95804196, 15.23290638)},
        xscale=-3.92778376, yscale=0.00768726]
      \end{scope}
    \end{tikzoverlay*}
  \end{scope}

  \begin{scope}
    \node[above left] at (14, 0.5) {\scriptsize
      \begin{tabular}{ll}
        0 & unknot \\
        1 & $K6a3$ \\
        2 & $K8n1$ \\
        3 & $K9n5$ \\
        4 & $K8a16$ \\
        5 & $K8a4$ \\
        6 & $K9a12$ \\
        7 & $K11a165$ \\
        8 & $K11a201$ \\
        9 & $K11a58$ \\
        10 & $K11a103$ \\
        11 & $K13n1196$ \\
        12 & $K13n1191$ \\
        13 & $K10a103$ \\
        14 & $K13n1156$ \\
        15 & $K12n440$ \\
        16 & $K13n1194$ \\
        17 & $K13n1198$ \\
        18 & $K12n49$ \\
        19 & $K12n553$ \\
        20 & $K12n556$ \\
        21 & $K14n14150$ \\
        22 & $K13n2199$ \\
        23 & $K13n2927$ \\
        24 & $K14n17953$ \\
        25 & $K14n11817$ \\
        26 & $K14n14151$ \\
        27 & $K14n14149$ \\
        28 & $K13a2044$ \\
        29 & $K13a3071$ \\
        30 & $K14n7615$ \\
        31 & $K13a521$ \\
        32 & $K14n18949$ \\
        33 & $K14n16092$ \\
        34 & $K14n5634$ \\
        35 & $K14n16105$ \\
        36 & $K14n17618$ \\
        37 & $K14n7628$ \\
        38 & $K14n7617$ \\
        39 & $K14n6300$ \\
        40 & $K14n7632$ \\
        41 & $K14n6223$ \\
        42 & $K14n216$ \\
        43 & $K14n229$ \\
        44 & $K14n4939$ \\
        45 & $K14n4482$ \\
        46 & $K14n4479$ \\
        47 & $K14n14502$ \\
        48 & $K14n14508$ \\
        49 & $K14n7685$ \\
        50 & $K14n7602$ \\
        51 & $K14a7319$ \\
        52 & $K14a4364$ \\
      \end{tabular}};
  \end{scope}
\end{tikzpicture}

%% file: plots/comp_size.tex
\begin{tikzoverlay*}[width=\figurewidth]{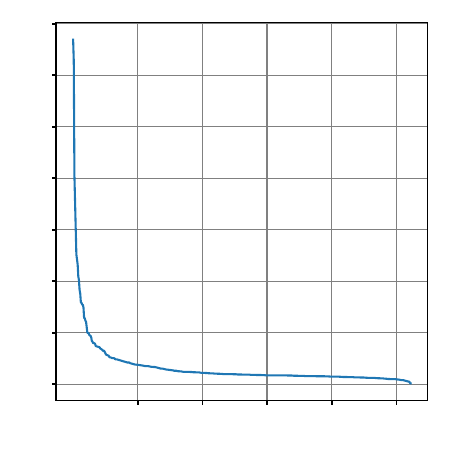}
  \draw (30.617816, 8.453704) node[below] {$100$};

  \draw (44.985632, 8.453704) node[below] {$200$};

  \draw (59.353448, 8.453704) node[below] {$300$};

  \draw (73.721264, 8.453704) node[below] {$400$};

  \draw (88.089080, 8.453704) node[below] {$500$};

  \draw (10.53704, 14.635165) node[left] {$0$};

  \draw (10.53704, 26.073697) node[left] {$250$};

  \draw (10.53704, 37.512228) node[left] {$500$};

  \draw (10.53704, 48.950760) node[left] {$750$};

  \draw (10.53704, 60.389291) node[left] {$1000$};

  \draw (10.53704, 71.827823) node[left] {$1250$};

  \draw (10.53704, 83.266354) node[left] {$1500$};

  \draw (10.53704, 94.704886) node[left] {$1750$};

  \begin{scope}[shift={(16.25000000, 14.63516531)},
    xscale=0.14367816, yscale=0.04575413]
  \end{scope}
  
  \node[rotate=90, anchor=south, axis label] at (-3, 55) {Size};
  \node[below, axis label] at (50, 0) {Component index};

\end{tikzoverlay*}

%% file: plots/comp_size_log_log.tex
\begin{tikzoverlay*}[width=\figurewidth]{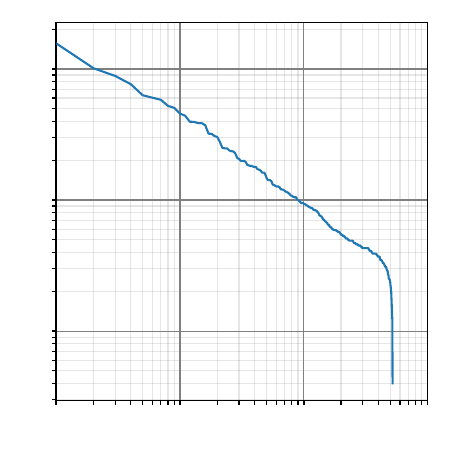}
  \draw (12.500000, 8.453704) node[below] {${10^{0}}$};

  \draw (40.000000, 8.453704) node[below] {${10^{1}}$};

  \draw (67.500000, 8.453704) node[below] {${10^{2}}$};

  \draw (95.000000, 8.453704) node[below] {${10^{3}}$};

  \draw (10.53704, 26.410358) node[left] {${10^{1}}$};

  \draw (10.53704, 55.540818) node[left] {${10^{2}}$};

  \draw (10.53704, 84.671279) node[left] {${10^{3}}$};

  \node[rotate=90, anchor=south, axis label] at (0, 55) {Size};
  \node[below, axis label] at (50, 0) {Component index};
\end{tikzoverlay*}